\documentclass[A4paper,12pt,twoside ]{article}
\usepackage{amsmath,amsthm,amssymb,mathabx}

\usepackage[left=25mm, right=25mm,top=25mm,bottom=25mm]{geometry}
\usepackage{imakeidx}
\usepackage{xr-hyper} 
\usepackage{xr}
\usepackage[all]{xy}
\usepackage{hyperref} 

\usepackage{cleveref,fancyhdr}

\usepackage{seqsplit}
\usepackage{xstring}
\usepackage[xcdraw]{xcolor}
\usepackage{todonotes}
\usepackage[nottoc]{tocbibind}

\definecolor{mycolor}{rgb}{0.122, 0.435, 0.698}

\hypersetup{
	colorlinks,
	citecolor=blue,
	filecolor=blue,
	linkcolor=blue,
	urlcolor=blue
}
\makeatletter
\DeclareOldFontCommand{\rm}{\normalfont\rmfamily}{\mathrm}
\DeclareOldFontCommand{\sf}{\normalfont\sffamily}{\mathsf}
\DeclareOldFontCommand{\tt}{\normalfont\ttfamily}{\mathtt}
\DeclareOldFontCommand{\bf}{\normalfont\bfseries}{\mathbf}
\DeclareOldFontCommand{\it}{\normalfont\itshape}{\mathit}
\DeclareOldFontCommand{\sl}{\normalfont\slshape}{\@nomath\sl}
\DeclareOldFontCommand{\sc}{\normalfont\scshape}{\@nomath\sc}
\makeatother

\usepackage[shortlabels]{enumitem}
\setlist{nolistsep} %

\newlist{asslist}{enumerate}{1} 
\setlist[asslist]{label=(\roman*), ref=\thethmT(\roman*)}
\crefalias{asslistenumi}{Assumption} 

\newlist{asslisttief}{enumerate}{1} 
\setlist[asslisttief]{label=(\roman*), ref=\thethmlistT(\roman*)}
\crefalias{asslisttiefenumi}{Property}

\newlist{thmlist}{enumerate}{1} 
\setlist[thmlist]{label=(\alph*), ref=\thethmT(\alph*)}

\usepackage{tikz} 

\usepackage{xcolor} 
\definecolor{ocre_old}{RGB}{243,102,25} 
\definecolor{ocre}{RGB}{64,64,64}
\definecolor{bluee}{rgb}{0.122, 0.435, 0.698} 

\usepackage{amsmath,amsfonts,amssymb,amsthm} 

\newcommand{\omicron}{\operatorname{sec}}

\newcommand{\AsF}{{\AHs_F}}
\newcommand{\AshochF}{{\AHs^F}}
\newcommand{\AHs}{{ A_\bH(s)}}

\makeatletter
\newtheoremstyle{ocrenumbox}
{0pt}
{0pt}
{\sl}
{}
{\small\bf\sffamily\color{ocre}}
{\;}
{0.25em}
{\small\sffamily\color{ocre}\thmname{#1}\nobreakspace\thmnumber{\@ifnotempty{#1}{}\@upn{#2}}
	\thmnote{\nobreakspace\the\thm@notefont\sffamily\bfseries\color{black}---\nobreakspace#3.}} 

\newtheoremstyle{ocrenumhypbox}
{0pt}
{0pt}
{}
{}
{\small\bf\sffamily\color{ocre}}
{\;}
{0.25em}
{\small\sffamily\color{ocre}\thmname{#1}\nobreakspace\thmnumber{\@ifnotempty{#1}{}\@upn{#2}}
	\thmnote{\nobreakspace\the\thm@notefont\sffamily\bfseries\color{black}---\nobreakspace#3.}} 

\newtheoremstyle{blacknumex}
{5pt}
{5pt}
{\sl}
{} 
{\small\bf\sffamily}
{\;}
{0.25em}
{\small\sffamily{\tiny\ensuremath{\blacksquare}}\nobreakspace\thmname{#1}\nobreakspace\thmnumber{\@ifnotempty{#1}{}\@upn{#2}}
	\thmnote{\nobreakspace\the\thm@notefont\sffamily\bfseries---\nobreakspace#3.}}

\newtheoremstyle{blacknumbox} 
{0pt}
{0pt}
{\normalfont}
{}
{\small\bf\sffamily}
{\;}
{0.25em}
{\small\sffamily\thmname{#1}\nobreakspace\thmnumber{\@ifnotempty{#1}{}\@upn{#2}}
	\thmnote{\nobreakspace\the\thm@notefont\sffamily\bfseries---\nobreakspace#3.}}

\newtheoremstyle{ocrenum}
{5pt}
{5pt}
{\sl}
{}
{\small\bf\sffamily\color{ocre}}
{\;}
{0.25em}
{\small\sffamily\color{ocre}\thmname{#1}\nobreakspace\thmnumber{\@ifnotempty{#1}{}\@upn{#2}}
	\thmnote{\nobreakspace\the\thm@notefont\sffamily\bfseries\color{black}---\nobreakspace#3.}} 
\makeatother

\theoremstyle{ocrenumbox}
\newtheorem{thmT}{Theorem}[section]
\newtheorem{theoT}{Theorem}
\newtheorem{theoremeT}[thmT]{Theorem}
\newtheorem{lemT}[thmT]{Lemma}

\theoremstyle{ocrenumhypbox}
\newtheorem{hypT}[thmT]{Hypothesis}
\theoremstyle{blacknumex}

\theoremstyle{blacknumbox}
\newtheorem{definitionT}[thmT]{Definition}
\newtheorem{notationT}[thmT]{Notation}

\theoremstyle{ocrenum}

\newtheorem{propT}[thmT]{Proposition}

\newtheorem{corollaryT}[thmT]{Corollary}

\RequirePackage[framemethod=default]{mdframed} 

\newmdenv[skipabove=7pt,
skipbelow=7pt,
backgroundcolor=black!5,
linecolor=ocre,
innerleftmargin=5pt,
innerrightmargin=5pt,
innertopmargin=5pt,
leftmargin=0cm,
rightmargin=0cm,
innerbottommargin=5pt]{tBox}

\newmdenv[skipabove=7pt,
skipbelow=7pt,
rightline=false,
leftline=true,
topline=false,
bottomline=false,
backgroundcolor=ocre!10,
linecolor=ocre,
innerleftmargin=5pt,
innerrightmargin=5pt,
innertopmargin=5pt,
innerbottommargin=5pt,
leftmargin=0cm,
rightmargin=0cm,
linewidth=4pt]{eBox}	

\newmdenv[skipabove=7pt,
skipbelow=7pt,
rightline=false,
leftline=true,
topline=false,
bottomline=false,
linecolor=ocre,
innerleftmargin=5pt,
innerrightmargin=5pt,
innertopmargin=0pt,
leftmargin=0cm,
rightmargin=0cm,
linewidth=4pt,
innerbottommargin=0pt]{dBox}	

\newmdenv[skipabove=7pt,
skipbelow=7pt,
rightline=false,
leftline=true,
topline=false,
bottomline=false,
linecolor=gray,
backgroundcolor=black!5,
innerleftmargin=5pt,
innerrightmargin=5pt,
innertopmargin=5pt,
leftmargin=0cm,
rightmargin=0cm,
linewidth=4pt,
innerbottommargin=5pt]{cBox}

\newenvironment{theorem}{\begin{tBox}\begin{theoremeT}}{\end{theoremeT}\end{tBox}}
\newenvironment{thmbox}{\begin{tBox}\begin{theoremeT}}{\end{theoremeT}\end{tBox}}
\newenvironment{hyp}{\begin{tBox}\begin{hypT}}{\end{hypT}\end{tBox}}
\newenvironment{theo}{\begin{tBox}\begin{theoT}}{\end{theoT}\end{tBox}}

\newenvironment{defi}{\begin{dBox}\begin{definitionT}}{\end{definitionT}\end{dBox}}	
\newenvironment{notation}{\begin{dBox}\begin{notationT}}{\end{notationT}\end{dBox}}	
\newenvironment{lem}{\begin{dBox}\begin{lemT}}{\end{lemT}\end{dBox}}	
\newenvironment{prop}{\begin{dBox}\begin{propT}}{\end{propT}\end{dBox}}

\newenvironment{cor}{\begin{dBox}\begin{corollaryT}}{\end{corollaryT}\end{dBox}}	

\makeatletter
\renewcommand{\@seccntformat}[1]{\llap{\textcolor{ocre}{\csname the#1\endcsname}\hspace{1em}}} 
\renewcommand{\section}{\@startsection{section}{1}{\z@}
	{-4ex \@plus -1ex \@minus -.4ex}
	{1ex \@plus.2ex }
	{\normalfont\large \bf \color{ocre}}}
\renewcommand{\subsection}{\@startsection {subsection}{2}{\z@}
	{-3ex \@plus -0.1ex \@minus -.4ex}
	{0.5ex \@plus.2ex }
	{\normalfont\large\bf\color{ocre} }}

\newcommand{\EGF}{{E(\GF)}}
\def\titlerunning#1{\gdef\titrun{#1}}
\makeatletter
\def\author#1{\gdef\autrun{\def\and{\unskip, }#1}\gdef\@author{#1}}
\def\address#1{{\def\and{\\\hspace*{18pt}}\renewcommand{\thefootnote}{}%
		\footnote {#1}}%
	\markboth{\titrun}{\titrun}}
\makeatother
\def\email#1{e-mail: #1}
\def\subjclass#1{{\renewcommand{\thefootnote}{}%
		\footnote{\emph{Mathematics Subject Classification (2010):} #1}}}

\newcommand{\tw}[1]{{}^#1\!}

\newcommand{\otw}{\text{otherwise}}
\newcommand{\inv}{^{-1}}

\theoremstyle{definition}

\newtheorem{rem}[thmT]{Remark}

{\color{ocre}}
\theoremstyle{plain}

\theoremstyle{definition}
\newtheorem{condi}[thmT]{Condition}

\numberwithin{equation}{section}
\numberwithin{table}{section}

\newcommand{\id}{\operatorname {id}}
\newcommand{\ad}{{\operatorname{ad}}}

\newcommand{\wt}{\widetilde}
\newcommand{\wc}{\widecheck}
\newcommand{\wh}{\widehat }

\newcommand{\wbT}{{\widetilde {\mathbf T}}}

\newcommand{\Ai}{{{A}(\infty)}}
\newcommand{\Ap}{{{A}'(\infty)}}

\newcommand{\wG}{{\widetilde G}}

\newcommand{\uE}{{\underline E}}
\newcommand{\bC}{{\mathbf C}}
\newcommand{\la}{\ensuremath{\lambda}}

\newcommand{\tDlsc}{\tD_{l,\mathrm{sc}}}

\newcommand{\tDlprimesc}{\tD_{l',\mathrm{sc}}}
\newcommand{\tDlad}{\tD_{l,\mathrm{\mathrm{ad}}}}
\newcommand{\twDlsc}{\tw 2\tD_{l,\mathrm{sc}}}

\newcommand{\bT}{{\mathbf T}}

\newcommand{\bG}{{{\mathbf G}}}

\newcommand{\bH}{{{\mathbf H}}}
\newcommand{\wbH}{{{\wt \bH}}}

\newcommand{\bK}{{{\mathbf K}}}

\newcommand{\bHssF}{\bH_{\textrm{ss}}^{F}}
 
	\newcommand{\Jor}{\operatorname{Jor}}
\newcommand{\oJor}{\ov\Jor}

\newcommand{\bM}{{\mathbf M}}
\newcommand{\bS}{{\mathbf S}}
\newcommand{\HF}{{{\bH}^F}}
\newcommand{\HFss}{{{\bH}_{\textrm{ss}}^F}}

\newcommand{\wHF}{{\wt\bH}^F}

\newcommand{\wbG}{\wt{\mathbf G}}

\newcommand{\wbK}{\wt{\mathbf K}}
\newcommand{\Irr}{{\mathrm{Irr}}}
\newcommand{\Lin}{\mathrm{Lin}}
\newcommand{\desc}{\mathrm{desc}}
\newcommand{\Tr}{\mathrm{Tr}}
\newcommand{\CF}{\mathrm{CF}}

\newcommand{\Char}{\mathrm{Char}}

\newcommand{\bZ}{{\mathbf Z}}

\newcommand{\SL}{\operatorname{SL}}

\newcommand{\GL}{\operatorname{GL}}
\newcommand{\SO}{\operatorname{SO}}

\newcommand{\PGL}{\operatorname{PGL}}

\newcommand{\ZZ}{\ensuremath{\mathbb{Z}}}
\newcommand{\CC}{\ensuremath{\mathbb{C}}}
\newcommand{\UU}{{\mathbb{U}}}
\newcommand{\rmUU}{{\mathrm{U}}}

\newcommand{\EE}{{\mathbb{E}}}

\newcommand{\EEnull}{{\EE}}

\newcommand{\ovEE}{{\ov\EE}}

\newcommand{\TT}{\ensuremath{\mathbb{T}}}

\newcommand{\ov}{\overline }

\newcommand{\xx}{\mathbf x }

\newcommand{\h}{\mathbf h }

\newcommand{\Cent}{\ensuremath{{\rm{C}}}}
\newcommand{\NNN}{\ensuremath{{\mathrm{N}}}}

\newcommand{\Sym}{{\mathcal{S}}}

\def\restr#1|#2{\left.#1\right\rceil_{#2}}

\usepackage{imakeidx}
\makeindex[columns=2,intoc]

\def\III#1{\index{#1@$#1$}{\color{ocre}#1}}
\def\II#1@#2{\index{#1@$#2$}{{\color{ocre}#2}}}

\newcommand{\QQ}{\ensuremath{\mathbb{Q}}}
\newcommand{\tD}{\ensuremath{\mathrm{D}}}
\newcommand{\Cy}{\mathrm C}
\newcommand{\tC}{\mathrm C}
\newcommand{\tB}{\mathrm B}
\newcommand{\JJ}{{\mathrm J}}
\newcommand{\cE}{\mathcal E}
\newcommand{\ocE}{\overline{\mathcal E}}
\newcommand{\cM}{\mathcal M}
\newcommand{\cK}{\mathcal K}

\newcommand{\calM}{\mathcal M}
\newcommand{\calL}{\mathcal L}

\newcommand{\calC}{\mathcal C}
\newcommand{\calZ}{\mathcal Z}

\newcommand{\Cl}{\mathfrak {Cl}}
\newcommand{\al}{{\alpha}}
\newcommand{\eps}{{\epsilon}}

\newcommand{\UCh}{\operatorname{Uch}}
\newcommand{\Uch}{\operatorname{Uch}}
\newcommand{\oUCh}{\ov{\UCh}}
\newcommand{\spannh}{\spann<h_0>}
\newcommand{\spannhHnull}{{\spa{ h_0^{(\bH_0)}}}}

\newcommand{\FF}{{\mathbb{F}}}
\newcommand{\FFtimes}{{\mathbb{F}^\times}}
\newcommand{\si}{\ensuremath{\sigma}}

\newcommand{\GF}{{{\bG^F}}}

\newcommand{\spannsi}{{\spann<\si>}}

\newcommand{\epsFa}{v_a}

\newcommand{\bGos}{{\Cent^\circ_\bH(s)}}
\newcommand{\bGs}{{\Cent_\bH(s)}}
\newcommand{\CoHFs}{{\Cent^\circ_\bH(s)^F}}
\newcommand{\CHFs}{{\Cent_\HF(s)}}

\newcommand{\wGF}{{{{\wbG}^F}}}
\newcommand{\wKF}{{{{\wbK}^F}}}

\newcommand{\wGFnull}{{{{\wbG}^{F_0}}}}
\newcommand{\GFnull}{{{{\bG}^{F_0}}}}
\newcommand{\HFnull}{{{{\bH}^{F_0}}}}

\makeatletter
\def\Set#1{\Set@h#1@}
\def\Lset#1{\Lset@h#1@}
\def\Set@h#1|#2@{\left\{\left.#1\vphantom{#2}\hskip.1em\,\right\mid \,\relax #2\right\}}
\def\Lset@h#1@{\left\{#1\right\}}
\def\CALC#1{\CALC@h#1@}
\def\CALC@h#1|#2@{\calC^{#1}(#2)}
\def\CALCrad#1{\CALCrad@h#1@}
\def\CALCrad@h#1|#2@{\calC_\radic^{#1}(#2)}

\def\CALCNC#1{\CALCNC@h#1@}
\def\CALCNC@h#1|#2@{\calC_{\radic,nc}^{#1}(#2)}

\def\restr#1|#2{\left.#1\right\rceil_{#2}}
\def\spann<#1>{\left\langle#1\right\rangle}
\def\spa#1{\left\langle#1\right\rangle}

\def\Spann<#1>{\Spann@h#1@}
\def\Spann@h#1|#2@{\left\langle\left.#1\vphantom{#2}\hskip.1em\right.\mid\relax #2 \right\rangle}
\def\Set#1{\Set@h#1@}
\def\Set@h#1|#2@{\left\{\left.#1\vphantom{#2}\hskip.1em\,\right.
	\mid\relax #2\right\}}
\def\set#1{\set@h#1@}
\def\set@h#1@{\left\{#1\right\}}
\def\spann<#1>{\left\langle#1\right\rangle}
\makeatother

\newcommand{\Aut}{\mathrm{Aut}}

\newcommand{\End}{\mathrm{End}}
\newcommand{\Hom}{\mathrm{Hom}}

\newcommand{\Out}{\ensuremath{\mathrm{Out}}}
\newcommand{\Z}{\operatorname Z}
\newcommand{\B}{ B}
\newcommand{\Sp}{\operatorname{Sp}}

\newcommand{\frakC}{{\mathfrak C}}

\newcommand{\forevery}{{\text{\quad\quad for every }}}
\newcommand{\und}{{\text{ and }}}

\newcommand{\btau}{\ensuremath{\overline{\tau}}}

\newcommand{\ra}{\rightarrow}
\newcommand{\lra}{\longrightarrow}
\newcommand{\tE}{\mathrm E}
\newcommand{\tA}{\mathrm A}

\newcommand{\Sh}{\operatorname{Sh}}
\newcommand{\cC}{\mathcal C}
\newcommand{\cF}{\mathcal F}
\newcommand{\cI}{\mathcal I}
\newcommand{\cJ}{\mathcal J}
\newcommand{\cZ}{\mathcal Z}

\newcommand{\cusp}{\Irr_{cusp}}
\newcommand{\wrt}{{with respect to\ }}

\titlerunning{Inductive McKay Condition in type D, II}
\title{Extensions of characters in type D and the inductive McKay condition, II\\
	{\small \textit{Dedicated to Gunter Malle for his fundamental contributions} }}
\author{Britta Sp\"ath \thanks{
		Part of the research for this paper was conducted in the framework of the research training group \emph{GRK 2240: Algebro-Geometric Methods in Algebra, Arithmetic and Topology}, funded by the DFG. The author would like to thank the Isaac Newton Institute for Mathematical Sciences for support during the program \textit{Groups, Representations and Applications}, when work on this paper was undertaken. This program was supported by EPSRC Grant Number EP/R014604/1.}
}
\date{}
\RequirePackage[normalem]{ulem} 
\RequirePackage{color}\definecolor{RED}{rgb}{1,0,0}\definecolor{BLUE}{rgb}{0,0,1} 

\begin{document} 
\maketitle
\abstract{We determine the action of the automorphism group Aut$(G)$ on the set of irreducible characters Irr$(G)$ for all finite quasi-simple groups $G$. For groups of Lie type, this includes the construction of an Aut$(G)$-equivariant Jordan decomposition of characters (Theorem B). We prove a property called $A(\infty)$ which includes an extendibility statement, known previously in types not $\mathrm{D}$ (Theorem A). Our methods blend here Shintani descent ideas introduced for type $\mathrm{B}$ along with an analysis of semisimple classes in the dual group $G^*$. The condition $A(\infty)$ originates in the program to prove the McKay conjecture using the classification of finite simple groups. Theorem C establishes the McKay conjecture for the prime 3.}
\address{School of Mathematics and Natural Sciences, 
University of Wuppertal, Gau\ss str. 20, 42119 Wuppertal, Germany, \email{bspaeth@uni-wuppertal.de}
}
{\small{\tableofcontents}}
	\section{Introduction} Several conjectures on representations of finite groups have been reduced to problems about finite simple groups. This has led to spectacular progress by use of the classification of finite simple groups (CFSG) and our knowledge of their representations, see for instance the survey \cite{MaLoGlo}. This is the case for several of the questions and conjectures brought forward by Brauer in his famous list of problems \cite{Brauer}. However, the comparatively more recent counting conjectures stemming from McKay's conjecture (McKay's conjecture, Alperin's weight conjecture, Dade's conjectures and their refinements) demand a stronger control on $\Irr(G)$ for quasi-simple groups $G$. In particular it seems indispensable to know the action of outer automorphisms on $\Irr(G)$ for each finite quasi-simple group $G$, see \cite[Problem 2.7]{MaLoGlo}. The main aim of this paper is to complete the answer to that question. 
	
	Even though these outer automorphism groups $\Out(G)$ are small, they can pose deep problems both to determine their action on $\Irr(G)$ and also to assess extendibility of characters in the case of inclusions $G\unlhd A$. Measuring the latter problem by the 2-rank of $\Out(G)$, it is expected that quasi-simple groups $G$ of type $\tA$, $\tD$ or $^2\tD$ are the most problematic. Type $\tA$ was solved thanks to the existence and remarkable properties of Kawanaka's so-called generalized Gelfand-Graev representations, see \cite{CS17A}. Since other types with nicer $\Out(G)$ groups were solved in \cite{CS13,CS17A,CS18B}, the present paper therefore deals a lot, though not only, with groups of type $\tD$ and $^2\tD$. 
	
	In what follows, $G$ is the universal covering of a simple group of Lie type $G/\Z(G)$ assumed to be realized by $G=\bG^F$ the group of rational points of a simple simply connected algebraic group $\bG=\bG_{\mathrm{sc}}$ defined over a finite field. We thus leave aside the 17 simple groups of Lie type whose universal coverings, though not of the form $\GF$ (see \cite[Table 6.1.3]{GLS3}), have a well-known character table complete with the action of automorphisms, see \cite{atlas}.
	
	One distinguishes three types of outer automorphisms of $G$, namely \emph{diagonal}, \emph{graph} and \emph{field} automorphisms. Diagonal automorphisms can be realized via a regular embedding of algebraic groups $\bG\leq \wbG$ (see \ref{not} below), resulting in $G=[\wG ,\wG]\unlhd \wG$ for $\wG=\wbG^F$, while graph and field automorphisms form a group $E(G)$ also acting on $\wG$. The semi-direct product $\wG\rtimes E(G)$ then induces all automorphisms of $G$, see \cite[Sect. 1.15]{GLS3}. 
	
	We show that the actions of $\wG$ and $E(G)$ on $\Irr(G)$ are of {\textit {transversal}}\ nature. Namely, there exists some $E(G)$-stable $\wG$-transversal in $\Irr(G)$. This transversality condition is called $\Ap$, see \ref{Ainfty} below. It is also equivalent to saying that for any $D\leq E(G)$ a $D$-stable $\wG$-orbit in $\Irr(G)$ always contains some $D$-invariant character. 
	This complements well known consequences of Lusztig's theory about $\Irr(\wG)$, explaining how $E(G)$ acts on the $\wG$-orbits in $\Irr(G)$. Then $\Ap$ fully determines  $\Irr(G)$ as a $\wG E(G)$-set, see \cite[Rem. 2.5]{CS18B} or Lemma~\ref{ZEsets} below.
	
	Our main result also addresses the extendibility question mentioned above:

	\begin{theo}\label{thm1}
		Let $G=\bG_{\mathrm{sc}}^F$ be the universal covering of a finite simple group of Lie type with $\wG$ and $E(G)$ as above, see also \ref{not}. Then there exists an $E(G)$-stable $\wG$-transversal $\TT$ in $\Irr(G)$ such that every character $\chi\in\TT$ extends to its stabilizer in $G E(G)$.
	\end{theo}

The above property of $\Irr(G)$ appeared first as Assumption 2.12(v) in \cite{S12} and is called Condition $\Ai$ in \cite[Def.~2.2]{CS18B}. For groups of types different from $\tD$ and $^2\tD$, for which Theorem A has been proved, it has already lead to many results by several authors towards the verification of inductive conditions for counting conjectures. In the generality given here, it is  used in \cite{FS} and in \cite{R_iAM3} through \cite{R_Adv}, hence contributing to the proof of Brauer's Height Zero Conjecture and the Alperin--McKay Conjecture for the prime $2$. 

Of course the determination of the action of $\wG E(G)$ on $\Irr(G)$ and hence $\Ap$ are strongly related with the possibility of parametrizing $\Irr(G)$ in an $\Out(G)$-equivariant way, typically by constructing some Jordan decomposition of characters compatible with $\Out(G)$-action, see for instance A.9 of \cite{GM}. For the groups of the form $\wG =\wbG^F$ where $\wbG$ has connected center, this was deduced in \cite{CS13} from the uniqueness of the Jordan decomposition proved by Digne-Michel, see \cite{DM90}. But devising a satisfactory uniqueness condition for quasi-simple groups $G$ of Lie type seems a quite difficult goal. 

 Here we show that such a Jordan decomposition exists, essentially as a consequence of Theorem A, though in our inductive proof the equivariant Jordan decomposition is used in lower ranks to prove $\Ai$, see \ref{sec8B}.

Let $G^*$ be dual to $G$ as finite group of Lie type, see \ref{duals} below. As a result of duality, for any semisimple element $s$ of $G^*$ there is a group morphism $\wh\omega _s\colon \wG/G\Z(\wG)\to \Lin(\Cent_{G^*}(s))$ associating a linear character of $\Cent_{G^*}(s)$ to any (diagonal) automorphism of $G$ induced by an element of $\wG$, see \Cref{omegahat}. Given any graph-field automorphism $\si\in E(G)$ one can  define a dual automorphism $\si^*\colon G^*\to G^*$, see \ref{duals}. With those notions we prove the following equivariant Jordan decomposition of characters where $G$ is as in Theorem A (note that types $\tA$ and $\tC$ were already known, see \cite[Sect. 8]{CS17A} and \cite[Sect. 5.B]{Li}).

\begin{theo}  [$\Out(G)$-equivariant Jordan decomposition] \label{theoB} Let $\Jor(G)$ be the set of $G^*$-conjugacy classes of pairs $(s,\phi)$ where $s$ is a semisimple element of $G^*$ and $\phi\in\UCh(\Cent_{G^*}(s))$ is a unipotent character of $\Cent_{G^*}(s)$. One defines an action of $\Out(G)=\wG/G\Z(\wG)\rtimes E(G)$ on $\Jor(G)$  by \begin{align*}
	z.(s,\phi)&=(s,\wh\omega _s(z)\phi)\ \ \  \text{\ for\ } z\in \wG/G\Z(\wG) \text{, and} \\ \si.(s,\phi)&=(\si^*{}^{-1}(s),\phi\circ \si^*) \ \ \text{ for }\si\in E(G).
	\end{align*}
	Then there is an $\Out(G)$-equivariant bijection $$\Jor(G)\xrightarrow{\sim}\Irr(G), \text{  \  }(s,\phi)\mapsto \chi_{s,\phi}$$ where $\chi_{s,\phi}$ is an element of the (rational) Lusztig  series $\cE(G,[s])$.
\end{theo}

The proof of Theorem B as a consequence of $\Ap$ is mapped out in \Cref{ZEsets}. A key step consists in showing that the action of $\wG E(G)$ on $\Jor(G)$ satisfies an analogue of $\Ap$, see \Cref{lem3_7}. Interestingly, this makes clear that $\Ap$ itself is in fact equivalent to the existence of an $\Out(G)$-equivariant Jordan decomposition of characters of $G$, a more natural condition - while on the other hand the proof of $\Ap$ for type $\tA$ in \cite{CS17A} points to a plausible property of unipotent supports. In the rest of the proof of Theorem B, each $\wG$-orbit in $\Irr(G)$ is identified with the (non irreducible) character given by its sum, a feature also used in the rest of the paper. 

Let us now comment on how we prove Theorem A which is the main task of the paper. Since all other types are solved, we take $\bG$ a simple simply connected algebraic group of type $\tD_l$ for $l\geq 4$ and $F\colon \bG\to \bG$ a Frobenius endomorphism defining $\bG$ over a finite field $\FF_q$. Excluding groups of type $^3\tD_4$ already studied in \cite[Sect. 5]{CS13}, we get $\bG^F=\tDlsc(q) = \text{Spin}_{2l}(q)$ or $\bG^F={}^2\tDlsc(q) = \text{Spin}^-_{2l}(q)$ according to $F$ being twisted or not. \Cref{hyp_cuspD_ext} is assumed, essentially stating that groups of simply connected type $\tD_{l'}$ ($4\leq l'<l$) satisfy $\Ai$. This serves as an induction hypothesis on the rank $l$ while \cite[Thm A]{TypeD1} implies that $\GF$ satisfies $\Ap$. Our main goal is then to ensure separately two main statements:
	 \begin{itemize} 
		\item when $\GF=\tw 2 \tDlsc(q)$ there exists an $E(\GF)$-stable $\wGF$-transversal in $\Irr(\GF)$; and
		\item when $\GF= \tDlsc(q)$, for each non-cyclic $D\leq E(\GF)$ every $D$-stable $\wGF$-orbit in $\Irr(\GF)$ contains a $D$-invariant character that extends to $\GF D$. 
	 \end{itemize}

	Via a careful application of Shintani descent, \Cref{prop82} and \Cref{thm77} show that both statements can be proved by counting certain sets of characters. 
	Shintani descent arguments rely on the fact that we can often focus on automorphisms of $\GF$ in $E(\GF)$ that are induced by Frobenius endomorphisms of $\bG$ and $\wbG$. 

	\Cref{prop82} considers the case where $\GF=\tw 2 \tDlsc(q)$ and shows roughly speaking that there exists an $E(\GF)$-stable $\wGF$-transversal in $\Irr(\GF)$ if (using the exponent notation for fixed points, see \ref{not11})
	$$| \Irr(\GF)^{\spann<\wGF, F_0>}|=|\Irr(\GFnull)^{\spann<\wGFnull,F >}| \forevery F_0\in \cF,$$
	where $\cF$ denotes a well chosen set of \textit{untwisted} Frobenius endomorphisms of $\wbG$.

	\Cref{thm77} deals with $\GF=\tDlsc(q)$ and focuses on the $D$-invariant characters in $\Irr(\GF)$ for some important non-cyclic $2$-subgroups $D\leq E(\GF)$: if $F_0$ and $\gamma$ denote certain generators of $D$ where $F_0$ is a (possibly \textit{twisted}) Frobenius endomorphism of $\bG$ with $[\Z(\bG),F_0]=1$, then there exist exactly $$\frac 1 2 \left( |\Irr(\GF)^D| + |\Irr(\GFnull)^{\spann<\gamma>}| \right) $$ 
	characters of $\GF$ that extend to $\GF D $. Afterwards labels for $\Irr(\GF)^D$ and $\Irr(\GFnull)^{\spann<\gamma>}$ are compared and through close examination of the associated unipotent characters we find exactly $\frac 1 2 \left( |\Irr(\GF)^D| - |\Irr(\GFnull)^{\spann<\gamma>}| \right) $ characters that are $D$-invariant, have no extension to $\GF D$ and lie in different $\wGF$-orbits of length $2$, see \Cref{cor822}.  Here we single out the corresponding $D$-stable $\wGF$-orbits of length $2$ and study them via associated characters of $\wGF$ and their Jordan decomposition of characters.  Altogether this allows us to find in each $D$-stable $\wGF$-orbit a character extending to $\GF D$.

	As the argument uses \cite[Thm~A]{TypeD1}, \Cref{hyp_cuspD_ext} on cuspidal characters has to be assumed. Via an inductive proof on the rank of $\bG$ this limitation can be removed in the final chapter and we obtain \Cref{thm1}. 

	Two important ideas are used throughout Chapters \ref{sec4_2D}, \ref{sec_ext_D} and \ref{SecEC}:
	 \begin{itemize} 
		\item For any semisimple conjugacy class $\cC$ of $\bG^*$ there exists some $s\in \cC$ and a complement $\wc A(s)$ of $\Cent_{\bG^*}^\circ(s)$ in $\Cent_{\bG^*}(s)$ that behaves well with respect to relevant automorphisms, see \cite[Thm~A]{CS22}.
		\item We develop and apply a variant $\ov\JJ$ of the Jordan decomposition for certain non-irreducible characters of $\GF$ that is equivariant with respect to automorphisms, see Chapter \ref{sec3}. 
	 \end{itemize} 
Those two ingredients enable us to consider simultaneously characters of different rational series in the same geometric series of $\GF$. We apply \Cref{cor_2Becht} and relate the characters to unipotent characters of different finite groups coming from the same algebraic group, see \Cref{newJdec}. Another key ingredient is the bijection from \Cref{propCSB} which relates to the \textit{generic} nature of unipotent characters in the sense of Brou\'e-Malle-Michel and was established in \cite{CS18B}. For a connected reductive group $\bC$ with two commuting Frobenius endomorphisms $F'$ and $F''$ there exists a bijection 
	\begin{equation*}
		f_{F',F''}:\UCh(\bC ^{F'})^{\spa{F''}}\lra \UCh(\bC ^{F''})^{\spa{F'}}.
	\end{equation*} which is equivariant for algebraic automorphisms of $\bC$ commuting with $F'$ and $F''$.
	This allows us to relate characters of $\GF$ to those of $\GFnull$ in \Cref{cor96} and \Cref{thm_desc}.
	
	Theorem \ref{thm1} also contributes to the general program to prove J. McKay's conjecture on character degrees for any prime $\ell$, see \cite{McK}. The reduction theorem from \cite{IMN}, the reformulation of the inductive McKay condition from \cite{S12} and its verification in various publications by Cabanes, Malle and the author have left open only the case of groups of types $\tD$ and $^2\tD$. Recall that for each prime $\ell$ it is necessary to verify the global condition $\Ai$ and some local conditions called $A(d)$ and $B(d)$ from \cite[Def.~2.2]{CS18B} for a quasi-simple group $G$ of Lie type $\tD$ or $^2\tD$ and $1\leq d\leq \ell -1$. The McKay conjecture is known for $\ell =2$ by \cite{MS16}. We obtain here the case $\ell =3$.

\begin{theo} \label{thmC} 
	Let $X$ be a finite group, let $P$ be a Sylow 3-subgroup of $X$. Then $X$ and $\NNN_{X}(P)$ have the same number of irreducible characters of degree prime to 3.
\end{theo}

	\textcolor{ocre}{\bf Structure of the paper.}  Our strategy has been largely described above. It is split as follows into the various chapters. In \Cref{sec2} we introduce most of the notation and recall some relevant results from \cite{TypeD1}. In \Cref{EquJor} we review known results on Jordan decomposition of characters and rational Lusztig series for simple groups of simply connected type.
	An important step is then \Cref{propunipext} on extendibility of unipotent characters. This is key to relate Condition $\Ap$ with the existence of an equivariant Jordan decomposition of characters. We then show there how Theorem A implies Theorem B. 
	
	From there the rest of the paper is about proving Theorem A for groups of type $\tD$ in odd characteristic, as \Cref{prop_ext_wG} implies already the statement of Theorem A in even characteristic. In \Cref{sec3} we deal with geometric and rational series of characters, $\cE(G,(s))$ and $\cE(G,[s])$, recalling and using the results of \cite{CS22} about centralizers of semisimple elements in the dual group $\bH$. This is also where we start to make use of orbit sums of characters, mostly under the action of $\wG$, thus leading to sets $\ov\cE(\GF,\cC)$ on which we then consider the action of $E(\GF)$. In preparation for some counting arguments we introduce certain disjoint unions ${\Bbb U}(s,F,\si)$ of sets of unipotent characters, see \Cref{lem717}.
	 This leads in \Cref{sec4_2D} to a proof of \Cref{thm1} for groups $\GF=\twDlsc(q)$ assuming  \Cref{hyp_cuspD_ext}.
	 
	
For a non-cyclic 2-subgroup $D$ of $E(\GF)$  we determine in \Cref{sec_ext_D} the cardinality of the set of $D$-invariant characters  that have no extension to $\GF D$. Then we introduce character sets ${\EE}(\cC)$ (in a geometric Lusztig series $\cE(\GF,\cC)$) with suitable cardinalities. In \Cref{SecEC} we study those characters of ${\EE}(\cC)$ and show that after some small modifications ${\EE}(\cC)$ is actually the set of all $D$-invariant but non-extendible elements of $\cE(\GF,\cC)$, see \Cref{prop821}. This ultimately ensures that enough $D$-invariant characters extend to $\GF D$, thus establishing the Condition $\Ai$. We conclude the paper with a proof of Theorems B and C. 

{\bf Acknowledgements:} 	Gunter Malle was my PhD advisor on the subject of McKay's conjecture. His guidance and support have immensely benefited my work, not speaking of his evident influence on the subject through his talks, books and papers. The present results build in a crucial way on \cite{Ma17} but my personal debt goes well beyond and I am very happy to express that here. 

The author would like to thank the anonymous referees for their questions and suggestions that greatly improved the text. Thanks also go to Marc Cabanes and Gunter Malle for a careful reading of several versions, and also to Julian Brough and Lucas Ruhstorfer for various discussions.

	\section{Some notation and previous results}\label{sec2}
	In this chapter we recall some notation around characters, finite groups of Lie type, their characters and the conditions $\Ap$ and $\Ai$.
 
	\subsection{Group actions and characters} \label{not11} \index{maximal extendibility}
	Group actions will be mostly on the left. When the group $G$ acts on the set $M$, we denote by $g.m$ or $^gm=m^{g\inv}$ the action of $g\in G$ on $m\in M$. The notation $\III{M^G}$ refers to the set of fixed points of $G$ in $M$. For $g\in G$ we also write $M^g:=M^{\spa{g}}=\{m\in M\mid g.m=m  \}$. For $m\in M$, we denote its $G$-orbit by $G.m$ and its stabilizer by $G_m:=\{g\in G\mid g.m=m  \}$. If $ M$ is a group on which $G$ acts by automorphisms, we write ${[M,g]}=\{m\inv .g\inv(m) \mid m\in M \}$. This is a subgroup when included in $\Z(M)$ and we then have $M/M^{\spa{g}}\cong [M,g]$. We define $\III{M_g}=M/[M,g]$ and ${[M,G]}=\spa{[M,g]\mid g\in G}$.
	
	The group $G$ acts on any normal subgroup $N$ by conjugation. We use the same letter for the element of $G$ and the automorphism it induces on $N$. For instance if $g\in G$ and $\si\in\Aut(N)$, then $g\si$ denotes the automorphism $n\mapsto g\si(n)g\inv$ of $N$.
	
About characters of finite groups we follow mostly the notation introduced in \cite{Isa}. For $X$ a finite group we denote by $\II LinX@{\Lin(X)}\subseteq\II IrrX@{\Irr(X)}$ the set of irreducible (complex) characters of $X$ and its subset of linear characters, seeing $\Lin(X)$ as a multiplicative group acting on $\Irr(X)$.
	
	 Let $X\leq Y$ be an inclusion of finite groups. If $\la\in\Irr(X)$, and $\psi\in\Irr(Y)$, we write $\II lambdaY@{\chi^Y}$ for the character induced \index{induced character} to $Y$ and $\II psiX@{\psi \rceil_X}$ for the {restricted character}\index{restricted character}.  We also use this notation when restricting a character to a subset of $Y$. 
	 For a generalized character $\chi\in\ZZ\Irr(X)$ we denote by $\II Irrchi@{\Irr(\chi)}$ the set of its irreducible constituents. We deal often with sets $\Irr(\chi^Y)\subseteq \Irr(Y)$ for $\chi\in\Irr(X)$. We then also define $\II IrrX@{\Irr(Y\mid \chi)}:= \Irr(\chi^Y)$.  If $\si\colon X\xrightarrow\sim  X'$ is an isomorphism of finite groups (possibly coming from an inclusion $X=X'\unlhd Y$ and conjugation by some $y\in Y$) and $\chi\in\ZZ\Irr(X')$ we write 
	$\II lambdasigma@{\chi^\sigma}= \,^{\sigma^{-1} }\chi$ for the character of $X$ defined by $\,^{\sigma^{-1}} \chi(x)= \chi^\si(x)=\chi(\si(x))$ for $x\in X$.

		Assume $X\unlhd Y$ both finite. If $\cM\subseteq \Irr(X)$, we say that \textit{maximal extendibility holds \wrt $X\unlhd Y$ for $\cM$ }if every $\chi\in \cM$ extends to an irreducible character of its stabilizer  $Y_\chi$. 
		Whenever $\cM=\Irr(X)$ we omit to mention $\cM$. If maximal extendibility holds \wrt $X\unlhd Y$ for $\cM$ and
		 $\calM$ is $Y$-stable, an \textit{extension map \wrt $X\unlhd Y$ for $\calM$} is a map $$\Lambda:\calM\lra \bigsqcup_{X\leq I \leq Y } \Irr(I)$$ such that for every $\chi\in\Irr(X)$, $\Lambda(\chi)\in\Irr(Y_\chi)$ satisfies $\restr\Lambda(\chi)|{X}=\chi$. It is easily seen that $\Lambda$ can be chosen to be $Y$-equivariant. Note also that whenever $Y/X$ is abelian, maximal extendibility is equivalent to the fact that for every $\chi\in\Irr(Y)$ the restriction $\restr\chi|{X}$ has no irreducible constituent with multiplicity $\geq 2$.

	\subsection{Quasi-simple groups of Lie type} \label{not}

We introduce now the groups and automorphisms considered afterwards. More specific notation will be given in Chapter~\ref{sec3} for groups of type $\tD$.

	Let $\II p @{p}$ be a prime and $\II FF@{\FF}$ the algebraic closure of $\FF_p$. Let $\II Gb @\bG$ be a simple simply connected linear algebraic group over ${\FF}$. We fix a maximal torus $\bT_0$ of $\bG$ and a basis $\Delta$ of the root system $\Phi(\bG,\bT_0)\subseteq X(\bT_0)$. For $\al\in\Phi(\bG,\bT_0)$ recall the Chevalley generators $\xx_\al(a)\in \bG$ ($a\in \FF$) satisfying $t \xx_\al(a)t\inv =\xx_\al(\al(t)a)$ for any $t\in \bT_0$, $a\in\FF$, see \cite[Sect. 1.3]{GLS3}, \cite[Thm 8.17]{MT}. Let $\II Fp@{F_p}\colon \bG\to\bG$ be defined by $F_p(\xx_\al(a))=\xx_\al(a^p)$ for any $\al\in \Phi{(\bG,\bT_0)}$. For $\tau\colon \Delta\to\Delta$ an automorphism of the Dynkin diagram, let ${{\btau}}\colon\bG\to\bG$ be the algebraic automorphism defined by $${{\btau}}(\xx_{\eps\beta}(a))=\xx_{\eps\tau(\beta)}(a)\text{ for any }\beta\in\Delta, \ \eps\in\{-1,1\}, \ a\in\FF .$$ Let $\III{E(\bG)}$ be the group of abstract automorphisms generated by $F_p$ and the ${{\btau}}$'s. Let $\III F:=F_p^{m}{{\btau}}$ for $\III{m}\geq 1$ and $\tau$ an automorphism (possibly trivial) of the Dynkin diagram. This is a Frobenius endomorphism  $ F\colon\bG\ra \bG$ defining an $\FF_q$-structure on $\bG$ for $q=p^{m}$. The group $\bG$ being classically denoted as $\tA_{l,\mathrm{sc}}(\FF)$, $\tB_{l,\mathrm{sc}}(\FF)$, $\dots$ according to its rank $l$ and type, we denote the finite group $\GF$ by $^o\tA_{l,\mathrm{sc}}(q)$, $\tB_{l,\mathrm{sc}}(q)$, $\dots$ where $o$ is the order of $\tau$ when $o>1$.

		We let $\III{E(\GF)}$ be the subgroup of $\Aut(\GF)$ generated by restrictions to $\GF$ of elements of $\Cent_{E(\bG)}(F)$, compare \cite[Thm~2.5.1]{GLS3} and \cite[Def.~2.1]{CS18B}. Note that the kernel of that restriction $\Cent_{E(\bG)}(F)\to\Aut(\GF)$ is generated by $F$, see \cite[Lem. 2.5.7]{GLS3}. This allows for instance to consider stabilizers $\big(\GF E(\GF)\big)_\bK$ for $\bK$ any closed $F$-stable subgroup of $\bG$.

		Let ${\bG\leq \wbG}$ be a regular embedding, i.e., a closed inclusion of algebraic groups with 
		$\II Gbtilde@{\wbG}
		=\Z( \wbG)\bG$ and connected $\Z( \wt\bG)$. Then $\wbG$ can be chosen so that the action of $E(\bG)$ extends to $\wbG$, see \cite[Sect. 2]{MS16}, \cite[Sect. 2.A]{CS18B} or \cite[Prop.~1.7.5]{GM}. We then obtain $F:\wbG\ra \wbG$ a Frobenius endomorphism extending the one of $\bG$ and an action of $E(\GF)$ on $\wGF$.
		
		We will often deal with the semi-direct products $\bG\rtimes E(\bG)$, $\bG^F\rtimes E(\bG^F)$ and $\wbG^F\rtimes E(\GF)$. Let us recall that $ \wGF\rtimes E(\GF)$ induces all automorphisms of $\GF$ when $\GF$ is a perfect group, see \cite[Thms 2.2.7, 2.5.1]{GLS3}.
		
		The automorphisms of $\GF$ induced by $\wbG^F$ or $(\bG/\Z(\bG))^F\cong \wGF/\Z(\wbG)^F$ are called \textit{diagonal automorphisms}.\index{diagonal automorphisms}\ 
		If we are only interested in the action on characters, $\wGF$ may be replaced by the action of $\II ZF@{\calZ_F}:=\wGF/(\GF\Z(\wGF))$. The isomorphism \begin{equation}\label{cZF}
		\calZ_F=(\bG\Z(\wbG))^F/\GF\Z(\wbG)^F\xrightarrow\sim \Z(\bG)_F=\Z(\bG)/[\Z(\bG),F]
		\end{equation} given by Lang's theorem (see \cite[Prop. 8.1(i)]{CE04}) allows one to associate the diagonal automorphism of $\GF$ induced by $\wt g\in\wGF$ with the element $t^{-1}F(t) [\Z(\bG),F]\in \Z(\bG)_F$ where $t\in \bG\cap \wt g\Z(\wbG)$, see also \cite[Prop.~6.3]{Cedric}.

\subsection{The Conditions $\Ap$ and $\Ai$}\label{Ainfty}

We recall the conditions $ {A}(\infty)$ and $\Ap$ in the form of \cite[Cond. 2.3]{TypeD1} and its reformulation following \cite[Lem.~2.4]{TypeD1}. We keep $\bG$, $\wbG$, $F$, $E(\GF)$ as in \ref{not}.
\begin{condi}\label{recallAinfty} 
		\index{$\Ai$}\index{$\Ap$}
		\noindent
\begin{itemize} 
    \item[$\Ai$:] There exists some $E(\GF)$-stable $\wGF$-transversal $\TT$ in $\Irr(\GF)$, such that every $\chi\in\TT$ extends to $\GF E(\GF)_\chi$.
    Equivalently there exists for every $\chi\in \Irr(\GF)$ some $\wGF$-conjugate $\chi_0$ of $\chi$ such that $(\wGF  E(\GF))_{\chi_0}=\wGF_{\chi_0}  E(\GF)_{\chi_0}$ and  $\chi_0$ extends to $\GF E(\GF)_{\chi_0}$.
	\item [$\Ap$:] There exists some $E(\GF)$-stable $\wGF$-transversal $\TT$ in $\Irr(\GF)$.
 Equivalently there exists for every $\chi\in \Irr(\GF)$ some $\wGF$-conjugate $\chi_0$ of $\chi$ with $(\wGF  E(\GF))_{\chi_0}=\wGF_{\chi_0}  E(\GF)_{\chi_0}$.
		 \end{itemize} 
	\end{condi}

	\noindent Condition $\Ai$ implies Assumption 2.12(v) of \cite{S12} according to \cite[Lem.~2.4]{TypeD1}, and has been essential to our approach to the McKay Conjecture. As recalled earlier, it has been established in all types different from $\tD$ and $^2\tD$, see \cite{CS18B} and the references given there. 
	
	For types $\tD$ and $^2\tD$, a first step towards Condition $\Ai$ is the result of \cite{TypeD1} that assumes the same condition in smaller ranks for cuspidal characters. 
We write $\II{IrrcuspGF}@{\cusp(\bG^F)}$ for the set of characters of $\bG^F$ that are cuspidal in the usual sense of \cite[Sect. 9.1]{Carter2}.

 Note that $\cusp(\bG^F)$, the set of irreducible cuspidal characters of $\GF$, is stable under $\wGF E(\GF)$ and satisfies a stabilizer statement similar to $\Ap$ thanks to a theorem of Malle, see \cite[Thm 2.13]{TypeD1}. 
 
 The following would make sense in the same generality as in \ref{not} but we concentrate on the groups of type untwisted $\tD$  where $\Ai$ is not yet known.
	
	\begin{hyp}[Extension of cuspidal characters of $\tDlprimesc(q)$]\label{hyp_cuspD_ext}
		Let $\bG'=\tD_{l',\mathrm{sc}}(\FF)$ ($l'\geq 4$) and $F=F_p^{m'}:\bG'\ra \bG'$ an untwisted Frobenius endomorphism with $m'\geq 1$ and $p$, the characteristic of $\FF$, an odd prime. Considering the action of $\Z(\bG')_F\rtimes E(\bG'{}^F)$ on $\Irr(\bG'{}^F)$ thanks to (\ref{cZF}), one assumes that there exists some $E(\bG'{}^F)$-stable $\Z(\bG')_F$-transversal $\TT$ in $\cusp(\bG'{}^F)$ such that every $\chi\in \TT$ extends to $\bG'{}^F E(\bG'{}^F)_\chi$.
	\end{hyp}
This hypothesis is of course true whenever ${\bG'}^{F}$ satisfies Condition $\Ai$. The main result of \cite{TypeD1}, a weak version of \Cref{thm1}, will be our starting point. 
\begin{theorem}[{\cite[Thm~A]{TypeD1}}]\label{thm_typeD1}
Let $\bG$, $\wbG$ and $F$ as in \ref{not} such that $\GF=\tDlsc(q)$ with odd $q$. If Hypothesis \ref{hyp_cuspD_ext} holds for all $4\leq l'<l$ and $1\leq m'$, then there exists some $E(\GF)$-stable $\wGF$-transversal in $\Irr(\GF)$. 
\end{theorem}

    \subsection{Adjoint groups as dual groups} \label{duals} We now consider a group $\II{H}@{\bH}=\bG^*$ dual to the group $\bG$ of \ref{not} above. This essentially means that a maximal torus $\II{S0}@{\bS_0}$ has been chosen in $\bH$ along with an isomorphism $\II{delta}@{\delta}\colon \III{X(\bT_0)}:=\Hom(\bT_0,\FFtimes)\xrightarrow{\sim} \III{Y(\bS_0)}:=\Hom(\FFtimes,\bS_0)$ compatible with roots. In particular the fact that $\bG$ is simply connected is equivalent to the equality $\ZZ\Phi(\bH,\bS_0)=X(\bS_0)$ so that $\bH =\bH_{{\mathrm{ad}}}$ is of adjoint type with a root system dual to the one of $\bG$ (types $\tB_l$ and $\tC_l$ for $l\geq 2$ are exchanged, other types are unchanged). We will also consider the simply connected covering 
    \begin{equation}\label{piH0}
\II{pi}@{\pi}\colon\bH_0\to \bH.
	\end{equation}
	
	The groups $E(\bH)\cong E(\bH_0)$ are defined on root subgroups in the same way as in \ref{not}, the choice of a basis $\Delta$ implying also a choice $\Delta^\vee$ in $\bH$. We define a group antiautomorphism \begin{equation}\label{sigma*}E(\bG)\to E(\bH),\ \ \ \si\mapsto \II sigmastar@{\si^*}\end{equation} by sending $F_p$ to the element of $E(\bH)$ defined in the same way, also denoted by $F_p$ and if $\tau\colon\Delta\to\Delta$ is an automorphism of the Dynkin diagram, we define $\btau^*$ as $\ov{\tau^\vee}\inv$ where $\tau^\vee$ is the same symmetry of $\Delta^\vee$. This makes $\si\colon \bG\to\bG$ and $\si^*\colon\bH\to\bH$ dual in the sense of \cite[Def. 2.1]{CS13} since the map induced by $\si$ on $X(\bT_0)$ is the scalar $p$ (hence self-adjoint) when $\si =F_p$, while $\si$ induces an isometry in the cases where $\si=\btau$ and so our choice of $\si^*$ is in each case the adjoint of $\si$ through the isomorphism $X(\bT_0)\xrightarrow{\sim} Y(\bS_0)\cong\text{Hom}(\ZZ\Phi^\vee,\ZZ)$.
	
We define $F^*\colon\bH\to\bH$ according to the above (i.e. $F^*=F_p^{m}\btau^*$ if $F=F_p^{m}\btau$ for $m\geq 1$ and $\tau\colon\Delta\to\Delta$ an automorphism of the Dynkin diagram) but we still denote it as $F\colon\bH\to\bH$. This also lifts into $F\colon\bH_0\to\bH_0$ with the same definition. We define $E(\HF)$ as before by restriction of $E(\bH)$ to $\HF$.

	\subsection{Semisimple classes} \label{ssclasses} The semisimple classes in the dual group $\bH$ and the finite group $\bH^F$ play a major role in the parametrization of $\Irr(\GF)$ and will be our main theme in the whole paper. We write $\II Hss@{\bH_{\textrm{ss}}}$ for the set of semisimple elements of $\bH$ and for $s\in \bH_{\textrm{ss}}$, we write $\II s@{(s)} $ for its $\bH$-conjugacy class and $\II AHs@{\AHs}:=\Cent_{\bH}(s)/\Cent^\circ_{\bH}(s)$.
	
	Let $\II{ClssH}@{\Cl_{\textrm{ss}}}(\bH):=\{(s) \mid s\in \bH_{\textrm{ss}} \}$ denote the set of semisimple conjugacy classes of $\bH$. For a semisimple element $s\in\HF$, we denote by $[s]_\HF$ its conjugacy class in $\HF$. The set ${\Cl_{\textrm{ss}}}(\bH)$ is acted upon by $E(\bH)$ and we often consider the set $\Cl_{\textrm{ss}}(\bH)^F$ of $F$-stable classes. If $\cC\in \Cl_{\textrm{ss}}(\bH)$ is a semisimple class of $\bH$ with $F(\cC)=\cC$, then $\cC^F\neq\emptyset$ by Lang's theorem \cite[Thm 21.7]{MT}. If $s\in\cC^F$, and therefore $\cC=(s)$, then \begin{equation}\label{CFs_a}\cC^F=\bigsqcup_{a\in A_\bH(s)_F}{[s_a]_\HF}\text{\ \ \ (disjoint union),}\end{equation} where $s_a={}^hs$ for some $h\in \bH$ with $h\inv F(h)$ an element of $\Cent_\bH(s)$ whose image in $A_\bH(s)$ represents $a\in A_\bH(s)/[A_\bH(s),F]$, see \cite[Thm 26.7(b)]{MT}.
	
	Let us also recall how this relates with the simply connected covering $\II {pi}@{\pi}\colon \II{H0}@{\bH_0}\to\bH$, see (\ref{piH0}). For $s\in\bH_{\textrm{ss}}$ an injective group morphism 
\begin{equation}\label{omegas}\II omegas@{\omega_{s}}:\AHs \lra \Z(\bH_0)=\pi\inv(1) \text{\ \ is given by\ \ } a \Cent^\circ _{\bH}(s) \mapsto [a, s_0],\end{equation}
	where $ s_0 \in \pi\inv(s)$, see \cite[Sect. 8.A]{Cedric}. We denote its image as $\II{Bs}@{B(s)}:=\omega_{  s}(\AHs)$.
 By abuse of notation we also denote by $\omega_s$ the induced morphism $\Cent_\bH(s)\lra \Z(\bH_0)$. 
 Note that whenever $s\in\HF$ then $\omega_s$ is clearly $F$- and $\Cent_{\HF E(\HF)} (s) $-equivariant.
	
\subsection{Lusztig series of irreducible characters} \label{series}
When $\bT$ is an $F$-stable maximal torus of $\bG$ and $\theta\in\Irr(\bT^F)$, Deligne-Lusztig have defined $R_\bT^\bG (\theta)\in\ZZ\Irr(\bG^F)$ (see for instance \cite[Sect.~2.2]{GM}), a generalized character that depends only on the $\GF$-conjugacy class of the pair $(\bT,\theta)$. Duality between $(\bG,F)$ and $(\bH,F)$ yields a bijection 
\begin{equation}
	\label{Ttheta}
	(\bT,\theta)\stackrel{\bG}{\longleftrightarrow }(\bT^*,s)
\end{equation}
 between $\GF$-conjugacy classes of pairs $(\bT,\theta)$ as above and $\HF$-conjugacy classes of pairs $(\bT^*,s)$ with $\bT^*$ an $F$-stable maximal torus of $\bH$ and $s\in \bT^*{}^F$. One then sets ${R}_{\bT^*}^\bG(s) :={R}^\bG_\bT(\theta) $, see \cite[Sect. 2.6]{GM}. If $s\in\bH_{\textrm{ss}}^F$ we denote by $\II EG@{\cE(\GF,{[s]})}$ (Lusztig's {\it rational series} of characters) the subset of $\Irr(\GF)$ defined as $\cE(\GF,{[s]}):=\bigcup_{\bT^*}\Irr({R}_{\bT^*}^\bG(s))$ where $\bT^*$ ranges over the $F$-stable maximal tori of $\bH$ containing $s$ (i.e. $F$-stable maximal tori of $\Cent^\circ_\bH(s)$), see \cite[Def. 2.6.1]{GM}. Then Lusztig's theorem tells us that \begin{equation}\label{RatSer}
\Irr(\GF)=\bigsqcup_{[s]}\cE(\GF,{[s]}),
\end{equation} a disjoint union where ${[s]}$ ranges over the semisimple classes of $\HF$, see \cite[Thm 2.6.2]{GM}. One has a similar notion of {\it geometric series} $\cE(\GF,\cC)$ for $\cC\in\Cl_{\textrm{ss}}(\bH)^F$ which is the union of all sets $\Irr(R_{\bT^*}^\bG(s))$ for an $F$-stable maximal torus $\bT^*$ and $s\in \bT^*{}^F\cap \cC$. Note that when $s\in \bH_{\textrm{ss}}^F$, then 
\begin{equation}\label{SplitSer}\cE(\GF,(s))=\bigsqcup_{a\in A_\bH(s)_F}\cE(\GF,[s_a])\end{equation}	as a consequence of the splitting of $(s)^F$ into $\HF$-conjugacy classes recalled in \ref{ssclasses}.

	Concerning automorphisms, $\si\in E(\GF)$ acts on the rational series of characters of $\GF$ via applying $\si^*$ (see \ref{duals}) on the semisimple element or the associated $\HF$-class, i.e. 
\begin{equation}	\chi^\si \in \cE(\GF,[\si^*{}(s)]) \ \ \text{ for every } s\in\bH_{\textrm{ss}}^F\ ,\ \ \chi \in \cE(\GF,[s]),\label{rem_ntt} \end{equation} by a result of Bonnaf\'e, see \cite[Cor.~2.4]{NTT}.
	
The above splittings (\ref{RatSer}) and (\ref{SplitSer}) among series apply in fact to any connected reductive group $\bK$ with a Frobenius endomorphism $F$ defining it over a finite field. Then the set of {\it unipotent characters} is $\II UCh@{\UCh(\bK^F)}:=\cE(\bK^F,(1))=\cE(\bK^F,[1])$.  Restriction and inflation maps induce natural bijections \begin{equation}\UCh(\bK^F)\longleftrightarrow \UCh([\bK,\bK]^F)\longleftrightarrow \UCh((\bK/\Z(\bK))^F),\label{UchKK}\end{equation}  see \cite[Prop. 11.3.8]{DiMi2}.

Returning to $(\bG,F)$ a simple simply connected group with dual $(\bH,F)$, a {\it Jordan decomposition of characters} defines for every $s\in\bH_{\textrm{ss}}^F$, a bijection $$\cE(\GF,[s])\xrightarrow{\sim}  \II UChCent@{\Uch(\Cent_\bH(s)^F)}:=\bigcup_{\phi\in\Uch(\Cent_\bH^\circ(s)^F)}\Irr(\phi^{\Cent_\bH(s)^F}),$$
see for instance \cite[Sect. 11.5]{DiMi2}, \cite[Thm 2.6.22]{GM}.


	\section{An equivariant Jordan decomposition of characters}\label{EquJor}
	Let $\GF$ be a finite group of Lie type as in \ref{not} above. The aim of this chapter is to establish that a Jordan decomposition of characters of $\GF$ (see \ref{series} above) can be chosen to be compatible with group automorphisms, whenever Condition $\Ap$ holds, see \Cref{thm66}. 

Let $(\bH, F)$ be associated with $(\bG ,F)$ as in \ref{duals}. We consider the following set that parametrizes $\Irr(\GF)$ through Jordan decomposition.

\begin{defi} \label{JorGF}
	Let $\II JorGF@{\Jor(\GF)}$ be the set of $\bH^F$-conjugacy classes of pairs $(s,\phi)$ where $s\in\bH_{\textrm{ss}}^F$ and $\phi\in   \Uch(\Cent_\bH(s)^F)$, where 
 $\Uch(\Cent_\bH(s)^F):=\bigcup_{\phi\in\Uch(\Cent_\bH^\circ(s)^F)}\Irr(\phi^{\Cent_\bH(s)^F})$ as in \ref{series}.
\end{defi}
	 As sketched in our introduction, an action of $\Out(\GF)\cong\cZ_F\rtimes E(\GF)$ can be defined on $\Jor(\GF)$, see \Cref{lem3_7} below. Then our strategy to define a $\cZ_F\rtimes E(\GF)$-equivariant Jordan decomposition $\Irr(\GF)\to \Jor(\GF)$ is given by the following trivial lemma. It is somehow implicit in the proof of \cite[Prop. 2.12]{S12} and the idea was applied on several similar occasions.  
	 
\begin{lem}\label{ZEsets} Let $Z$ and $E$ be finite groups, such that $E$ acts on $Z$ and $Z$ is abelian. Let $\cI, \cJ$ be $Z\rtimes E$-sets. 
	 	Let $\ov \cI$, $\ov \cJ$ be the $E$-sets given by the $Z$-orbits of $\cI$ and $\cJ$ respectively. Let $$\ov{\frak J}\colon\ov \cI\longrightarrow{} \ov \cJ$$ be an $E$-equivariant map. 
	 	 Assume that\begin{asslist} \item $\cI$ and $\cJ$ satisfy $(A')$, namely every $Z$-orbit in $\cK\in\{ \cI,\cJ \}$ contains some $x$ with  $(ZE)_x=Z_xE_x$ ;
   			\item  $Z_x=Z_y$ for every $x\in \cI$ with $Z$-orbit $\ov x$ and $y\in \ov{\frak J}(\ov x)$ ; and	\item  $\ov{\frak J}$ is a bijection.
 		\end{asslist} 
	 Then \ $\ov{\frak J}$ lifts to a $ZE$-equivariant  bijection $${\frak J}\colon \cI\xrightarrow{\sim} \cJ .$$
	 	\end{lem}
	 \begin{proof}
	 	 By (i) there are $E$-stable $Z$-transversals $\Bbb I$ and $\Bbb J$ in $\cI$ and $\cJ$, respectively, see \cite[Lem.~2.4 ]{TypeD1}. Then $\ov{\frak J}$ induces an $E$-equivariant bijection ${\frak J}\colon\Bbb I\to\Bbb J$ by (iii). For any $x\in \cI$, one has $x=z.x_0$ for some $z\in Z$ and a unique $x_0\in \Bbb I$. We then define ${\frak J}(x)=z.{\frak J}(x_0)$. This is independent of the choice of $z$ since $Z_{x_0}=Z_{{\frak J}(x_0)}$ by (ii). The map ${\frak J}$ so defined is clearly bijective and $ZE$-equivariant. 
 	 \end{proof}
  
	We plan to apply the lemma with $\cI$ being $\Irr(\GF)$, $\cJ$ being $\Jor(\GF)$ and $\cZ_F\rtimes E(\GF)$ acting on them. We will prove first in \ref{ExUnCh} below that $\Jor(\GF)$ indeed satisfies an analogue of Condition $\Ap$ from \ref{Ainfty}. Then we will have to check an isomorphism of $E(\GF)$-sets between the $\cZ_F$-orbits in $\Irr(\GF)$ and $\Jor(\GF)$. This is essentially due to Lusztig through the theorem \cite[Prop. 5.1]{L88} comparing his Jordan decompositions   $\Irr(\wGF)\longleftrightarrow\Jor(\wGF)$ and  $\Irr(\GF)\longleftrightarrow\Jor(\GF)$, see also \cite[Chap.~15] {CE04}, \cite[Sect. 2.6]{GM}. We rephrase this in \Cref{Jordandec} as a variant of Jordan decomposition $\ov \JJ\colon \Irr(\GF)/\cZ_F\text{-action}\to \Jor(\GF)/\cZ_F\text{-action}$ for characters of $\GF$ in terms of $\wGF$-orbit sums. Those orbit sums are non-irreducible characters that are introduced in \ref{3B_orbitsums}. They will be used in our discussion of  Condition $\Ai$ in terms of counting invariant characters throughout the rest of the paper.

\subsection{Extending unipotent characters}\label{ExUnCh}
Our goal is to prove that the parametrizing set $\Jor(\GF)$ satisfies an analogue of $\Ap$, see \Cref{lem3_7}. The main task is to extend the unipotent characters of some $\Cent_{\bH}^\circ(s)^F$ to their stabilizer in $\Cent_{\bH^FE(\HF)}(s)$. We apply results of Lusztig and Malle, see also \cite[Hyp. 3.3]{DM90} and \cite[Prop. 11.5.3]{DiMi2}. A weak version of the statement below was shown in the proof of \cite[Thm 7.2]{FS} with a different argument.

\begin{theorem}\label{propunipext}
	Let $\bH$ be a simple algebraic group of adjoint type and $F\colon \bH\to\bH$ be a Frobenius endomorphism as in \ref{duals}. Denote by $\wh H:=\bH^F E(\HF)$ the semidirect product of the finite group $\HF$ with its field and graph automorphisms. 
	 Let $\bC$ be a closed connected $F$-stable reductive subgroup of $\bH$ and let $\wh C$ be its stabilizer under the action of $\bH^F E(\HF)$. Then maximal extendibility (see \ref{not11}) holds \wrt $\bC^F\unlhd \wh C$ for $\UCh(\bC^F)$.
		
\end{theorem}

A main ingredient is Malle's result from \cite{MaExt} asserting that every unipotent character of a simple group $S$ of Lie type extends to its stabilizer in $\Aut(S)$, see also \Cref{rem_extun} below.

\begin{lem}\label{sim_uni} 
	Let $\bM$ be a simple simply connected algebraic group endowed with an $\FF_q$-structure given by a Steinberg endomorphism $F:\bM\lra \bM$. Let $\Aut_F(\bM^F)$ be the group of automorphisms of $\bM^F$ induced by bijective algebraic endomorphisms of $\bM$ commuting with $F$. Then any element of $\UCh(\bM^F)$ has a kernel containing $\Z(\bM^F)=\Z(\bM)^F$ (see \cite[Rem.~6.2]{Cedric}) and, seen as a character of $\bM^F/\Z(\bM^F)$, extends to its stabilizer in $\Aut_F(\bM^F)$.
\end{lem}
\begin{proof} When 
	$S:=\bM_{\mathrm{}}^F/\Z(\bM_{\mathrm{}}^F)$ is a simple group, this is indeed Malle's theorem, see \cite[Thm~2.4]{MaExt}. Note however that the proof given in \cite{MaExt} actually applies when assuming only that $\bM$ is simple simply connected since the simplicity of $S$ is used there only through the fact that the automorphisms of $S$ then come from bijective algebraic endomorphisms of $\bM_{\mathrm{ }}$ commuting with $F$, see beginning of proof of \cite[Prop.~2.2]{MaExt}. 
\end{proof}

The above will be used in conjunction with the following general statement.

\begin{lem}
	\label{ext_ker} Let $N$, $A$ be finite groups, let $\chi\in\Irr(N)$ with $\Z(N)$ in its kernel. Assume $N/\Z(N)\leq A\leq \Aut(N)$ and that $\chi$, seen as a character of $N/\Z(N)$ extends to $A_\chi$. Then the character  $\chi$ extends to $G_\chi$, whenever $G$ is a finite group with $N\unlhd G$ and the image of $G$ in $\Aut(N)$ is a subgroup of $A$. 
\end{lem}

\begin{proof}
We first extend $\chi$ into some $\wt\chi\in \Irr(N\Cent_G(N))$ by imposing that $\Cent_G(N)$ is in the kernel of $\wt\chi$, which is possible since $\Cent_G(N)\unlhd G$ and $N\cap\Cent_G(N)=\Z(N)$ is in the kernel of $\chi$. Now extending $\chi$ to its stabilizer in $G$ is equivalent to extending $\wt\chi$ seen as a character of $N \Cent_G(N)/\Cent_G(N)$ to its stabilizer in $G/\Cent_G(N)$. But $G/\Cent_G(N)$ is the image of $G$ in $\Aut(N)$, so our problem reduces to the initial situation $N/\Z(N)\leq A$.
\end{proof}

\begin{proof}[Proof of Theorem~\ref{propunipext}.]
	Recall $\wh H=\bH^F E(\HF)$ and $\wh C=\wh H_\bC$. As explained in 1.B, we can define the stabilizer under $\wh H$ of an $F$-stable subgroup of $\bH$, such as $\bC$. Define $\bZ=\Z(\bC)$ and $\bC_\ad:=\bC/\bZ$. Then $\wh C$ acts on both $\bC^F$ and $\bC_\ad^F$. On the other hand $\bC_\ad$ is a direct product of simple algebraic groups $\bC_1,\dots ,\bC_t$ with trivial centers that are the minimal normal closed connected subgroups of $\bC_\ad$, see \cite[Thm 8.21]{MT}. They are permuted by any element of $\bH^F\Cent_{E(\bH)}(F)$ normalizing $\bC$, in particular by $F$. Let $m_i$ be the cardinality of the orbit of $i$ under $F$, so that $F^{m_i}(\bC_i)=\bC_i$ and $$\bC_\ad^F= \prod_{i\in\Bbb I}C_i\text{ with }C_i:=\bC_i^{F^{m_i}}$$ for a transversal $\Bbb I$ of the $i$'s under the permutation induced by $F$. Denote by $\Aut_{F^{m_i}}(C_i)$ the restrictions to $C_i=\bC_i^{F^{m_i}}$
	of the bijective endomorphisms of $\bC_i$ commuting with $ {F^{m_i}}$. One has $\Aut_{F^{m_i}}(C_i)=C_i E(C_i)$ in the notation of \ref{duals}, see \cite[Thm 1.15.6]{GLS3}.

	Both $\bC^F$ and $\Cent_{\wh C}(\bC_\ad^F) $ are normal in $\wh C$, their intersection being $\bZ^F$ since $\Cent_{\bC_\ad}(\bC_\ad^F)=\Z(\bC_\ad)=1$ by \cite[Lem.~6.1]{Cedric}. Let us show that the action of $\wh C$ on $\bC_\ad^F $ induces an injection $$ \wh C/\Cent_{\wh C}(\bC_\ad^F) \hookrightarrow \prod_{i\in\ov{\Bbb I}}\Aut_{F^{m_i}}(C_i)\wr \Sym_{c_i}$$ where $\ov{\Bbb I}$ is a transversal in ${\Bbb I}$ modulo isomorphism between the corresponding $C_i =\bC_i^{ F^{m_i}}$'s and $c_i$ is the number of $j\in\Bbb I$ with $C_j\cong C_i$. This essentially comes from what has been said above about $\wh C$ acting on the components $\bC_i$. An element of $\wh C$ permutes the components $\bC_i$ of $\bC_\ad$, so it induces an automorphism of $\bC_\ad^F= \prod_{i\in\Bbb I}\bC_i^{F^{m_i}}$ that permutes the terms $C_i=\bC_i^{F^{m_i}}$ ($i\in\Bbb I$) of the latter product. For a given $i\in\Bbb I$, the action of $\wh C_{C_i}$ on $C_i=\bC_i^{F^{m_i}}$ is by the restriction of a bijective algebraic endomorphism of $\bC_i$ that commutes with ${F^{m_i}}$. This gives the claimed wreath product structure.
	
	Let's take $\phi\in\UCh(\bC^F)$. It is trivial on  $\bC^F\cap\Cent_{\wh C}(\bC_\ad^F) =\bZ^F $ so we may extend $\phi$ into $\phi_\ad\in \Irr (\bC^F\Cent_{\wh C}(\bC_\ad^F) )$ by imposing that the subgroup $\Cent_{\wh C}(\bC_\ad^F) $, normal in $\wh C$, is in the kernel of $\phi_\ad$. One has clearly $\wh C_{\phi_\ad}=\wh C_\phi$.  To show our extendibility statement it now suffices to check the image of our situation through the injection above, namely that $\phi_\ad$ extends to its stabilizer in  $\prod_{i\in\ov{\Bbb I}}\Aut_{F^{m_i}}(C_i)\wr \Sym_{c_i}$, since $\phi_\ad $ has   $\Cent_{\wh C}(\bC_\ad^F) $ in its kernel.
	
	Let $C'_i:= \mathrm{O}^{p'}(C_i)=\pi_i((\bC_i)_\mathrm{sc}^{ F^{m_i}})  $ where $\pi_i$ is the simply connected covering of $\bC_i$, see \cite[Thm~24.15]{MT}.
	
	The adjoint quotient $\bC\to\bC_\ad$ yields an injection of $\bC^F/\bZ^F$ into  $C'$ such that   $\prod_{i\in\Bbb I}C_i
	'\leq C'\leq\prod_{i\in\Bbb I}C_i =\bC_\ad^F$, see \cite[Prop.~24.21]{MT}. 
	We then get a character $\phi'\in\UCh(C')$ which extends to a unique $\phi_{\Bbb I}\in \UCh( \prod_{i\in\Bbb I}C_i)$ and restricts to a unique $\phi_{\mathrm{sc}}\in \UCh(\prod_{i\in\Bbb I}C_i')$ through the natural identifications of (\ref{UchKK}) $$\UCh (C')\longleftrightarrow \UCh (\bC^F)\longleftrightarrow \UCh (\bC_\ad^F)\longleftrightarrow \UCh (\bC_{\mathrm{sc}}^F)=\UCh(\prod_{i\in\Bbb I}C_i')$$ which also imply that $\Aut_F(\bC_\ad^F)_{\phi'_{ }}$ fixes $\phi_{\Bbb I}$ and $\phi_{\mathrm{sc}}$. Now Lemma~\ref{ext_ker} and Lemma~\ref{sim_uni} imply that $\phi_{\mathrm{sc}}$ and therefore $\phi '$ extends to its stabilizer in $\prod_{i\in {\Bbb I}}\Aut_{F^{m_i}}(C_i)$. Then the following lemma gives an extension of $\phi'$ to its stabilizer in $\prod_{i\in\ov{\Bbb I}}\Aut_{F^{m_i}}(C_i)\wr \Sym_{c_i}$.
\end{proof}

\begin{lem}
	Let $X\unlhd Y$ be finite groups and assume maximal extendibility holds for a $Y$-stable subset $\cM$ of $\Irr(X)$. Let $n\geq1$ and $\cM^n$ be the corresponding subset of $\Irr(X^n)=\Irr(X)^n$. Then maximal extendibility holds \wrt $X^n\unlhd Y\wr\Sym_n$ for $\cM^n$.
\end{lem}

\begin{proof} Let $\chi_1,\dots ,\chi_n\in\cM$ and $\chi=\chi_1\times\dots\times\chi_n\in \Irr(X^n)$. Extending $\chi$ to its stabilizer in $Y\wr\Sym_n$ is equivalent to doing the same for $\chi^y$ for $y\in Y\wr\Sym_n$. Taking $y$ in $Y^n$, this allows to arrange that, for any $1\leq i,j\leq n$, $\chi_j$ is in the $Y$-orbit of $\chi_i$ if and only if $\chi_i=\chi_j$. This being now assumed, the stabilizer of $\chi$ in $Y\wr\Sym_n$ is a direct product along the fibers of  $i\mapsto \chi_i$, so we can assume that $\chi_1=\dots=\chi_n$ and therefore $(Y\wr\Sym_n)_\chi = Y_{\chi_1}\wr\Sym_n$. Taking now $\wt\chi_1\in\Irr(Y_{\chi_1}) $ an extension of $\chi_1$, ensured by assumption, we get $(\wt\chi_1)^n\in\Irr((Y_{\chi_1})^n)$ which is fixed under $\Sym_n$ and extends to $Y_{\chi_1}\wr\Sym_n$ by the representation theory of wreath products, see \cite[Cor.~10.2]{Navarro_book}.
\end{proof}

\begin{rem} \label{remtoTh2.3} (a) \Cref{propunipext} will be applied through \Cref{extCent} below where $\bH$ is as in \ref{duals} the dual group of $\bG$, hence the assumption that $\bH$ is of adjoint type. In view of the proof, this assumption could be removed up to defining $E(\bH^F)$ as a group of automorphisms of $\bH^F$ obtained by restriction of bijective endomorphisms of $\bH$ commuting with $F$.
	
	(b) The proof of \Cref{propunipext} simplifies whenever $p\geq 5$ since then the wreath product action argument is made more evident by the fact that the finite groups $C_i$ ($i\in \Bbb I$) are quasisimple, thus implying that the full automorphism group of $\prod_{i\in \Bbb I} C_i$ is a wreath product, see for instance \cite[Lem. 10.24(a)]{Navarro_book}.
\end{rem}

A consequence of the above \Cref{propunipext} is the answer to a question of M. Brou\'e (see \cite{Broue}), originally on $d$-cuspidal unipotent characters. See also Question (P) in \cite[Introduction]{TypeD1} for a more general problem.

\begin{cor}
	Let $\bG$ be a connected reductive group defined over a finite field with associated Frobenius $F\colon\bG\to\bG$, let $\bM$ be an $F$-stable Levi subgroup of $\bG$. Then maximal extendibility holds \wrt $\bM^F\unlhd\NNN_{\bG}(\bM)^F$ for $\UCh(\bM^F)$.
\end{cor}

\begin{proof} By definition of algebraic groups defined over a finite field, the group $\bG$ embeds as a closed subgroup of $\GL_n(\FF)$ where $F$ extends into the raising of matrix entries to the $q$-th power. Embedding further $\GL_n(\FF)\hookrightarrow \bH:=\PGL_{n+1}(\FF)$, we can now apply Theorem~\ref{propunipext} to that $\bH$ and $\bC:=\bM$. We get that every $\phi\in\UCh(\bM^F)$ extends to $\NNN_{\HF E(\HF)}(\bM,\phi)\geq \NNN_{\GF}(\bM,\phi)$. 
\end{proof}

Our main application of \Cref{propunipext} is to centralizers of semisimple elements. 
For any semisimple $s\in \bH^F$ the group $\Cent^\circ_\bH(s)$ is an $F$-stable connected reductive group hence taking $\Cent^\circ_\bH(s)$ as $\bC$ in \Cref{propunipext} immediately implies the following.

\begin{cor}\label{extCent} Let $(\bH , F)$ as in \ref{duals}.
	 If $s\in \bH_{\textrm{ss}}^F$ and $\phi\in \UCh(\Cent^\circ_\bH(s)^F)$, then $\phi$ has a $\Cent_{\bH^F E(\bH^F)}(s)_\phi$-invariant extension to $\Cent_\bH(s)^F_\phi$ (and in fact also extends to $\Cent_{\bH^F E(\bH^F)}(s)_\phi$).
\end{cor}

The above \Cref{extCent} is the main point in proving that $\Jor(\GF)$ (see \Cref{JorGF}) satisfies an analogue of Condition $\Ap$. We first define the action of $\wbG^F E(\GF)$ on  $\Jor(\GF)$,  sketched  in our Introduction.

Recall that ${\calZ_F}:=\wGF/(\GF\Z(\wGF))$ from (\ref{cZF}) induces the diagonal outer automorphisms of $\GF$ and acts on $\Irr(\GF)$. Let us recall how to associate a linear character of $\AHs^F=\CHFs/ 
\CoHFs$ to every element of $\calZ_F$.

\begin{defi}\label{omegahat} 
Let $s\in \bH_{\textrm{ss}}^F$. Combining the isomorphism $\calZ_F\ra \Lin(\Z(\bH_0)^F)$ given by duality, with the morphism 
$\omega_s: \AHs\lra \Z(\bH_0)$ (see (\ref{omegas}) in \ref{ssclasses}) we obtain a surjective group morphism 
$$\II omegahat0s@{\protect{\wh\omega _{ s}}}: {\mathcal Z}_F\to \Lin(\AHs^F) ,$$
see also \cite[Sect.~8A]{Cedric} and \cite[Sect.~8]{CS17A} where the map is denoted by $\wh\omega^0 _{\bG,s}$.
\end{defi}

\begin{prop}\label{lem3_7}  For $s\in\bHssF$, $\phi\in\UCh(\Cent_{\HF}(s))$, $z\in\cZ_F$, $\si\in E(\GF)$, set
	$$z.(s,\phi):=(s,\wh \omega _{s}(z)\phi )  \text{ \ \ \ \ 
	and \ \ \ \ \ }\sigma.(s,\phi )=(\sigma^*{}^{-1}(s), \phi\circ \sigma^*) $$
	where $\si^*\in E(\HF)$ is associated with $\si$ as in \ref{duals}. 
	\begin{thmlist}
	\item	This defines an action of $\calZ_F\rtimes E(\GF)$ on $\Jor(\GF)$.
		
	\item	$\Jor(\GF)$ satisfies $\Ap$, namely any element of $\Jor(\GF)$ has a $\cZ_F$-conjugate $x$ such that $(\cZ_FE(\GF))_x=(\cZ_F)_xE(\GF)_x$.
	\end{thmlist}
	
\end{prop}
\begin{proof} We prove (a). First $\phi\circ\si^*\in \UCh(\Cent_{ \bH}(\sigma^*{}\inv (s))^F)$ by (\ref{rem_ntt}). Then to check that our formulas define a group operation it suffices to verify 
	$$\sigma^*\circ\wh\omega _{s}=\wh\omega _{\sigma^*(s)}\circ \sigma \text{ on }{\calZ}_F$$
	whenever $s\in\HFss$ and $\sigma\in\EGF$. The proof is similar to the one of \cite[Lem.~8.4]{CS17A}. One just has to replace the argument given there about the graph automorphism in type $\tA$ by the remark made in \ref{duals} above that $\si$ and $\si^*$ are dual in the sense of \cite[Def.~2.1]{CS13}. This implies that $\calZ_F\rtimes E(\GF)$ acts on the set (called $\wh{\Jor}(\GF)$ in \cite[Sect. 8]{CS17A}) of pairs $(s,\phi)$ with $s\in\bHssF$, $\phi\in\UCh(\Cent_{\HF}(s))$. To get our claim (a) it remains to show that this action preserves the $\HF$-conjugacy relation. We see that as a consequence of $$\wh\omega _{s^h}(z)= \wh\omega _{s}(z)\circ c_h\ \text{in} \ \Lin(\Cent_{\HF}(s^h))$$ for $h\in \HF$, $z\in {\calZ}_F$ and $c_h\colon \Cent_{\HF}(s^h)=\Cent_{\HF}(s)^h\to \Cent_{\HF}(s)$ being conjugacy by $h$. The above can be in turn deduced from $\omega_{s^h}=\omega_{s}\circ c_h$ again as in the proof of \cite[Lem.~8.4]{CS17A}.

	We now show (b). Let $s\in \bHssF$ and $\phi\in\UCh(\Cent_\HF(s))$. Denote by $\ov{(s,\phi)}\in\Jor(\GF)$ its orbit under $\HF$-conjugacy. We must prove that there is $\ov{(s,\phi')}\in\Jor(\GF)$, a $\cZ_F$-conjugate of $\ov{(s,\phi)}$, such that $\left(\cZ_FE(\GF)\right)_{\ov{(s,\phi ')}}=(\cZ_F)_{\ov{(s,\phi')}}E(\GF)_{\ov{(s,\phi')}}$. Define  $$C:=\Cent_\bH^\circ(s)^F\leq\wt C:=\Cent_\bH(s)^F\leq\wh C:=\Cent_{\HF E(\HF)}(s)$$ with both $C$ and $\wt C$ normal in $\wh C$.
	
	By definition of $\UCh(\wt C)$, $\restr\phi|{C}$ is a sum of elements of $\UCh(C)$, let $\phi_0$ be one of them. Then \Cref{extCent} ensures the existence of an extension $\wt \phi_0$ of $\phi_0$ to $\wt C_{\phi_0}$ that is $\wh C_{\phi_0}$-invariant. Let $$\phi ':=(\wt \phi_0)^{\wt C}.$$ 
	By Clifford correspondence this is an irreducible character such that $\wh C_{\phi '}= \wh C_{\phi_0} \wt C= \wh C_{{\restr\phi'|{C}}}$. By definition $\phi '\in\UCh(\wt C)$. From the surjectivity of $\wh\omega_{s} $ we see that the action of $\cZ_F$ induces the multiplication by the whole $\Lin(\wt C/C)$ on $\UCh(\wt C)$. By Clifford theory again $\phi '$ differs from $\phi$ by a linear character of $\wt C$ with $\wt C_{\phi_0}$ in its kernel, so this  implies that $\ov{(s,\phi ')}$ is a $\cZ_F$-conjugate of the original $\ov{(s,\phi )}$. We now show that $(\cZ_F\rtimes E(\GF))_{\ov{(s,\phi ')}}=(\cZ_F)_{\ov{(s,\phi ')}} E(\GF)_{\ov{(s,\phi ')}}$ or equivalently $(\cZ_F\rtimes E(\GF))_{\ov{(s,\phi ')}}\leq (\cZ_F)_{\ov{(s,\phi ')}} E(\GF)_{\ov{(s,\phi ')}}$.

	Since $\wh C_{\phi '}=  \wh C_{{\restr\phi'|{C}}}$ the action of $\Lin(\wt C/C) \rtimes \wh C$ on $\Irr(\wt C)$ satisfies
	\[   (\Lin(\wt C/C) \rtimes \wh C)_ {\phi'}\leq  (\Lin(\wt C/C) \rtimes \wh C)_ {\restr \phi '|C}=  \Lin(\wt C/C) \rtimes \wh C_ {{\restr\phi'|{C}}}= \Lin(\wt C/C) \rtimes \wh C_ { \phi'} .\eqno(2.1) \]
	
	Let now $z\in \cZ_F$ and $\si\in E(\GF)$ such that $z\si$ fixes $\ov{(s,\phi')}$. We must show that at least $\si$, and therefore both $\si$ and $z$, fix $\ov{(s,\phi')}$. Let $\tau :=(\si^*)\inv\in E(\HF)$. Then $z\si .\ov{(s,\phi')}=\ov{(\tau (s),{}^\tau\phi' .\wh\omega _{\tau\inv(s)}(z))}$. The fact that $z\si$ fixes $\ov{(s,\phi')}$ amounts to $$ (\tau (s),{}^\tau\phi' .\wh\omega _{\tau\inv(s)}(z))  ={(s,\phi')}^h \eqno(2.2)$$ for some $h\in\HF$. Then $h\tau \in \wh C$ and we get a pair $(\omega,h\tau)$ with $\omega :={}^h \left(\wh\omega _{\tau\inv(s)}(z)\right)\in \Lin(\wt C/C)$ and $h\tau\in\wh C$ to which we can apply ($2.1$) above. This tells us that $h\tau$ fixes $\phi '$. Then ($2.2$) also holds if we substitute $z=1$, keeping the same $h$. This now means that $\sigma$ fixes $\ov{(s,\phi')}$ as claimed.	
\end{proof}

	\subsection{Orbit sums and disjoint unions}\label{3B_orbitsums}
	In the following we introduce some notation for the character sets involved in the Jordan decomposition of characters. 

	Recall that for any finite group $X$, every (non necessarily irreducible) character $\chi$ of $X$ defines the subset $\II Irrchi@{\Irr(\chi)}\subseteq\Irr(X)$. Conversely a subset $\calM\subseteq \Irr(X)$ defines the character $\sum_{\psi\in \calM} \psi$. Let $\II CharX @ {\Char(X)}={\Bbb Z_{\geq 0}}\Irr(X)$ be the set of characters of $X$, that is sums of elements of $\Irr(X)$. If a group $A$ acts on a finite group $X$, we study the $A$-orbits in $\Irr(X)$ by the associated (in general non-irreducible) characters of $X$ given by $A$-orbit sums. We explore this idea and introduce some notation. 

	\begin{defi}\label{def31}
		For finite groups $X\unlhd \wt X$ we define $\II Pi tildeX @{\Pi_{\tilde X }}: \Irr(X)\lra \Char(X)$ by $\phi\mapsto \sum_{x\in \wt X/\wt X_\phi} \phi^x$. We call $\Pi_{\wt X}(\phi)$ \textit{the $\tilde X$-orbit sum of $\phi$}. \end{defi}
Note that $\Irr(\Pi_{\wt X }( \phi))= \Pi^{-1}_{\wt X}(\Pi_{\wt X }( \phi))$ is the $\wt X$-orbit of $\phi$ and hence $|\Irr(\Pi_{\wt X}( \phi))|=|\wt X:\wt X_\phi|$.
 
 \medskip

 \begin{defi}[Disjoint unions and subsets]\label{multi}
We often consider in the following formal disjoint unions (coproducts) of sets: 
For $r\geq 1$ and sets $M_i$ ($1\leq i \leq r$), we denote by $\II coprod Mi@{\bigsqcup_{1\leq i \leq r} M_i}$  the (formal) disjoint union of the sets $M_i$ ($1\leq i \leq r$). If $M'_i\subseteq M_i$ for any $1\leq i \leq r$, one has ${\bigsqcup_{1\leq i \leq r} M'_i}\subseteq {\bigsqcup_{1\leq i \leq r} M_i}$. We have of course $|{\bigsqcup_{1\leq i \leq r} M_i}|={\sum_{1\leq i \leq r} |M_i|}$.
 	\end{defi}
 
The disjoint unions $\bigsqcup_{1\leq i \leq r} M_i$ considered later will be defined using character sets $M_i$, in general with some $r\leq 16$. The union of those sets form a finite set $M$ and the natural map  $\bigsqcup_{1\leq i \leq r} M_i\to M$ will have fibers of cardinality $\leq 4$, see sections 3.B, 3.C, 3.D, 4.B, 5.C below.

\begin{lem}\label{defPI_allg}
Let $X\unlhd \wt X$ be finite groups, $A:= \wt X/X$ and $\si\in \Aut(\wt X)_X$. Set $\II IrroverlineX@{{\ov{\Irr}}(X) }:=\Pi_{\wt X}(\Irr(X))$.
	\begin{thmlist}
		\item Let $\phi\in\Irr(X)$. Then $\Pi_{\wt X}(\phi)$ is $\si$-invariant if and only if $\phi$ is $b\sigma$-invariant for some $b\in A$, i.e., 		
		$$\Pi^{-1}_{\wt X} (\ov \Irr(X)^{\spa\si})= \bigcup_{b\in A}\Irr(X)^{\spa{b \sigma} }\subseteq \Irr(X).$$ 
 
			\item \label{lem72b}
			The map $\Pi_{\wt X}$ induces on the disjoint union $\bigsqcup_{b\in A}\Irr(X)^{\spa{b\si}}$ the surjection 
			$$\II PiXtildeXsi@{\Pi_{\tilde X,\si}}: \bigsqcup_{b\in A}\Irr(X)^{\spa{b \sigma}}\lra \ov \Irr(X)^{\spa\si}, $$
			that satisfies $|\Pi_{\wt X,\si}^{-1} (\ov \phi)|=|A|$ for each $\ov \phi\in \ov \Irr(X)^{\spa\si}$.
			
			\item \label{bij_Pi} Assume $A$ is abelian.
			Let $\wt \cE\subseteq \Irr(\wt X)$ be $\Lin(A)$-stable and $\cE:=\bigcup_{\wt \phi\in \wt \cE} \Irr(\restr \wt \phi|X)$. Assume maximal extendibility holds \wrt $X\unlhd \wt X$ for $\cE$. Then the map defined by $ \wt \phi\mapsto \restr \wt \phi|X   $ induces a bijection  between $ \Lin(A)$-orbits in $\wt \cE$ and $  \Pi_{\wt X}(\cE) $. Moreover, if $\wt\phi\in\wt\cE$ and $\phi\in\Irr(\restr \wt \phi|X)$, then $\Lin(A)_{\wt\phi}=(A_\phi)^\perp$ where orthogonality is for the perfect pairing $\Lin(A)\times A\to\CC^\times$ defined by $(\alpha ,a)\mapsto\alpha(a)$.
		\end{thmlist}
	\end{lem}

Note that in (b) above the disjoint union is between subsets of $\Irr(X)$ that have non-trivial intersection since the trivial character belongs to all.
\begin{proof} Part (a) is clear from the fact that $\Pi\inv _{\wt X}(\Pi_{\wt X}(\phi))$ is an $\wt X$-orbit. 

	For part (b) assume that $\ov \phi\in \ov \Irr(X)$ is $\si$-invariant and let $\phi\in \Pi_{\wt X,\si}^{-1} (\ov \phi)$. Then $\phi\in \Irr(X)^{\spa{b_0 \sigma}}$ for some $b_0\in A$ and then  $\phi \in \Irr(X)^{\spa{b'b_0 \sigma}}$ if and only if $b'\in A_\phi$. Hence $\phi$ occurs in the summands of $\bigsqcup_{b\in A}\Irr(X)^{\spa{b\si}}$ exactly $|A_\phi|$-times. Any other character of the $\wt X$-orbit of $\phi$ has a $\wt X$-conjugate stabilizer and occurs with the same multiplicity. This clearly leads to 
	\[|\Pi_{\wt X,\si}^{-1} (\ov \phi)|=|A_\phi| \ \cdot \ |A:A_\phi|=|A|.\] 

	 Part (c) can be found in  \cite[Sect.~9(a)]{L88} and \cite[Thm~15.12(ii)]{CE04}, recalling that maximal extendibility is equivalent to restrictions being without multiplicities, see \ref{not11}. 
	\end{proof}

\medskip

As an illustration of  \Cref{defPI_allg}(b) consider the case where $|A|=4$ and $\si=\id$. The preimage of an orbit sum has always cardinality $4$. If  $\ov \chi\in \ov \Irr(X)$ and $\phi\in \Pi_{\wt X,\id}\inv (\ov \chi)$, then 
$$  \Pi_{\wt X,\id}\inv (\ov \chi)=\begin{cases}
\{ \phi^{}\}\sqcup \{ \phi^{}\}\sqcup \{ \phi^{}\}\sqcup \{ \phi^{}\} & \text{ if } |A_\phi|=4,\\
\{ \phi^{}\}\sqcup \{ \phi^{}\}\sqcup \{ \phi^{a}\}\sqcup \{ \phi^{a}\} & \text{ if } |A_\phi|=2, \ a\in A\setminus A_\phi\\
 \{ \phi^{}\}\sqcup \{ \phi^{a}\}\sqcup \{ \phi^{b}\}\sqcup \{ \phi^{c}\} & \text{ if } |A_\phi|=1, \ A=\{1,a,b,c\}.\\
\end{cases}$$

Later we determine the cardinality of certain character sets $\cM\subseteq \ov \Irr(X)^{\spa \si}$ using the equality $|\Pi\inv_{\wt X,\si}(\cM) | = |A|\ \cdot |\cM|$, which is only true because we are here considering a disjoint union. 

\bigskip

\subsection{An equivariant Jordan decomposition of characters}\label{sec3B}

We go back to our groups $\bG^F$, $E(\GF)$, $\wbG$, $\cZ_F$, $\bH$, defined in \ref{not} and \ref{duals}.

The aim of this chapter is to prove \Cref{thm66}, namely that if $\GF$ satisfies Condition $\Ap$, then we can construct an equivariant Jordan decomposition $$\Irr(\GF)\longrightarrow\Jor(\GF).$$ Following the plan laid out by \Cref{ZEsets} and in view of \Cref{lem3_7}(b), we turn our attention to the sets of $\cZ_F$-orbits in $\Irr(\GF)$ and $\Jor(\GF)$ respectively, and exhibit an $E(\GF)$-equivariant bijection between them. 
This is ensured by Lusztig's results \cite{L88} about $\Irr(\GF)$ and the Jordan decomposition for $\wGF$. We give in \Cref{Jordandec} a formulation in terms of $\wGF$-orbit sums that will also be used in the rest of the paper.

In the following we essentially expound the content of \cite{L88}, see also \cite[Sect.~15.4]{CE04}, \cite[Sect.~2.6]{GM}, \cite[Sect.~11.5]{DiMi2}. We emphasize the action of automorphisms, using orbit sums.

The regular embedding $\bG\leq \wbG$ has the important property that the inclusion of finite groups $\GF\unlhd \wbG^F$ satisfies maximal extendibility according to  \cite[Prop. 10]{L88}, see also  \cite[Thm~15.11]{CE04}, or in other words character restrictions from $\wbG^F$ to $\bG^F$ are without multiplicities. Hence \Cref{defPI_allg}(c) can be applied to it. The embedding $\bG\leq\wbG$ has a dual 
$$ \II pitilde@{\wt \pi}\colon \wbH\to \bH\ ,$$ an epimorphism with connected central kernel, where $\II Htilde@{\wbH}:=(\wbG)^*$. Note that the simply connected covering $\pi$ from (\ref{piH0}) factors through $\wbH$, and this defines a regular embedding $\bH_0\leq \wt\bH$. A Frobenius endomorphism also called $F$ is defined on $\wt\bH$, and it extends the one given on $\bH_0$. 

\medskip

\begin{notation}\label{not3_4} 	If $v\in \bH$ and $s\in\bH_{\textrm{ss}}^{vF}$, with $C:=\Cent_{\bH}^\circ(s)^{vF}\unlhd\wt C:=\Cent_{\bH}(s)^{vF}$, we set
		$$\II Uchoverline C@{\oUCh(C)}:= \Pi_{\wt C }(\UCh (C)),$$ 
		where $\Pi_{\wt C}:\Irr(C)\lra \Char(C)$ is defined as in Definition~\ref{def31} with respect to $C\unlhd \wt C$.
		Let $\Pi_\wGF:\Irr(\GF)\lra \Char(\GF)$ be the map defined \wrt $\GF\unlhd\wGF$ and set $$\II IrrGFov@{\ov\Irr(\GF)}=\Pi_\wGF(\Irr(\GF))=\{ \restr\wt\chi|{\GF}\mid \wt\chi\in\Irr(\wGF) \}.$$ 
		
		Let $\II JoroGF@{\overline {\textrm {Jor}}(\GF)}$ be the set of $\HF$-orbits in $\{ (s, \ov \phi) \mid s\in\bH_{\textrm{ss}}^F, \ \ov \phi\in \oUCh(\Cent^\circ_\bH(s)^F ) \}$. We get an action of $E(\GF)$ on $\ov\Jor(\GF)$ by the same formula as in \Cref{lem3_7}. 	Let 
		$\II JortildeGF@{\Jor(\wGF)}$ be the set of $\wHF$-orbits on $\{ (\wt s, \phi) \mid \wt s\in\wbH_{\textrm{ss}}^F, \phi\in \UCh(\Cent_{\wHF}(\wt s) ) \}$.
		
		Let $s\in\HFss$.  We set 
		$$\II EGFs@{{\ocE} (\GF,(s))} :=\Pi_{\wGF}\big(\cE(\GF,(s))\big) \und 
		\II EeGFsrat@{{\ocE} (\GF,[s])}:= \Pi_{\wGF}\big( \cE(\GF,[s])\big),$$ see \ref{series} for Lusztig's series of characters. Recall $A_\bH(s)=\Cent_{ \bH}(s)/\Cent_{ \bH}^\circ(s)$ and the epimorphism $\omega_s\colon A_\bH(s)\to\Z(\bH_0)$ of (\ref{omegas}). Recall ${B(s)}:=\omega_s(A_\bH(s))= \Z(\bH_0)\cap\{[s_0,g^*]\mid s_0\in\pi\inv(s), g^*\in\bH_0 \}$ which is also a subgroup of $\Z(\wbH)$. The duality between $\wbG$ and $\wbH$ then gives a map $B(s)^F\to\Lin(\wGF/\GF)$ (see \cite[Prop. 11.4.12]{DiMi2}), hence an action of $B(s)^F$ on $\Irr(\wGF)$.
\end{notation}
	
The Jordan decomposition for characters of groups with connected center can be summed up as follows, see \cite[Thm~15.13]{CE04} and \cite[Thm 2.6.21]{GM}. The equivariance comes from \cite[Thm~3.1]{CS13}.   
			 
\begin{prop}\label{JorTilde} There is an $E(\GF)$-equivariant bijection  
			 	\[ \II{Jtilde}@{\protect{\wt\JJ}} : \Irr(\wGF) \lra \Jor(\wGF) \]
			 	inducing bijections on rational series
			 	$$\II{Psis}@{\Psi_{\wt s}}: \cE(\wGF,[{\wt s}]_{})\xrightarrow{1-1} \UCh(\Cent_{\wt\bH}({\wt s})^F) \text{ with }\wt\chi \mapsto \phi, $$
			 	whenever $({\wt s},\phi)\in \wt \JJ(\wt \chi)$. 
    For $s:=\wt\pi(\wt s)$, $\Psi_{\wt s}$ is equivariant for the given actions of $B({  s})^F$ on $\cE(\wGF,[{\wt s}]_{})$ and $A_\bH(s)^F$ on $\UCh(\Cent_{\wt\bH}({\wt s})^F)$   via the isomorphism $B({  s})^F\cong A_\bH(s)^F$ by $\omega_{  s}$, i.e. for a character $\wt \chi\in\cE(\wGF,[{\wt s}]_{})$ and for $(\wt s,\phi)\in \wt \JJ (\wt \chi)$, we have
	\[A_\bH ({ s})^F_{\phi}\cong \omega_{  s}(A_\bH ({  s})^F_{\phi}) =B({  s})^F_{\wt \chi}.\] \end{prop} 

The natural epimorphism ${\wt \pi:\wbH\ra\bH}$ induces a surjective map \[\II{pi}@{\pi^*}: \Jor(\wGF)\ra \ov\Jor(\GF) \] 
defined by $(\wt s,\phi)\mapsto ( \wt \pi(\wt s), \Pi_{\Cent_\bH(\wt \pi(\wt s))^F } (\wc \phi) )$, where $\wc \phi$ is the character of $\Cent_\bH^\circ (\wt \pi(\wt s))^F$ given by $\phi$ via $\wt \pi$. This is well-defined since $\wt\pi(\Cent_\wbH(\wt s))=\Cent^\circ_\bH(\wt\pi(\wt s))$ and $\wt\pi(\Cent_\wbH(\wt s)^F)=\Cent^\circ_\bH(\wt\pi(\wt s))^F$, remembering that $\wt\pi$ has connected kernel. This also gives $\wt\pi(\wHF)=\HF$ and the surjectivity on orbits.

\begin{prop}[Rational Jordan decomposition of $\ov \Irr(\GF)$]\label{Jordandec}
The following diagram is commutative
\[\xymatrix{
	\ov\Irr(\GF) \ar[r]^{\ov\JJ}_{\sim}&\ov\Jor(\GF)\\
	\Irr(\wGF)\ar[r]_{\wt\JJ}^{\sim}\ar@{->>}[u]^{\mathrm{res}} &\Jor(\wGF)\ar@{->>}[u]_{\pi^*},
}\]
where $\mathrm{res}$ is the restriction map $\wt\chi\mapsto \restr\wt\chi|{\GF}$  and  $\II{J overline}@{\ov\JJ}:\ov\Irr(\GF)\lra \oJor(\GF)$ is well-defined by 
	$\restr \wt \chi|{\GF}\mapsto \ \ \pi^*(\wt\JJ(\wt \chi))$. Moreover,
\begin{thmlist} \item ${\ov\JJ}$ is an $E(\GF)$-equivariant bijection. 
\item The following statements for $\wt \chi\in\Irr(\wGF)$ and $\ov \chi \in \ov\Irr(\GF)$ are equivalent: 
\begin{asslist}
	\item $\restr \wt\chi|\GF=\ov \chi$,
	\item there exists some $\wt s\in \wt \bH^F_{{\textrm {ss}}}$ and $\phi'\in\UCh(\Cent_\wbH(\wt s)^F)$ with $(\wt s, \phi')\in \wt \JJ(\wt \chi)$ such that $(s, \Pi_{\Cent_\bH(s)^F}(\phi))\in \ov \JJ(\ov \chi)$, where $s:=\wt \pi(\wt s)$ and $\phi\in \Irr(\Cent_\bH^\circ(s)^F)$ is the character of $\wt\pi(\Cent_{\wHF}(\wt s))=\Cent_\bH^\circ(s)^F$ induced by $\wt\pi$ and $\phi'$.
\end{asslist}
\item  \label{Jdnuconst}
If $\ov \chi\in \ov \Irr(\GF)$, $\chi\in \Irr(\ov \chi)$, $(s,\ov \phi)\in \ov \JJ(\ov \chi)$ and $\phi_0\in \Irr(\ov \phi)$, then $\wGF_\chi$ coincides with the intersection of the kernels of all linear characters of $\wGF$ corresponding to the elements of $\omega_s(\AHs^F_{\phi_0})$ through the isomorphism $\Z(\bH_0)^F\to\Lin(\wbG^F/\GF\Z(\wbG^F))$ given by the duality between $\bG$ and $\bH$, see  \cite[Prop.~11.4.12]{DiMi2}. Then 
\[|\Irr(\ov \chi)|\, \cdot \, |\Irr(\ov \phi)| =|\AHs^F|.\] 
\end{thmlist}
\end{prop}
Note that Lusztig's Jordan decomposition of characters of $\GF$ given in 
\cite[Cor.~15.14]{CE04} or \cite[Thm~2.6.22]{GM} is deduced from the bijections  
\[ \wGF\text{-orbits in } \cE(\GF,[s]) \longleftrightarrow \Lin(\Cent_\HF(s)/\Cent_\bH^\circ(s)^F) \text{-orbits on }\UCh(\Cent_\HF(s)),\]
for any $s\in\bH_{\textrm{ss}}^F$. 
The map $\ov \JJ$ is a rephrasement of this bijection, since 
the $\Lin(A_\bH(s)^F)$-orbits on $\UCh(\Cent_\HF(s))$ are in bijection with $\Pi_{\CHFs}(\UCh(\CoHFs))$, by Lemma \ref{bij_Pi} and \Cref{extCent} (or  \cite[Prop. 11.5.3]{DiMi2}).

\begin{proof} First we show that $\ov \JJ$ is well-defined. For a given $\ov \chi$ in $\ov \Irr(\GF)$ let $\wt \chi_1,\wt \chi_2\in \Irr(\wGF)$ with $\ov \chi=\restr \wt \chi_1|{\GF}=\restr \wt \chi_2|{\GF}$ and therefore $\wt \chi_2 = \wh z \wt \chi_1$ for some $z\in \Z(\wHF)$.  If $(\wt s, \phi_1)\in \wt \JJ(\wt \chi_1)$, then $(\wt s z, \phi_1)\in \wt \JJ(\wt \chi_2)$ by \cite[Thm~4.7.1(3)]{GM} and it is then clear that the $\wt\bH^F$-orbits of $(\wt s, \phi_1)$ and $(\wt sz, \phi_1)$ have the same image under $\pi^*$. So the map $\ov \JJ$ is well and uniquely defined by the commutative diagram expressing $\pi^*\circ \wt \JJ=\ov \JJ\circ \textrm{res}$. This also gives surjectivity and $E(\GF)$-equivariance of $\ov \JJ$ by the corresponding properties of the three other maps, see \Cref{JorTilde} for $\wt \JJ$. 
To finish proving (a) we must show that $\ov \JJ$ is injective, namely that if $\wt\chi_1,\wt\chi_2\in\Irr(\wbG^F)$ have same image under $\pi^*\circ \wt \JJ$ then they have same restriction to $\GF$. Since $ \wt \JJ$ is a bijection it is equivalent to take $(\wt s_1,\phi_1 ), (\wt s_2,\phi_2 )\in \JJ(\wbG^F)$ with same image by $\pi^*$. The consequence on the semisimple elements is that $\wt s_2=\wt z\wt s_1$ for a central $\wt z$. But then as seen above up to multiplying $\wt \chi_2$ by a linear character trivial on $\GF$, one can assume $\wt s_2=\wt s_1$. Then $(\wt s_1,\phi_1 ) \und (\wt s_1,\phi_2 )$ having the same image by $\pi^* $ would mean that $\phi_1$ and $\phi_2$ are conjugate under $A_\bH(\wt\pi(s_1))^F$ and using the compatibility of $\wt J$ with the actions of $A_\bH(\wt\pi(s_1))^F$ and $B(\wt\pi(s_1))^F$ recalled in \Cref{JorTilde}, we indeed get that $\wt\chi_1$ and $\wt\chi_2$ have same restriction to $\GF$.

The above discussion also proves (b).


(c) By Lemma~\ref{bij_Pi}  $\wGF_\chi$ is with the intersection of the kernels of linear characters of $\wGF$ given by the elements of  $\omega_s(\AHs^F_{\phi_0})=B(s)^F_\chi$. Then the index $|\wGF:\wGF_\chi|$  coincides with $|\AHs^F_{\phi_0}|$. Hence by orthogonality in the perfect pairing between $\AHs^F$ and $\Lin(\AHs^F)$, we obtain 
\begin{align*}
|\AHs^F| &=|\Irr(\ov \chi)|\, \cdot \, |\Irr(\ov\phi)|.
\qedhere\end{align*} 
\end{proof}

We can also deduce from \Cref{Jordandec} a parametrization of $\ov \cE(\GF,[s])$ for every $s\in \HFss$. 
\begin{cor}\label{cor_Jordandec}
For every $s\in\HFss$ the map $\ov \JJ$ from \Cref{Jordandec} induces a bijection 
$$\II {Psi overline s }@{\ov\Psi_s}:\ocE(\GF,[s]) \lra \oUCh(\Cent_\bH^\circ(s)^F) $$
with the following properties: 
\begin{thmlist}
\item If $\ov \chi\in\ocE(\GF,[s])$ and $\sigma\in E(\GF)$ with dual 
$\si^*\in E(\HF)$ as in \ref{duals}, then 
\[\ov\Psi_s(\ov \chi)^{\si^*{}\inv}=\ov\Psi_{\si^*{} (s)}(\ov \chi^{\si })\ \ \text{in}\ \  \ov\UCh(\Cent_{ \bH}^\circ(\si^*(s))^F) 
\und \]
\[
|\AHs^F|=|\Irr(\ov \chi)|\, \cdot \, |\Irr(\ov\Psi_s(\ov\chi))|  .\]
\item If $\si\in E(\GF)$ and $\si^*(s)=s^h$ for some $h\in\HF$, then 
\[\ov\Psi_s(\ov \chi)^{(h\si^*{})\inv}=\ov\Psi_{s}(\ov \chi^{\si })
\]
in particular
\[\ov \Psi_s(\ov \cE(\GF,[s])^{\spannsi}) = \oUCh( \Cent^\circ _\bH(s)^F )^{\spann<h\si^*>} .\]
\end{thmlist}
\end{cor}
\begin{proof}
By the construction of $\ov \JJ$ using $\wt \JJ$ we see that in each $\HF$-orbit of some $\chi\in\ov\cE(\HF,[s])$ there is exactly one pair of the form $(s,\ov \phi)$ with some $\ov \phi\in \oUCh(\Cent_\bH^\circ (s)^F)$. As $\ov \JJ$ is a bijection the map $\ov \Psi_s$ is bijective as well. This ensures that the required bijection $\ov \Psi_s$ exists. Note that $\ov\Psi_s=(\ov\Psi_{^h s})^{h }$ for any $h\in\HF$ since $\ov\JJ$ has values in $\HF$-orbits.

The first part of (a) is a natural consequence of the $E(\GF)$-equivariance of $\ov \JJ$ and the definition of $\ov \Jor(\GF)$. The second part of (a) is a reformulation of Proposition~\ref{Jdnuconst}. For (b) we can write, for any $\ov\chi\in\ov\cE(\GF,[s])$, $\ov\Psi_{s}(\ov \chi^\si)^{}=\ov\Psi_{^h\si^*(s)}(\ov \chi^\si)^{}=\ov\Psi_{\si^*(s)}(\ov \chi^\si)^{h\inv}=\ov\Psi_{s}(\ov \chi)^{(h\si^*)\inv}$, the last equality thanks to (a). This tells us that $\ov\chi$ is $\si$-invariant if and only if  $\ov\Psi_{s}(\ov \chi)$ is ${h\si^*}$-invariant as claimed. \end{proof}

The next theorem shows that Condition $\Ap$ for $\GF$ implies the existence of an equivariant Jordan decomposition for $\Irr(\GF)$.

The above \Cref{Jordandec} gave us a commutative square  where horizontal arrows are bijections and vertical arrows are onto with fibers being $\Lin(\wGF/\GF)$-orbits. Once we show that $\ov\Jor(\GF)$ is also $\Jor(\GF)$ modulo $\wGF$-conjugation, \Cref{ZEsets} tells us that we need only to check an equality of stabilizers.

\begin{theorem}\label{thm66}
Assume that Condition $\Ap$ from \ref{Ainfty} holds for $\GF$. 
\begin{thmlist}
\item Then there exists a $ \calZ_F\rtimes E(\GF) $-equivariant bijection 
$$ \II J@{\JJ}: \Irr(\GF)\lra \Jor (\GF)$$ where the action of $ \calZ_F\rtimes E(\GF) $ on $ \Jor (\GF)$ is given by \Cref{lem3_7}(a).
\item \label{thm66_c}
If $\chi\in\Irr(\GF)$, $\ov \chi:=\Pi_\wGF(\chi)$ and $(s,\phi)\in \JJ(\chi)$, then 
\[(s, \restr \phi |{\Cent_\bH^\circ (s)^F} ) \in \ov \JJ(\ov \chi).\]
\end{thmlist}
\end{theorem}
 
\begin{proof} As announced above, we first show that $\ov\Jor(\GF)$ identifies with the set of $\cZ_F$-orbits in $\Jor(\GF)$. Let $(s,\phi)\in \Jor(\GF)$, define $C=\Cent_\bH^\circ(s)^F$, $\wt C=\Cent_\bH(s)^F$. Since $\omega_s\colon \cZ_F\to\Lin(A_\bH(s)^F)$ is onto, the $\cZ_F$-orbit of $(s,\phi)$ is $\{s\}\times \Lin(A_\bH(s)^F).\phi$ where the second term identifies with $\Pi_{\wt C}(\Irr(\restr\phi|{C}))$ by Lemma~\ref{bij_Pi} and \Cref{extCent} as recalled in the  proof of \Cref{Jordandec}. This gives our claim and therefore $\ov\JJ$ of \Cref{Jordandec} can be seen as an $E(\GF)$-equivariant bijection between the sets of $\cZ_F$-orbits in $\Irr(\GF)$ and $\Jor(\GF)$.
	This corresponds to Assumption (iii) of \Cref{ZEsets}.

Let us keep $\ov{(s,\phi)}\in\Jor(\GF)$, so that the $\HF$-orbit of $(s,\restr\phi|C)=(s,\ov\phi) $ belongs to $\ov\Jor(\GF)$. Let $\wt\chi\in\Irr(\wGF)$ such that $ (s,\ov\phi)\in\ov\JJ(\restr{\wt\chi}|\GF)$ and pick $\chi\in\Irr(\restr\wt\chi|\GF)$. This is the situation of \Cref{Jordandec}(c). We check that $\ov{(s,\phi)}$ and $\chi $ have same stabilizer in $\cZ_F$. Let $z\in\cZ_F$. Then $z.\ov{(s,\phi)}=\ov{(s,\phi)}$ means that ${(s,\wh\omega_{  s} (z)\phi)}={(s,\phi)}^h$ for some $h\in\HF$. This forces $h\in\Cent_{ \bH}(s)^F$ and the condition we get on $z$ is that $\wh\omega_{  s} (z)\in \Lin(A_\bH(s)^F)_\phi$. The latter group is the orthogonal of $A_\bH(s)^F_{\phi_0}$ for any $\phi_0\in\Irr(\restr\phi|{\Cent_{ \bH}^\circ(s)^F})$, again by combining Lemma~\ref{bij_Pi} and \Cref{extCent}. So $(\cZ_F)_{\ov{ (s,\phi)}}=\wh\omega_s \inv((A_\bH(s)^F_{\phi_0})^\perp)$. On the other hand \Cref{Jordandec}(c) tells us that $\wbG^F_\chi/\GF\Z(\wbG^F)$ identifies through $\omega_s$ and duality{} with the same orthogonal of $A_\bH(s)^F_{\phi_0}$. So we get our claim by \Cref{omegahat} on $\wh\omega_s $. This corresponds to Assumption (ii) of \Cref{ZEsets}.

Assumption (i) of \Cref{ZEsets} is ensured by our hypothesis for $\Irr(\GF)$ and by \Cref{lem3_7}(b) for $\Jor(\GF)$. We can now apply \Cref{ZEsets} as explained at the start of the chapter. This means that the $E(\GF)$-equivariant bijection $\ov\JJ$ of \Cref{Jordandec} lifts into a  $\cZ_F\rtimes E(\GF)$-equivariant bijection $\Irr(\GF)\to\Jor(\GF)$. This gives (a), while part (b) just expresses the fact that $\JJ$ lifts $\ov\JJ$.  
\end{proof}

As usual we obtain bijections $\cE(\GF,[s])\xleftarrow{\sim}\UCh(\Cent_\bH(s)^F)$ for every $s\in \bH^F_{\textrm{ss}}$. The following is a translation of the $\calZ_F\rtimes \EGF$-action on $\Jor(\GF)$ from \Cref{thm66}(b).
\begin{cor} \label{cor_Jordandec2}
Assume that Condition $\Ap$ holds for $(\bG,F)$. Let $\JJ$ be the map from \ref{thm66} and $(\ov \Psi_s)_{s\in \bHssF}$  the bijections from  \Cref{Jordandec}. For every  $s\in \ \bH_{\textrm{ss}}^F$ the map
$$\II Psis@{\Psi_s}:\cE(\GF,[s]) \lra \UCh(\Cent_\bH(s)^F)$$ 
given by $\chi\mapsto \phi$ whenever $(s,\phi)\in \JJ(\chi)$, is well-defined and  bijective. 
\begin{thmlist}
\item	The maps $(\Psi_s)_{s\in \bH^F_{\textrm{ss}}}$ satisfy
\begin{asslist}
		\item $\Psi_s(\chi)^{\si^*{}\inv} =\Psi_{\si^*(s)}(\chi^\si)$ for every $\sigma\in  E(\GF)$ with dual $\si^*\in E(\HF)$ as in (\ref{sigma*}) of \ref{duals}, 
		\item $\Psi_s(\chi ^g) = \Psi_s(\chi) \wh\omega _{s}(\ov g\inv) $ for every $g\in \wGF$ where $\ov g\in\cZ_F$ is the  $\GF\Z(\wGF)$-coset containing $g$, and
		\item $\Psi_s\big(\Irr(\Pi_{\wGF}(\chi))\big)=\Irr\big(\Cent_\bH(s)^F\mid \ov \Psi_s(\Pi_{\wGF}(\chi))\big)$, where $\ov \Psi_s$ is from \Cref{cor_Jordandec}.
\end{asslist}
\item If $\si\in E(\GF)$ and $s\in\bH^F_{\textrm{ss}}$ with $\si^*(s)=s^h$ for some $h\in\HF$, then $$\Psi_s(\cE(\GF,[s])^{\spannsi}) = \UCh (\CoHFs) ^{\spann<h\si^*>} .$$
\end{thmlist}
\end{cor}

\subsection{Extendibility in groups with connected center}\label{subsecGtilde}

We take here the opportunity to generalize a theorem from \cite{CS18B} showing that maximal extendibility holds \wrt $\wGF\unlhd \wGF\rtimes E(\GF)$. This was originally used for type $\tE_6$. The statement could be seen as a corollary of Theorem A, but it is more a precursor with a very direct proof. In fact it will be useful to us when dealing with type $\tD$ in characteristic 2 or with the graph automorphism of order 3 in $\tD_4$. Shintani descent arguments will be decisive again in the proofs of \Cref{prop82} and \Cref{thm77}.

\begin{theorem}\label{prop_ext_wG} Let $\wbG$, $F$, $E(\GF)$ be as in \ref{not}. Then every $\chi\in\Irr(\wGF)$ extends to its stabilizer in $\wGF \EGF$. 
\end{theorem}

\begin{proof} By \cite[Cor.~11.31]{Isa} it suffices to check that for any prime $\ell$ and any non-cyclic $\ell$-subgroup $E'$ of $E(\GF)$, every element of $\Irr(\wGF)^{E'}$ extends to $\wGF.E'$. Whenever  $E(\GF)$ itself is non-cyclic one has $F=F_p^{m}$ for some $m\geq 1$ and $E(\GF)=  \spann<\restr{F_p}|\GF >\times \Gamma$ for $\Gamma$ the group of diagram automorphisms $\ov\tau$ defined in \ref{not}. Then $E'$ must contain a non-trivial element of $\Gamma$, otherwise $E'$ would inject into the cyclic group $E(\GF)/\Gamma$. This forces $\ell =2$ or $3$ and $$E'=\spann<F_0,\gamma>,$$ 
where $F_0$ is the restriction to $\wGF$ of $F_p^{m'}$ for some positive divisor $m'$ of $m$, and $\gamma$ is the restriction to $\wGF$ of a graph automorphism $\ov\tau\colon \bG\to\bG$ of order $\ell$ suitably extended to $\wbG$.
	
	Now \cite[Cor.~3.9]{CS18B} gives our claim when $\gamma$ has order 2. Let us show how the proof of that theorem can be adapted in the case of a $\gamma$ of order $3$.
	
	The proof of \cite[Cor.~3.9]{CS18B}  is in two steps. The aim of the first step is to ensure the equality $|\Irr({\wt\bG}^F)^{\spann<F_0,\gamma >}|=|\Irr({\wt\bG}^{F_0})^{\spann<\gamma >}|$, called equation (3.8.1) in \cite[Thm~3.8]{CS18B}. This is established in the proof of \cite[Cor.~3.9]{CS18B} as an application of the equivariant Jordan decomposition of characters of $\wGF$ and it does not use that $\gamma$ has order 2. The second step, which is the main part of the proof of \cite[Thm~3.8]{CS18B}, needs the following adaptation.

	If $X$ is a finite group and $Y\subseteq X$, a union of $X$-conjugacy classes, then we denote by $\II CFXG@{\CF_X(Y)}$ the $\CC$-vector space of class functions on $Y$, namely the $\CC$-valued functions on $Y$ that are constant on $X$-conjugacy classes.
	
	We consider the semi-direct product $\wGF\rtimes\spa{F_0}$. Since $\gamma$ commutes with $F_p$, hence with $F$ and $F_0$, it induces linear maps $\gamma_1\colon \CF_{\wGF\spa{F_0}}(\wGF F_0)\lra \CF_{\wGF\spa{F_0}}(\wGF F_0)$ and $\gamma_0: \CF_{\wt\bG^{F_0}}({\wt\bG^{F_0}})\lra \CF_{\wt\bG^{F_0}}({\wt\bG^{F_0}}) $. We clearly have $ \mathrm{Tr}(\gamma_0)=|\Irr({\wt\bG^{F_0}})^{\spann<\gamma >}|$ by using irreducible characters as a basis. On the other hand, the Shintani descent (see \cite[Sect.~3.A]{CS18B}) induces a $\Bbb C$-linear ${\gamma }$-equivariant isomorphism 
	$\Sh_{F,F_0}: \CF_{\wGF\spa{F_0}}(\wGF F_0)\lra \CF_{\wt\bG^{F_0}}({\wt\bG^{F_0}})$ and we get a commutative diagram of bijections \begin{align}\label{comdiag_Shint}
		\xymatrix{
			\CF_{\wGF\spann<F_0>}(\wGF F_0) \ar[r]^{\gamma_1} \ar[d]_{\Sh_{F, F_0}}&\CF_{\wGF\spann<F_0>}(\wGF F_0)\ar[d]^{\Sh_{F, F_0}}\\
			\CF_{{\wt\bG^{F_0}}}({\wt\bG^{F_0}})\ar[r]_{\gamma_0} &\CF_{{\wt\bG^{F_0}}}({\wt\bG^{F_0}})
		}
	\end{align} ensuring that $\mathrm{Tr}(\gamma_1)=\mathrm{Tr}(\gamma_0)=|\Irr({\wt\bG^{F_0}})^{\spann<\gamma >}|$ by using $\Irr({\wt\bG^{F_0}})$ as a $\gamma$-stable basis of $\CF_{{\wt\bG^{F_0}}}({\wt\bG^{F_0}})$. Thanks to the equality recalled in the first step, we then get $$\mathrm{Tr}(\gamma_1) =
	|\Irr({\wt\bG}^F)^{\spann<F_0,\gamma >}|.$$
	For the inclusion $\wGF\unlhd\wbG^{F}\rtimes\spann<F_0>$, we denote by $\Lambda\ \colon\ \Irr(\wGF)^{\spann<F_0>}\ \longrightarrow\ \Irr({\wt\bG}^{F}\rtimes\spann<F_0>)$ the restriction of an {extension map} for $\Irr(\wGF)$, that is for any $\chi\in \Irr(\wGF)^{\spann<F_0>}$, $\Lambda(\chi)\in \Irr({\wt\bG}^{F}\rtimes\spann<F_0>)$ is such that $\restr\Lambda(\chi)|{\wGF}  =\chi$ (see end of \ref{not11}). Let $\EE'$ be the set of characters in $\Irr(\wGF)^{\spa{F_0,\gamma}}$ that have no extension to $\wGF\rtimes \spa{F_0,\gamma}=\wGF\rtimes E'$. By taking representatives, we can assume that $\Lambda$ is $\gamma$-equivariant on $\Irr({\wt\bG} ^F)^{\spa{F_0}}\setminus \EE'$. 
	Now the action of $\gamma$ on the $\Lambda(\chi)$'s can be described as follows, see also proof of \cite[Thm 3.8]{CS18B}. It permutes the  $\Lambda(\chi)$'s for $\chi\in\Irr(\wGF)^{\spann<F_0>}\setminus\Irr(\wGF)^{\spa{F_0,\gamma}}$ without fixing any. It
	fixes $\Lambda(\chi)$ whenever $\chi\in \Irr(\wGF)^{\spa{F_0,\gamma}}$ extends to $\wGF E'$, and it acts on the others, i.e. the elements of $\EE'$, by multiplying with a non-trivial linear character. As $\restr \Lambda(\chi)|{\wGF F_0}$ ($\chi\in \Irr(\wGF)^{\spa{F_0}}$) defines a basis of $\CF_{\wGF\spa{F_0}}(\wGF F_0)$ (see \cite[Prop. 3.1]{CS18B}) we can conclude that $\Tr(\gamma_1)$ is the sum of the integer $|\Irr(\wGF)^{\spa{F_0,\gamma}} \setminus \EE'|$ plus a sum of $|\EE'|$ complex roots of unity of order 3. 
	Since we know that $\Tr(\gamma_1)=|\Irr(\wGF)^{\spa{F_0,\gamma}} | $, we get $\EE'=\emptyset$ and hence our claim that every $\chi\in \Irr({\wt\bG} ^F)^{E'}$ extends to ${\wt\bG} ^F\rtimes E'$.
\end{proof}

\begin{rem}\label{rem_extun}
	Note that \Cref{sim_uni} above is an easy consequence of \Cref{prop_ext_wG} and (\ref{UchKK}).
\end{rem}

\section{Jordan decomposition for geometric series and labellings of stable $\wGF$-orbits}\label{sec3}

 We keep the notation introduced in \ref{not} and \ref{duals} but we assume from now on that $\bG=\tDlsc(\FF) $ for $l\geq 4$ and $\FF$ is the algebraic closure of the field with $p$ elements, where $p$ is an {\it odd} prime. We also take $F\colon \bG\to\bG$ with $F=F_{p^{m}}\ov\tau$ for $m\geq 1$ and $\tau$ a graph automorphism of order 1 or 2. We recall that we have fixed a regular embedding $\bG=[\wbG,\wbG]\leq \wbG$ in \ref{not} and  see $F$ also as Frobenius endomorphism of $\wbG$. 
 
 Note that $\GF/\Z(\GF)$ is always a simple group for which $\GF$ is a universal covering, see \cite[Thm 6.1.4]{GLS3}. We keep the notation $(\bH ,F)$ for a dual of $(\bG,F)$, that is $\bH=\tDlad(\FF) $, see \ref{duals}. We have $\bH\cong \bG/\Z(\bG)$ as algebraic groups since $p\neq2$, see \cite[Prop.~2.4.4]{DiMi2}. Also $\bH_0$, the simply connected covering of $\bH$ introduced in (\ref{piH0}) is here isomorphic to $\bG$.
 
The Jordan decomposition $\ov \JJ$ from \Cref{Jordandec} gives a labelling of $\wGF$-orbits on $\Irr(\GF)$ in terms of orbit sums of unipotent characters. 
	We give in \Cref{newJdec} a more precise version of this result in type $\tD$: For a given semisimple $F$-stable conjugacy class $\cC$ of $\bH$ and $s\in \cC^F$ we apply \Cref{cor_2Becht} below and obtain a finite group $\wc A(s)$ with $\Cent_\bH(s)=\Cent_\bH^\circ(s)\rtimes \wc A(s)$. If we denote by $\II{va}@{v_a}$ an element of $\wc A (s)$ representing $a\in \AHs_F$ as in \Cref{cor_CS22} and \Cref{noteps} respectively, we deduce from $\ov \JJ$ a bijection 
$$ \ov\Psi_s':\ocE(\GF,\cC) \lra \bigsqcup_{a\in \AHs_F} \oUCh(\Cent_\bH^\circ(s)^{v_a F}) .$$ 
Note that in \cite{DiMi21} the authors study similar maps 
$$ \cE(\GF,\cC) \lra \bigsqcup_{a\in \AHs_F}\UCh(\Cent_\bH(s)^{v_a F})$$
deduced from Jordan decomposition. 

Since the bijection $\ov \JJ$ is equivariant with respect to automorphisms of $\GF$, we obtain with \Cref{labelEGFFnull} a labelling of the $\si'$-stable $\wGF$-orbits contained in $\cE(\GF,\cC)$, whenever $\si'\in E(\bG)$ and $\cC\in \Cl_{\textrm{ss}}(\bH)^{\spa{\si',F}}$.
In \ref{ssec4C} we label the $\wGF$-orbits that are stable under a given non-cyclic subgroup of $E(\GF)$.

\medskip

Some more specific notation will be used throughout. 

\begin{notation}\label{not2_1} The group $\bG$ has a graph automorphism $\II{gamma}@{\gamma}\in E(\bG)$ of order 2 specified as sending the fundamental root $\al_2$ to $\al_1$ in the notation of \cite[Not.~3.3]{TypeD1}. It is the unique non trivial graph automorphism in $E(\bG)$ when $l>4$.
Let $\II{E2G}@{\uE(\bG)}:=\spa{F_p,\gamma}$ considered as a group of abstract automorphisms of $\bG$ and $\II{E2GF}@{\uE(\GF)}$ the corresponding subgroup of $\Aut(\GF)$. We define $\II{E2plus}@{\uE^+(\bG)}:=\{ F_p^i\gamma^j\mid i\geq 1,\ j=0,1  \}$ and assume $F\in \uE^+(\bG)$. We let $\uE(\bH)$, $\uE^+(\bH)$ and $\uE(\HF)$ be defined in the same fashion. Note that the duality $\si\mapsto\si^*$ of (\ref{sigma*}) sends the elements $F_p$ and $\gamma$ of $\uE(\bG)$ to elements of $\uE(\bH)$ defined in the same way, so we will routinely omit the ${}^*$.

For $\si\in\End(\bH)$ we denote by $\II{span}@{\protect{ \left\langle\sigma \right\rangle}^+}:=\{\sigma^i\mid i\geq 1 \}$
the set of positive powers of $\si$.

Recall that $p$ is an \textit{odd} prime number. Concerning the action of $\uE(\bG)$ on  $\Z(\bG)$, see \cite[Table 1.12.6]{GLS3}. The group $\Z(\bG)$ has order 4, being cyclic if and only if $l$ is odd. The action of $F_p$ on $\Z(\bG)$ is by raising elements to their $p$-th power hence trivial unless $l$ is odd and $p\equiv 3 \mod 4$. In all cases $\Z(\bG)^{\spann<\gamma>}=\Z(\bG)^{\spann<F_p,\gamma>}$ is a group of order 2 generated by an element denoted by $\III{h_0}$. We define $\III{Z_0}=\spa{h_0}$. The same applies to $\bH_0$ and we denote the generator of $\Z(\bH_0)^{\spann<\gamma>}=\Z(\bH_0)^{\spann<F_p,\gamma>}$ by $\III{h_0^{(\bH_0)}}$.
\end{notation}

\subsection{A lift of the component groups in $\tDlad(\FF)$}\label{sec1}
We recall here the result from \cite{CS22} exhibiting a lift $\wc A(s)$ of the component group of the centralizer of a semisimple element $s\in\bH=\tD_{l,\rm {\mathrm{ad}}}(\FF)$ that behaves remarkably well with respect to automorphisms. Its elements are very useful in the context of Jordan decomposition of characters of $\GF=\tDlsc(q)$. 

\begin{theorem} [{\cite[Thm~A]{CS22}}] \label{cor_2Becht}
	Let $\calC\in\Cl_{\textrm{ss}}(\bH)^F$ and $s\in\calC^{F}$.
	\begin{thmlist}
		\item \label{CS22_3A} There exists $\II{Acheck}@{\protect{\wc A(s)}}$ an $F$-stable abelian subgroup of  $\Cent_\bH(s)$ with \[\Cent_\bH(s)=\Cent_\bH^\circ(s)\rtimes \wc A(s).\]
		\item \label{CS22_3B} 
		If $\si'\in \uE(\bH)$ with $\si'(s)\in[s]_\HF$, then the above group $\wc A(s)$ is $\si$-stable for some $\si \in \Cent_{\HF \si'}(s)$.
		\item \label{CS22_3C}
		If $\si',F_0\in \uE(\bH)$ with $F\in\spa{F_0}^+$, $F_0(s)=s$ and $\si'(s)\in[s]_{\HFnull}$, then the group $\wc A(s)$ from \ref{CS22_3A} is 
		$\spa{\si,F_0}$-stable for some $\si \in \Cent_{\HFnull \si'}(s)$.
	\end{thmlist}
\end{theorem}

It is clear that automorphisms stabilizing $s$ also act on $\AHs$. If those automorphisms stabilize $\wc A(s)$ the actions on $\AHs$ and $\wc A(s)$ are closely linked. 

\begin{cor} \label{cor_CS22} In the situation of \Cref{cor_2Becht} let $\II{va}@{v'_a} \in \wc A(s)$ be the element corresponding to $a\in A_\bH(s):=\Cent_{\bH}(s)/\Cent_{\bH}^\circ(s)$. Then: 
	\begin{thmlist}
		\item $F(v'_a)=v'_{F(a)}$ for every $a\in \AHs$;
		\item $\si(v'_a)=v'_{\si(a)}$ for every $a\in \AHs$ with $\sigma$ as in \ref{cor_2Becht}(b);
		\item $F_0(v'_a)=v'_{F_0(a)}$ and $\si(v'_a)=v'_{\si(a)}$ for every $a\in \AHs$ in the situation of \ref{cor_2Becht}(c).
	\end{thmlist}
\end{cor}
\begin{proof}
Recall the natural epimorphism $\pi_\AHs:\Cent_\bH(s)\lra \AHs$ given by $\AHs=\Cent_\bH(s)/\Cent^\circ_\bH(s)$. The map $\pi_\AHs$ is accordingly equivariant with respect to all bijective endomorphisms of $\Cent_\bH(s)$ stabilizing $\Cent^\circ_\bH(s)$. 
Due to $\Cent_\bH(s)=\Cent_\bH^\circ(s)\rtimes \wc A(s)$, the restriction of $\pi_\AHs$ to $\wc A(s)$ is an isomorphism. We see that the endomorphisms $F$, $\si$ and $F_0$ from \Cref{cor_2Becht} stabilize $\wc A(s)$ and $\Cent^\circ_\bH(s)$. This implies that the isomorphism $\wc A(s)\lra \AHs$ given by $v'_a\mapsto a$ is equivariant for them.
\end{proof}

\subsection{Jordan decomposition and geometric series }\label{3C}
In the following we simultaneously study the rational series contained in a given geometric Lusztig series. 
We fix $\calC\in\Cl_{\textrm{ss}}(\bH)^F$ and $s\in\calC^F$ for the rest of the chapter.
\begin{notation}\label{noteps} Let $\wc A(s)$ be the $F$-stable abelian group with $\bGs=\bGos\rtimes \wc A(s)$ from Theorem~\ref{CS22_3A} and $\{\II{vb}@{v'_b\in \wc A(s)}\mid b\in \AHs\}$ be defined as in \Cref{cor_CS22}.
Recall $\Z(\bH_0)^{\spa{F_p,\gamma}} =\spa{h_0^{(\bH_0)}}$ from \Cref{not2_1}. Fix a section \[\II{sec}@{\omicron} : \Z(\bH_0)/\spannhHnull\lra \Z(\bH_0)\] with $\omicron(1)=1$.   

For  $a\in \II AsF@{\AsF:=\AHs/[\AHs,F]}$, we define below some $v_a\in \wc A(s)$ with $v_a[\AHs,F]=a$. 

When $[\AHs,F]=1$, then $a\in \AHs$ and $v'_a\in\wc A(s)$ from \Cref{cor_CS22} is defined, so we set $\II va@{v_a:=v'_{a}}$. 

When $[\AHs,F]\neq 1$, then $\omega_s([\AHs,F])=\spannhHnull$ and $\omicron(\omega_s(a))$ is well-defined. Then there exists a unique element $a_0\in \AHs$ with $a=a_0[\AHs,F]$ and 
$\omicron(\omega_s(a_0[\AHs,F]))=\omega_s(a_0)$. In this situation we
set $\II va@{v_a:=v'_{a_0}}\in\wc A(s)$. (This precise definition of $v_a$ using $\omicron$ is required only in \ref{sec4B}.)

Additionally, applying Lang's theorem \cite[Thm~21.7]{MT}, we fix some $\II ga@{g_a\in \bH}$ with ${g_a}^{-1}F(g_a)=v_a$. We set $\II sa@{s_a:=s^{g_a^{-1}}}$.
\end{notation}
Now (\ref{CFs_a}) gives at once the following.
\begin{lem}\label{lem38}
The various $[s_a]_\HF$ ($a\in \AHs_{F}$) are the distinct $\HF$-conjugacy classes in $\calC^F$. 
\end{lem}
 If $s,s'\in \cC^F$, then some inner automorphism of $\bH$ maps $s'$ to $s$. In the following we study how such an automorphism interacts with automorphisms of $\HF$ fixing $s$ or $s'$ respectively. 
Recall $F^y= y^{-1}F(y) F$ in the semi-direct product $\bH\rtimes\spann<F>$ whenever $y\in\bH$.
\begin{prop}\label{propiota}
	Let $E'$ be an abelian group of abstract group automorphisms of $\bH$ with $F\in E'$. We work in the semi-direct product $\bH\rtimes  {E'}$. Let $g\in \bH$ and $v:=g\inv F(g)$ (so that 
	$vF= F^{g}$ in $\bH\rtimes E'$). Let 
	$\iota: \bH\rtimes E'\lra \bH\rtimes E'$ be defined by $x\mapsto x^g$. Then
		\begin{thmlist}
			\item  $\iota(\HF)=\bH^{vF}$. 
			\item If $\tau\in E'$ with $\tau(v)=v$, then $\iota( \tau)\in \bH^{vF}\tau$ and $\iota\inv(\tau)\in\bH^{F} \tau$. 
			\item
			Let $e\geq 1$ be an integer such that $g\in \bH^{F^e}$. Let $E'_e\leq \Aut(\bH^{F^e})$ be the subgroup obtained from $E'$ by restriction to $\bH^{F^e}$. Then $\iota$ induces an automorphism of $\bH^{F^e}\rtimes E'_e$.
\end{thmlist}
\end{prop}
A map similar to $\iota$ was constructed in the proof of Proposition 5.3 of \cite{CS17A}, where $v$ was a specific element coming from a Sylow $d$-twist of $(\bG,F)$.
	\begin{proof} For the proof of (a) note $\iota(F)=g\inv F g=vF$ in $\bH\rtimes E'$ and therefore $\iota(\HF)=\bH^{vF}$. 

		 For part (b) we consider the element $\iota\inv(\tau)=\tau^{g^{-1}}$. Recall that $F$ and $\tau$ commute since they are elements of the abelian group $E'$. Note that $\tau^{-1}\tau^{g^{-1}} \in \bH$ and 
		 $F(g)=gv$ by the assumption on $g$ and $v$. 
		 We compute $F(\tau^{-1}\tau^{g^{-1}} )$ and get: 
		\begin{align*}
			F(\tau^{-1}\tau^{g^{-1}})&= 
			\tau^{-1}\, F(g)\, \tau\, F(g^{-1})= 
			\tau^{-1}\, gv \, \tau\, v^{-1}g^{-1}\\
			&= 
			\tau\inv gv\tau(v)\inv \tau g\inv=\tau^{-1}\, g \, \tau\, g^{-1}
			= \tau^{-1}\tau^{g^{-1}},
		\end{align*}
	(the three subgroups lemma could also be used). We get $\iota\inv(\tau)\in\HF \tau$. Combined with (a) this also implies $ \tau\in \bH^{vF}\iota( \tau)$, so we get part (b). 

		Part (c) follows from the fact that since $g\in \bH^{F^e}$ then $\iota$ induces an inner automorphism of $\bH^{F^e}\rtimes E'_e$.
	\end{proof}

\begin{prop}\label{78}
Let $a\in \AsF$, $g_a$, $v_a$, $s_a$, as in \Cref{noteps}. Then there exists  $e\geq 1$ with $g_a\in \bH^{F^e}$ and the isomorphism 
\begin{align*}\label{defiotaa}
\II iotaa@{\iota_a}: \bH^{F^{e}}\rtimes \uE(\bH^{F^e}) \lra &\bH^{F^{e}}\rtimes \uE(\bH^{F^e}) \text{ given by }x\mapsto x^{g_a}
\end{align*}
has the following properties : 
\begin{thmlist}
	\item $s_a=\iota_a^{-1}(s)$, $\iota_a(\HF)=\bH^{v_a F}$, $\iota_a(\Cent_\bH^\circ(s_a)^F)= \Cent_\bH^\circ(s)^{v_a F}$ and $\iota_a(\Cent_\bH(s_a)^F)= \Cent_\bH(s)^{v_a F}$.
	\item $\iota_a^{-1}$ induces a bijection
	\begin{align*}
	\II iotaa*@{\iota_a^*}: \oUCh(\Cent_\bH^\circ(s)^{v_a F}) \lra \oUCh(\Cent_\bH^\circ(s_a)^{F}).
	\end{align*}
	\end{thmlist} 
\end{prop}
\begin{proof} The existence of $e$ with $g_a\in \bH^{F^e}$ follows from $\bH=\bigcup_{n\geq 1}\bH^{F^n}$. Conjugation with $g_a$ then defines an inner automorphism of the finite group $\bH^{F^e}\rtimes \uE(\bH^{F^e})$. The equality follows from the definition of $s_a$ and $g_a$, see \Cref{noteps}. The other two equalities in (a) are implied by $F^{g_a}=v_a F$ in $\bH^{F^{e}}  \uE(\bH^{F^e})$. 
Note that the set of unipotent characters of 
$\Cent_\bH^\circ(s)^{v_a F}$ is sent to those of $\Cent_\bH^\circ(s_a)^{F}$ by composing with $\iota_a$. Via $\iota_a$ the action of $\Cent_\bH(s)^{v_aF}$ corresponds to the action of $\Cent_\bH(s_a)^F$ implying the statement in (b). 
\end{proof}
We combine the bijections $\iota_a^*$ ($a\in \AsF$) of \Cref{78} with the character correspondences $\ov\Psi_{s_a}$ from \Cref{cor_Jordandec}. This variation of the Jordan decomposition takes into account the different rational series at the same time and parametrizes all $\wGF$-orbits in a geometric series simultaneously.
	\begin{prop}[Jordan decomposition of orbit sums in a geometric series]
		\label{newJdec}
		With the above notation there exists a bijection 
\[\II Psi'overlines @{\ov\Psi'_s}:\ocE(\GF,\cC) \lra \bigsqcup_{a\in \AHs_F}\oUCh(\Cent_\bH^\circ(s)^{v_a F}),\] 
such that $|\AshochF|=|\Irr(\ov \chi)|\, \cdot \, |\Irr(\ov\Psi'_s(\ov\chi))|$ for every $\ov \chi\in\ov\cE(\GF,\cC)$. 
	\end{prop}
	\begin{proof} Using \Cref{lem38} and (\ref{CFs_a}) we get 
		$$\ov\cE(\GF, \cC)= \bigsqcup_{a \in \AHs_F} \ov\cE(\GF,[s_a]).$$
		For $a\in \AHs_F$ let 
		$\ov\Psi_{s_a}: \ov\cE(\GF,[s_a]) \ra \oUCh(\Cent^\circ_\bH(s_a)^{ F})$ and $\iota^*_a: \oUCh(\Cent^\circ_\bH(s)^{v_a F}) \ra
		\oUCh(\Cent^\circ_\bH(s_a)^{F})$ be the bijections from \Cref{cor_Jordandec} and  \Cref{78}(b), respectively. 
		We define $\ov\Psi'_s$ as $(\iota_a^*)\inv\circ \ov\Psi_{s_a}$ on $\ocE(\GF,[s_a])$. This gives a bijection between the given character sets.

		Let $\ov \chi\in\ov\cE(\GF,\cC)$ and $a\in \AsF$ such that $\ov \chi\in\ocE(\GF,[s_a])$. According to \Cref{cor_Jordandec}(a) the map $\ov\Psi_{s_a}$ satisfies 
		\[ |\AHs^F|= | \Irr(\ov \chi) | \ \cdot\ |\Irr(\ov\Psi_{s_a}(\ov \chi))|.\]
		This ensures $|\AHs^F|=|\Irr(\ov \chi)|\, \cdot \, |\Irr(\ov\Psi'_s(\ov\chi))|$.
	\end{proof}
	The above bijection induces a non-injective map between the irreducible unipotent characters of $\bGos^{v_aF}$ ($a\in \AHs_F$) and characters in $\ov\cE(\GF,\cC)$. 
\begin{rem}[\textbf{Action of $\AshochF$ on $\UCh(\Cent_\bH^\circ(s)^{v_a F})$}] \label{rem3_6}
	\Cref{cor_CS22} implies $[v_a F,v'_b] =1$  for $a\in \AHs_F$ and $b\in \AshochF$, see also the definitions of $v_a$ and $v'_b$ in \Cref{noteps}.
	We let  $b\in \AHs^F$ act on $ \UCh(\Cent_\bH^\circ(s)^{v_a F})$  by conjugating the character with $v'_b$. This defines an action of  $\AHs^F$  on $ \UCh(\Cent_\bH^\circ(s)^{v_a F})$ and it coincides with the action of $\Cent_\bH(s)^{v_aF}/\Cent^\circ_\bH(s)^{v_aF}$ on $\UCh(\Cent_\bH^\circ(s)^{v_a F})$ since $\Cent_\bH(s)^{v_aF}=\Cent_\bH(s)^{v_aF} \rtimes \wc A(s)^{v_aF}= \Cent_\bH(s)^{v_aF} \rtimes \wc A(s)^{F}$, remembering that $v_a$ belongs to the abelian group $\wc A(s)$.
 \end{rem}

\begin{cor}[Geometric Jordan decomposition of orbit sums] \label{cor3_13}
There exists a surjective map
$$\II GammasF@{\Gamma_{s,F}}:
\bigsqcup_{a\in \AHs_F}\UCh(\Cent_\bH^\circ(s)^{v_a F})
\lra \ocE(\GF,\cC),$$ 
such that 
		 \begin{itemize} 
			\item $\Gamma_{s,F}^{-1}(\ov \chi)$ is an $\AshochF$-orbit for every $\ov \chi\in\ocE(\GF,\cC)$, and 
			\item $|\Irr(\Gamma_{s,F}(\phi))|= |(\AshochF)_{\phi}|$ for every $\phi\in \UCh(\Cent_\bH^\circ(s)^{v_a F})$.
		 \end{itemize} 
	\end{cor}
	\begin{proof}
		For each $a\in \AHs_F$ let $\II{Catilde}@{\protect{\wt C}_a}:=\bGs^{ v_aF}$ and $\II{Ca}@{\protect{ C}_a} :=\bGos^{ v_aF}$. The map $\Pi_{\wt C_a}$ from \Cref{def31} induces a surjective map 
		$$ \UCh(C_{a})\lra \oUCh(C_{a})$$ 
		by the definition of $\oUCh(C_a)$.
		Note that $\UCh(C_a)$ is $\wt C_a$-stable. Then $\Pi_{\wt C_a}$ has the property that for every $\ov \phi\in \oUCh(C_a)$ the set $\Pi_{\wt C}\inv (\ov \phi)$ is a $\wt C_a$-orbit in $\UCh(C_{a})$. Using the universal property of coproducts we can set $\Gamma_{s,F}:=\bigsqcup_{a\in \AHs_F}  
		(\ov \Psi'_s)^{-1} \circ \restr \Pi_{\wt C_a} |{\UCh(C_{a})}$.
		
By the properties of $\ov \Psi'_{s}$ from \Cref{newJdec} we see $|\AHs^F|= |\Irr(\ov \chi)|\ \cdot\ |\Irr(\Pi_{\wt C_a}(\phi))|$ for every $\ov \chi\in \ov\cE(\GF,\cC) $ and $\phi\in\Irr(\ov \Psi_s'(\ov \chi))$, where $a\in \AHs_F$ is chosen such that  $\phi\in \UCh(C_a)$.
Together with $|\Irr(\Pi_{\wt C_a}(\phi))|= |\AHs^F: (\AHs^F)_\phi|$ we obtain $|\Irr(\ov \chi)|=|(\AHs^F)_\phi|$.
\end{proof}
\subsection{Labelling of $\ocE(\GF,\cC)^{\spann<F_0>}$}\label{sec3D}
We keep $\calC\in \Cl_{\textrm{ss}}(\bH)^F$ and we introduce a labelling of $\si'_{\GF}$-stable $\wGF$-orbits in $\cE(\GF,\calC)$ for an automorphism $\si'_\GF \in \uE(\bG)$, see \Cref{labelEGFFnull}. In the applications this map will be a Frobenius endomorphism of $\bG$ denoted by $F_0$. By (\ref{sigma*}) in \ref{duals}, if $\cE(\GF,\calC)$ contains a $\si'_{\GF}$-stable $\wGF$-orbit, then $\cC$ contains a $\spa{\si'_\HF}$-stable rational $\HF$-class for some $\si'_\HF\in\HF.(\si'_\GF)^*$ in the notation of (\ref{sigma*}).

In the following we  parametrize $\ocE(\GF,\cC)^{\spa{\si'_\GF}}$ via a refinement of the map from \Cref{cor3_13}. We assume that $\cC$ contains a $(\si'_\GF)^*$-stable rational class, as otherwise $\ocE(\GF,\cC)^{\spa{\si'_\GF}}=\emptyset$.  
We adapt the choice of $s\in \cC^F$ and the elements $\{v'_b\mid b\in \AHs\}$ from \Cref{noteps} in connection with $(\si'_\GF)^*$  by  assuming additionally that $[s]_\HF$ is $(\si'_\GF)^*$-stable. Recall that the elements of $E(\bH)$ act on $\bH_0$.

\begin{notation} \label{not313}
Let $\si'_\HF \in \uE(\bH)$ acting on $\HF$ with even order, and 
let $\si'_\GF = (\si'_\HF)^*\in \uE(\bG)$, see (\ref{sigma*}). 
Moreover let $\cC\in \Cl_{\textrm{ss}}(\bH)^F$ and $s\in \cC^F$  fixed at the beginning of \ref{3C}. 
Assume:
\begin{asslist}
	\item  $\si'_\HF(s)\in [s]_\HF$; and
	\item  $\si'_\HF$ or $F$ acts trivially on $B(s)$. (Recall $B(s)=\omega_s(\AHs)$.)
\end{asslist}
Let $\wc A(s)$ and $\si\in\Cent_{\HF \si'_\HF}(s)$ associated to $s$ via \Cref{cor_2Becht}(b). In the following let  $\{v'_b\in \wc A(s) \mid b\in \AHs\}$ be defined as in \Cref{noteps}. 
\end{notation}
 Whenever $F$ or $\si'_\GF$ acts trivially on $\Z(\bG)$ or when $|\AHs|=2$, then $F$ or $\si'_\HF$ acts trivially on $\Z(\bH_0)$. Hence, Assumption (ii) above is satisfied in this case. 
  
\begin{lem} Keep the notation from \ref{not313}.
	 \begin{enumerate} 
	 \item  $\si$ or $F$ acts trivially on $\AHs$. 
	\item \label{lem3_15}  $[\si, F]=1$, $\si(v_a)=v_a$,  and  $\delta(v'_b)= v'_{\delta (b)}$ for every $a\in ({\AHs_F})^{\spa\si}$, $ b\in \AHs$ and $\delta\in\spa{\si,F}$. 
	\item $[\si(s_a)]_\HF=[s_{\si(a)}]_\HF$ for every $a\in \AHs_F$. 
	\item \label{lem3_15d} Let $\iota_a$ be the isomorphism from \Cref{78} and $\II sia@{\si_a:=\iota_a\inv(\si)}$. Then $\si_a\in\HF\si$. 
\end{enumerate} 
\end{lem}
\begin{proof} 
	Because of $\si\in\Cent_{\HF \si'_\HF}(s)$ the endomorphisms of $\bH_0$ given by $\si'_\HF$ and $\si$ act in the same way on $\Z(\bH_0)$.  
	Because of $F(s)=\si(s)=s$, the map $\omega_s:\AHs\lra \Z(\bH_0)$ is $\spa{F,\si}$-equivariant and has $B(s)$ as image. Hence Assumption \ref{not313}(ii) implies that $\si$ or $F$ acts trivially on $\AHs$. This ensures part (a). In case $B(s)=\Z(\bH_0)$ this also implies that one of them acts trivially on $\Z(\bH_0)$.
	
	For part (b) we apply \Cref{cor_CS22} and obtain $\delta(v'_b)=v'_b$ for every $b\in \AHs$ and $\delta\in\spa{F,\si}$.
	By construction $\si\in \Cent_{\HF \si'_\HF}(s)$ and hence $[\si,F]=1$.  
	
	Let $a\in ({\AHs_F})^{\spa\si}$ and $a_0\in \AHs$ with $a_0 [\AHs, F]=a$ and $v_a=v'_{a_0}$. If $F$ acts trivially on $\AHs$, then $a_0=a$ and hence  $\si(a_0)=\si(a)=a=a_0$. Otherwise $\si$ acts trivially on $\AHs$ by part (a) and hence fixes $\wc A(s)$. In all cases $v_a$ is fixed by $\si$, since $a\in (\AHs_F)^{\spa\si}$.

	For part (c), let $a\in \AHs_F$ (not necessarily $\si$-fixed). Then we observe $\si(s_a)=\si(\tw{{g_a}}s)= \tw{{\si(g_a)}}s$. The equality $$\si(g_a)^{-1} F(\si(g_a)) \Cent_\bH^\circ(s)=\si((g_a)^{-1} F(g_a)) \Cent_\bH^\circ(s)= \si(v_a)\Cent_\bH^\circ(s)$$ shows that $[\si(s_a)]_\HF$, the $\HF$-conjugacy class containing $\si(s_a)=\tw{{\si(g_a)}}s$, corresponds to $\si(a)$ via the parametrization given in (\ref{CFs_a}). Recall that $v_{\si(a)}=\si(v_a)$  by \Cref{cor_CS22}. 

	Part (b) implies $\si(v_a)=v_a$ for every $a\in  (\AsF)^{\spa{\si}}$. Recall that $\iota_a$ from \Cref{78} is given by conjugation with $g_a$ and we have $g_a^{-1}F(g_a)=v_a$ by its definition in \Cref{noteps}. Hence $\iota_a(\HF \si)=\bH^{v_aF} \si$ and $\si_a=\iota_a\inv(\si)\in\HF\si$ according to \Cref{propiota}(b).  This ensures part (d).
\end{proof}

	\begin{lem} \label{newJdec_si}
Let $s\in\cC^F$ and $\si$ be as in Notations~\ref{noteps} and \ref{not313}. Let $C_a:=\Cent_\bH(s)^{v_aF}$ for $a\in \AHs_F$. Using the notation of (\ref{sigma*}), if $\si'_\GF =(\si'_\HF)^*\in \uE(\bG)$ and hence $\si\in \HF(\si'_\GF)^*$, then the map 
		$\ov \Psi'_s:\ov\cE(\GF,\cC)\lra \bigsqcup_{a\in \AsF} \oUCh(C_a) $
		from \Cref{newJdec} satisfies 
\[ \ov \Psi'_s(\chi^{\si'_\GF})=\ov \Psi'_s(\chi)^{\si} \text{ for every }\chi\in \bigsqcup_{a\in \Cent_{\AsF}(\si)} \ocE(\GF,[s_a]_{\HF}), \]
		in particular 
\[\ov\Psi'_s(\ocE(\GF,(s))^{\spa {\si'_\GF} } ) 
    = \bigsqcup_{a\in \Cent_{\AsF} (\si) } \oUCh( C_a )^{\spa{\si} }.\]
	\end{lem}
	\begin{proof} Recall the definition of $\ov \Psi'_s(\chi)$ for $\chi\in\ocE(\GF,\cC)$: 
    There exists some $a\in\AHs_F$ such that $\chi\in\ocE(\GF,[s_a]_\HF) $. Then $\ov\Psi_s'(\chi)=(\iota_a^*)^{-1}(\ov\Psi_{s_a}(\chi))$, where  $\ov\Psi_{s_a}$ is the map from \Cref{cor_Jordandec} and $\iota_a^*$ is the bijection from \Cref{78}(b). 

	If $\II sigmaa@{\si_a}:=\iota_a^{-1}( \si)$ for $a\in \Cent_{\AsF}(\si)$, we observe that $\si_a(s_a)=s_a$ and hence $\si_a$ acts on $\Cent_\bH(s_a)$ and $\Cent^\circ_\bH(s_a)$. As $\si_a\in \HF\si$ we see that $\si_a$ commutes with $F$ and hence acts on the finite groups $\Cent_\bH(s_a)^F$ and $\Cent^\circ_\bH(s_a)^F$, as well as on the character sets $\UCh(\Cent^\circ_\bH(s_a)^F)$ and $\oUCh(\Cent^\circ_\bH(s_a)^F)$. We notice that $\si_a$ is a composition of $\si'_\HF$ and an inner automorphism of $\HF$ by its construction. Recall $\si'_\GF=(\si'_\HF)^*$.
    The map $\ov \Psi_{s_a}$ from \Cref{cor_Jordandec}
    satisfies $\ov \Psi_{s_a}(\chi^{\si'_\GF})=\ov\Psi_{s_a}(\chi)^{\si_a}$ and hence 
    \[\ov \Psi'_s(\chi^{\si'_\GF})=(\iota_a^*)^{-1}(\ov\Psi_{s_a}(\chi)^{\si_a})=(\iota_a^*)^{-1}(\ov\Psi_{s_a}(\chi))^{\si}=\ov\Psi'_s(\chi)^\si.\] 
    This leads to 
 \[ \ov \Psi_{s_a}(\ocE(\GF,[s_a]_{\HF})^{\spa{\si}})=\oUCh(\Cent_\bH^\circ(s_a)^F)^{\spa{\si_a}} .\]
From the definition of $\iota_a$ and $\si_a$ we observe $$\iota_a^*(\psi^\si)=(\iota_a^*(\psi))^{\si_a}$$ for every $\psi\in\ov\UCh(\Cent^\circ_\bH(s)^{v_aF})$. We now have  
\[\iota_a^*( \oUCh(\Cent_\bH^\circ(s)^{\epsFa F})^{\spa \si} ) = \oUCh(\Cent_\bH^\circ(s_a)^F)^{\spa{\si_a}}.\]
Note that by (\ref{rem_ntt}) any $\si'_\GF$-invariant character contained in $\ocE(\GF, \cC)$ lies in a rational Lusztig series $\ocE(\GF,[s_a]_\HF)$ with $a\in \Cent_{\AsF}(\si)$. We then get the stated equality 
\[\ov\Psi'_s(\ocE(\GF,(s))^{\spa {\si'_\GF} } ) 
		= \bigsqcup_{a\in \Cent_{\AsF} (\si) } 
		\oUCh( \Cent^\circ_\bH(s)^{v_aF})^{\spa{\si} }.
        \qedhere\]
\end{proof}
	By its definition for every $a\in \AsF$ the set $\oUCh(C_a)$ ($a\in \AsF$) is defined via the action of $\wt C_a:=\Cent_\bH(s)^{v_a F}$ on $C_a:=\Cent^\circ_\bH(s)^{v_a F}$ and hence $\UCh(C_a)$. 
	Let $\Pi_{\wt C_a}: \Irr(C_a) \lra \Char(C_a)$ be the map from \Cref{defPI_allg} and recall $\oUCh(C_a)= \Pi_{\wt C_a}(\UCh(C_a))$. By assumption either $F$ or $\si$ acts trivially on $\AHs$. 
If $a\in \Cent_{\AsF}(\si ) $ and $b\in \AHs^F$, then $\si(v_a)=v_a$ by \Cref{cor_CS22}. We see that $v_a F$ and $v'_b\si$ commute and then the set $\UCh( C_a)^{\spa{v'_b\si}}$ is well-defined. 

We now apply the considerations of \Cref{defPI_allg}(a) and (b).
As before we denote by $ \bigsqcup_{b\in \AHs^F}\UCh(C_a)^{\spa{v'_b \si}} $ a disjoint union, see Definition~\ref{multi}. A character $\phi\in \UCh(C_a)$ is contained in $\UCh(C_a)^{\spa{v'_b \si}}$ if it is $v'_b\si$-invariant. Then it is also contained in $\UCh(C_a)^{\spa{v'_{bb'} \si}}$ for every $b'\in \AHs^F_\phi$. 
	If $\phi\in \bigcup_{b\in \AHs^F}\UCh(C_a)^{\spa{v'_b \si}} $, then $\phi$ occurs $|\AHs^F_\phi|$-times in $ \bigsqcup_{b\in \AHs^F}\UCh(C_a)^{\spa{v'_b \si}} $. 
	The characters of $\Pi_{\wt C_a}\inv (\Pi_{\wt C_a}(\phi))$ form an $\AHs^F$-orbit and  $\Pi_{\wt C_a}\inv (\Pi_{\wt C_a}(\phi))$ defines  a subset of the disjoint union $\bigsqcup_{b\in \AHs^F}\UCh(C_a)^{\spa{v'_b \si}} $ with cardinality $|\AshochF |$. 
	\begin{lem}\label{defPI}
    Let $a\in \Cent_{\AsF}(\si)$ and set $\wt C_a:=\Cent_\bH(s)^{v_aF}$, $C_a:=\Cent^\circ_\bH(s)^{ v_aF}$. 
    \begin{thmlist}
	\item Let $\Pi_{\wt C_a}: \Irr(C_a) \lra \Char(C_a)$ be the map from \Cref{defPI_allg}. Then 
	\[\Pi_{\wt C_a}^{-1}(\oUCh(C_a)^{\spa{\si}})= \bigcup_{b\in \AHs^F }\UCh(C_a)^{\spa{v'_b \si}} \subseteq \UCh(C_a).\] 
	\item Let now $$\II PiUChCa@{\Pi^{\UCh}_{\wt C_a, \si}}: \bigsqcup_{b\in \AHs^F}\UCh(C_a)^{\spa{v'_b \si}} \lra \oUCh(C_a)^{\spa{\si}}$$ be associated to $\Pi_{\wt C_a}$ as in Lemma \ref{lem72b}. Then $|\Pi^{\UCh}_{\wt C_a,\si}{}^{-1}(\ov \chi)|=|\AHs^F|$ for every $\ov \chi\in \oUCh(C_a)^{\spa{\si}}$.
		\end{thmlist}
	\end{lem}
	\begin{proof}
		The two statements follow from \Cref{defPI_allg}.
	\end{proof}
	So we study in the following 
	$\bigsqcup_{a\in \AHs_F}\oUCh( C_a)^{\spa{\si}}$ using the maps $\Pi^{\UCh}_{\wt C_a,\si}$ from above. 
	\begin{notation} \label{lem717}	
Whenever $s\in \bH_{\textrm{ss}}$, $F'$ is a Frobenius endomorphisms of $\bH$ with $F'(s)=s$ and $\kappa$ is an endomorphism of $\bH$ inducing an automorphism of $\Cent_\bH^\circ(s)^{F'}$, we set 
\[\II{UsFkappa}@{\rmUU(s, F', \kappa)}:=\UCh(\Cent_\bH^\circ (s)^{F'})^{\spa \kappa}\]
For $\II A0@{A_0:=\Cent_{\AsF}(\si) }$ we define
\begin{align*}
\II UsFsi@{\UU(s,F,\si)}:=\bigsqcup _{\substack{a \in A_0 \\ b\in \AHs^F}} 
\UCh( \Cent_{ \bH}^\circ (s)^{v_aF})^{\spa{v'_b\si}}=
  \bigsqcup _{\substack{a \in A_0 \\ b\in \AHs^F}} 
			\rmUU(s,v_aF,v'_b\si ).
\end{align*}
\end{notation}
Note that $\UU(s,F,\si)$ is a disjoint union where a single character occurs up to $|\AHs^F|$-times among the various summands of $\UU(s,F,\si)$. \Cref{defPI_allg} and \Cref{defPI} allow us to introduce a map between the unipotent characters and the $\si$-invariant orbit sums. 
	\begin{lem}\label{Pi_sdef}
		There is a natural map 
		\begin{align*}
			\II Pis@{\Pi_s}: \UU (s,F,\si) \lra 
			\bigsqcup_{a\in A_0}\oUCh(C_a)^{\spa{\si}}, 
		\end{align*}
such that whenever $\phi \in \oUCh(C_a)^{\spa{\si}}$ the characters in $\Pi_{s}\inv (\phi)$ form a $\wt C_a$-orbit and 
\begin{align*}|\Pi_{s}\inv (\phi)| &=|\AHs^F| \text{ for every }
\phi\in \bigsqcup_{a\in A_0} \oUCh(C_a) ^{\spa{\si}},
\end{align*}
where $\Pi_{s}\inv (\phi)$ is seen as a subset of a disjoint union, see \Cref{multi}. 
\end{lem}
\begin{proof}
Let $a\in\AsF$ and $\wt C_a:=\Cent_\bH(s)^{v_aF}$. We define the map $\Pi_s$ on $\bigsqcup_{b\in \AHs^F} \UCh( C_a)^{\spa{v_b' \si}}$ to be the map 
\[ \Pi^{\UCh}_{\wt C_a,\si}: \bigsqcup_{b\in \AshochF} \UCh( C_a)^{\spa{v_b' \si}} \lra \oUCh(C_a)^{\spa \si} \] 
from Lemma~\ref{defPI}. Then $|\Pi_{s}\inv (\phi)|$ is part of a disjoint union and satisfies $|\Pi_{s}\inv (\phi)|=|\AHs^F|$ as $\Pi_{ \wt C_a,\si}$ has this property and $\UCh(C_a)$ is $\wt C_a$-stable.
	\end{proof}
	Combining the maps $\Pi_s$ and $\ov \Psi'_s$ we obtain a labelling of $\ocE(\GF,\cC)^{\spa{\si'}}$ for $\si '\in E(\GF)$. 
	\begin{thmbox}[Labelling of $\ov \cE(\GF,\cC)^{\spa{\si'_\GF}}$]\label{labelEGFFnull} Recall $\si'_\HF\in \uE(\bH)$, $\si'_\GF= (\si'_\HF)^*\in \uE(\bG)$ and $\si\in \Cent_{ \bH \uE(\bH)}(s)\cap \HF\si '_\HF$ as in \Cref{not313}. For the disjoint union $\UU(s,F,\si)$ from \ref{lem717} there exists a surjective map 
		\begin{equation*}
			\II GammasFsi@{\Gamma_{s,F,\si}}:\UU(s,F,\si)
			\lra \ocE(\GF,\cC)^{\spa{\si'_\GF}},
		\end{equation*} 
		such that for every $\ov \chi\in \ov \cE(\GF,(s))^{\spa{\si'_\GF}}$ and $\phi\in \UU(s,F,\si)$ we have: 
		\begin{thmlist}
	\item \label{eq_samepreimagesize} The elements of $\Gamma_{s,F,\si}^{-1}(\ov \chi)$ form an $\AHs^F$-orbit;
			\item $|\Gamma_{s,F,\si}^{-1}(\ov \chi)|=|\AHs^F| $ if $\Gamma_{s,F,\si}^{-1}(\ov \chi)$ is considered as a subset of a disjoint union, and
\item \label{eq_labelEGFFnull} $|\Irr(\Gamma_{s,F,\si}(\phi))|= |(\AHs^F)_\phi|$.
		\end{thmlist}
\end{thmbox}
\begin{proof} 
		Recall $C_a=\Cent_\bH^\circ(s)^{v_a F}$. Let $\ov\Psi'_s: 	\ocE(\GF,\cC) \lra \bigsqcup_{a\in \AHs_F}\oUCh(C_a)$ be the bijection from \Cref{newJdec} and $\Pi_s: \UU(s,F,\si)\lra \bigsqcup_{a\in A_0} \oUCh(C_a)^{\spa{\si}}$ the map from \Cref{Pi_sdef}. We define $\Gamma_{s,F,\si}$ by $\phi \mapsto {\ov\Psi_s' }\inv(\Pi_s(\phi)) $ for every $\phi\in\UU(s,F,\si)$. 
        Note that $\Pi_s(\UU(s,F,\si))=\bigsqcup_{a\in A_0} \oUCh(C_a)^{\spa{\si}}$ according to \Cref{Pi_sdef} and $\Psi'_s(\ocE(\GF,\cC)^{\spa{\si'_\GF}})=\bigsqcup_{a\in A_0} \oUCh(C_a)^{\spa{\si}}$ according to \Cref{newJdec_si}. Hence $\Gamma_{s,F,\si}$ is surjective. 

        Recall that $\ov\Psi'_s$ is a bijection, see \Cref{newJdec}. Let $\ov \chi \in \ocE(\GF,\cC)^{\spa{\si'_\GF}}$. Then $\Pi_s^{-1}(\ov\Psi'_s(\ov\chi))= \Gamma_{s,F,\si}^{-1}(\ov\chi)$ is a $\wt C_a$-orbit in $\UCh(C_a)$ for some $a\in A_0$, see also \Cref{Pi_sdef}. This ensures part (a) as the action of $\wt C_a$ coincides with the action of $\AHs^F$.

        According to \Cref{Pi_sdef}, we can consider $\Pi_s^{-1}(\ov\Psi'_s(\ov\chi))$ as a subset of a disjoint union with cardinality $|\AHs^F|$. This ensures part (b).
        
        Let $\phi\in \UU(s,F,\si)$ and $\ov \phi:=\Pi_s(\phi)$. Then 
	$ | \Irr(\ov\phi)|= \frac{|\AHs^F|}{|\AHs^F_\phi|}$, as $\Irr(\ov\phi)$ is the $\AHs^F$-orbit containing $\phi$. 
    Then
		\[|\Irr({\ov\Psi'_s}\inv (\ov \phi))| = \frac{|\AHs^F|}{|\Irr(\ov\phi)|} \] by Proposition \ref{Jdnuconst} and hence $|\Irr(\Gamma_{s,F,\si}(\phi))|=|\AHs^F_\phi|$. This proves part (c).
	\end{proof} 
	We later apply this map in order to determine the cardinality of 
	$\cE(\GF,\cC)^{\spa{\si'_\GF,\wGF}}$.
	\begin{cor}\label{lem_whU}
Whenever $F'$ is a Frobenius endomorphisms of $\bH$ with $F'(s)=s$ and $\kappa$ is an endomorphism of $\bH$ inducing an automorphism of $\Cent_\bH^\circ(s)^{F'}$, we set 
 \[\II{UhatsFkappa}@{\protect{\wh \rmUU}(s, F', \kappa)}:=\{ \phi \in \rmUU(s,F',\kappa)=\UCh(\Cent_\bH^\circ (s)^{F'})^{\spa \kappa}\mid  \AHs^{F'}_\phi=1\}\]
 and   
 \begin{align*}	\II UhatsFsi@{ \protect{\wh {\UU}(s,F,\si)}}:= 
        \bigsqcup_{\substack{a \in A_0 \\ b\in \AHs^F}} \wh \rmUU(s,v_a F, v_b' \si)=
        \{\phi\in \UU(s,F,\si) \mid \AHs^F_\phi =1 \} ,
		\end{align*} see \Cref{lem717} and
	\Cref{rem3_6} about the action of $\AHs^F$ on the various $\Cent_{ \bH}^\circ (s)^{v_aF}$ ($a\in \AHs_{F}$).
		Then the map $\Gamma_{s,F,\si}$ from \Cref{labelEGFFnull} satisfies 
		$$\Gamma_{s,F,\si}^{-1}\left (\ocE(\GF,\cC)^{\spa{\si'_\GF}}\cap \Irr(\GF)\right )=\wh {\UU}(s,F,\si)$$
		and $| \wh {\UU}(s,F,\si)|= |\AHs^F| \cdot |\ocE(\GF,\cC)^{\spa{\si'_\GF}}\cap \Irr(\GF)|$.
	\end{cor}
	\begin{proof}
		Let $\phi\in \UU(s,F,\si)$ and $\chi =\Gamma_{s,F,\si}(\phi)$. 
		The equality $|\Irr(\chi)|=|(\AHs^F)_\phi|$ from \Cref{labelEGFFnull}(b) implies that $\chi$ is irreducible if and only if $|(\AHs)^F_\phi|=1$. 

		On the other hand the characters from $\Gamma_{s,F,\si}^{-1}(\chi)$ form an $\AHs^F$-orbit. If $\AHs_\phi^F=1$ then $|\AHs^F|$ different characters are contained in that orbit. The map $\Gamma_{s,F,\si}$ is defined on a disjoint union using the map $\Pi_{s}$ and hence by \Cref{defPI_allg} the preimage has always $|\AHs^F|$ elements. This implies
		\begin{align*}
			|\wh \UU(s,F,\si)|=& |\AHs^F| \,\cdot \, |\ocE(\GF,(s))^{\spa{\si'_\GF}}\cap \Irr(\GF)^{\spa{\si'_\GF}}|.\qedhere
		\end{align*}
	\end{proof}

The above statement allows to determine the number of $\spa{\wGF,\si'_\GF}$-invariant characters in $\cE(\GF,\cC)$, as those are exactly the elements of $\ocE(\GF,\cC)^{\spa{\si'_\GF}}\cap \Irr(\GF)$, but this requires  that $F$ or $\si'_\HF$ acts trivially on the subgroup $B(s)$ of $\Z(\bH_0)$. In what follows we generalize this result by removing that assumption. 

Recall that for  $s\in \bH_{\textrm{ss}}$, a Frobenius endomorphism $F'$ of $\bH$ with $F'(s)=s$ and an endomorphism $\kappa$ of $\bH$ inducing an automorphism of $\Cent_\bH^\circ(s)^{F'}$, we have introduced 
\[{\rmUU(s, F', \kappa)}:=\UCh(\Cent_\bH^\circ (s)^{F'})^{\spa \kappa}\text{ and }\]
 $${{\wh \rmUU}(s, F', \kappa)}:=\{ \phi \in \rmUU (s,F',\kappa)\mid  \AHs^{F'}_\phi=1\}.$$
Note that with this notation we have 
\[\UU(s,F,\si)=\bigsqcup_{\stackrel{ a\in A_0 }{
	b\in  \AHs^F }} \rmUU(s,v_aF,v'_b\si) 
\ \ \und\ \  
\wh \UU(s,F,\si)=\bigsqcup_{\stackrel{ a\in A_0 }{
		b\in  \AHs^F } } \wh \rmUU(s,v_aF,v'_b\si),\]
        where $A_0:=\Cent_{\AHs_F}(\si)$.

We generalize this and determine the cardinality 
$|\ocE(\GF,(s))^{\spa{\tau'_{\GF}}}\cap \Irr(\GF)|$ in a case where $\tau'_\GF$ is an automorphism of $\GF$ and both $\tau'_\GF$ and $F$ induce non-trivial actions on $\AHs$.
\begin{cor}\label{cor3_19}
Let $\tau'_\HF\in \uE(\bH)$ acting on $\HF$ with even order, and let $\tau'_\GF=(\tau'_\HF)^*\in\uE(\bG)$, see \eqref{sigma*}. 
Moreover, let $\cC\in\Cl_{ss}(\HF)$ and $s\in \cC^F$ fixed at the beginning of \ref{3C}. Assume that $\tau'_\HF(s)\in[s]_\HF$ and that both $\tau'_\HF$ and $F$ act non-trivially on $B(s)$. 

Define $\si'_\HF:=\tau'_\HF\circ F$ and $\si'_\GF:=\tau'_\GF\circ F$. Then $\si'_\HF$ acts trivially on $B(s)$. Let $\wc A(s)$ and $\si\in\Cent_{\HF \si'_\HF}(s)$ be associated to $s$ via \Cref{cor_2Becht} and define $\tau$ as endomorphism of $\bH$ such that $\tau:=\si\circ F$. 
Then 
\[| \wh {\UU}(s,F,\tau)|= |\AHs^F| \cdot |\ocE(\GF,(s))^{\spa{\tau'_{\GF}}}\cap \Irr(\GF)|,\]
where 
\[\II{UsFtau}@{\wh \UU(s,F,\tau)}:=\bigsqcup_{\substack{a\in \Cent_{\AHs_F}(\si ) \\ b\in \AHs^F }} \wh \rmUU(s, v_a F,  v_a v'_b \tau)\]
for elements 
$v_a, v'_b\in \wc A(s)$ defined as in \Cref{noteps}. 
\end{cor}
\begin{proof}
As we assume that $F$ and $\tau'_\HF$ act non-trivially on $\AHs$, the automorphism $\si$ acts trivially on $\AHs$. 
Note that the set $\UU(s,F,\si)$ coincides with the set $\UU(s,F,\tau)$ defined here: the characters in  $\rmUU (s,v_aF, v'_b \si)$ are defined on $\Cent_\bH^\circ(s)^{v_a F}$ and hence are $v_aF$-invariant. Accordingly   \[\rmUU (s,v_aF, v'_b \si)=  \rmUU (s,v_aF, v'_b \si v_a  F)=  \rmUU (s,v_aF, v'_b  v_a \si F)= \rmUU (s,v_aF, v'_b  v_a \tau) .\]  
This leads to $\wh\UU(s,F,\si)=\wh\UU(s,F,\tau)$. 

By definition, the automorphisms of $\GF$ induced by $\si'_\GF$ and $\tau'_\GF$ coincide. Hence $\ocE(\GF,\cC)^{\spa{\tau'_\GF}}=\ocE(\GF,\cC)^{\spa{\si'_\GF}}$ and the statement follows directly from \Cref{lem_whU}.
%
\end{proof}

\subsection{Labelling of $\ocE(\GF,\cC)^{\spann<F_0,\gamma>}$} \label{ssec4C}
For later applications in the proof of $\Ai$ for untwisted $\tD$ (see \Cref{Ainfty_D}) we parametrize $\ocE(\GF,\cC)^{D}$,  the $D$-stable $\wGF$-orbits in $\cE(\GF,\cC)$ for any non-cyclic $2$-group $D\leq  \uE(\GF)$, see \Cref{not2_1}.
Such a group $D\leq \uE(\GF)$ is always of the form $\spa{ \restr F_0|{\GF} ,\gamma }$ with $F_0\in \uE^+(\bG)$ such that $F\in \spa{F_0^2}^+$, see \Cref{rem89}. But we also need a description of $\ocE(\GF,\cC)^{\spann<F_0,\gamma>}$ in the symmetric case where $F_0\in\spa{F}^+$, see \Cref{prop5_12}. In the following $F$ and $F_0$ play undifferentiated roles and we just apply the  considerations of the preceding section. 

\begin{defi}\label{not3_15}
	Let $F_0\in \uE^+(\bG)$ be a Frobenius endomorphism of $\bG$ such that 
    $F,F_0\in \spa{F_1}^+$ for some $F_1\in\{F,F_0\}$ with
    $[F_1, \Z(\bG)]=1$. Let $\cC\in\Cl_{\textrm{ss}}(\bH)^{F_1}$ such that there is some $s\in  \cC^{F_1}$ with $\gamma(s)\in [s]_{\bH^{F_1}}$.
	
	Let $\wc A(s)$ be the $F_1$-stable subgroup of $\Cent_\bH(s)$ with $\Cent_\bH(s)=\Cent^\circ _\bH(s)\rtimes \wc A(s)$ from  \Cref{cor_2Becht}. 
	Let $\gamma'\in \Cent_{\bH \gamma}(s)$ be  associated to $\gamma$ as in \cite[Cor.~3.1]{CS22},  such that $\wc A(s)$ is $\gamma'$-stable and $[F_1,\gamma']=1$.
	
The map $\omega_s: \AHs\lra \Z(\bH_0)$ from \ref{ssclasses} is $\spann<\gamma'>$-equivariant, as $s$ is $\spann<\gamma',F_1>$-fixed . Via $\omega_s$ we also see that $F$ and $F_0$ act trivially on $\AHs$, since 
$[F_1,\Z(\bH_0)]=1$. Hence $\AshochF=\AHs=\AsF$, as well as $\AHs^{F_0}=\AHs=\AHs_{F_0}$. 
	
	Let $\{ v_a \mid a \in \AHs_{F_1} \}\subseteq \{ v'_a \mid a \in \AHs \}= \wc A(s)$ be chosen as in \Cref{noteps} with respect to $F_1$. Note that $v'_a=v_a$ because of $\AHs=\AHs_{F_1}$.   Accordingly $\delta(v'_a)=v'_{\delta(a)}$ for every $a\in \AHs$ and $\delta\in \spa{F_1, \gamma'}$, see \Cref{cor_CS22}.
\end{defi}

\begin{prop}\label{prop616} We keep the assumptions of \Cref{not3_15} about $s$, $F$, $F_0$.
	Let  $A_0:=\Cent_{\AHs}(\gamma')$ and $\Gamma_{s,F,F_0}: \UU(s,F,F_0)\lra \ocE(\GF,\cC)^{\spa{F_0}}$ the map from \Cref{labelEGFFnull}, defined using the elements $v'_a=v_a$ from above. Then 
	\[\Gamma_{s,F,F_0 }^{-1}(\ocE(\GF,\cC)^{\spann<F_0 ,\gamma>})= \bigcup_{c\in \AHs} \UU(s,F,F_0 )^{ \spann<v_c\gamma'>}
    \]
	and 
	\[  \UU(s,F,F_0 )^{ \spann<v_c\gamma'>}
    =\bigsqcup_{ \substack{a\in A_0 \\ b\in \AHs}}  \rmUU(s,v_a F,v_b F_0)^{\spann< v_c \gamma' >}=
    \bigsqcup_{ \substack{a\in A_0 \\ b\in \AHs}}  \UCh(\Cent_\bH^\circ(s)^{v_a F})^{\spann<v_b F_0, v_c \gamma' >}.\]
\end{prop}\medskip
In Section \ref{sec5C} we will denote  $\II UsFsi@{\UU(s,F,F_0)^{[\gamma']}}:=\bigsqcup_{ \substack{a\in A_0 \\ b\in \AHs}} \bigcup_{c\in \AHs} \UCh(\Cent_\bH^\circ(s)^{v_a F})^{\spann<v_b F_0, v_c \gamma' >}$.
\begin{proof} 
We study first $\Gamma_{s,F,F_0 }^{-1}(\ocE(\GF,\cC)^{\spann<F_0 ,\gamma>})$.
We abbreviate $A:=\AHs$ and $C_a:= \Cent_\bH^\circ(s)^{v_a F}$, $\wt C_a:= \Cent_\bH(s)^{v_a F}$ for $a\in \AsF=A$.
According to (\ref{rem_ntt}), $\gamma$ acts as automorphism of $\GF$ on the rational series via applying $\gamma'$ to the semisimple element or the associated $\HF$-class. 
	The $\HF$-classes in $\cC$ are in bijection with $\AsF$, see (\ref{CFs_a}). Since $s$ is $\gamma'$-fixed, the correspondence between the $\HF$-classes in $\calC$ and $A$ is $\spa{\gamma'}$-equivariant. Consequently $\gamma'$-stable $\HF$-classes in $\cC$ are in bijection with the elements in $A_0=\Cent_{A}{(\gamma')}$. 
	This implies $$\ocE(\GF,\cC)^{\spann<\gamma>}=\bigsqcup_{a\in A_0} \ocE(\GF,[s_a]_\HF )^{\spann<\gamma>}.$$ 
	
By \Cref{newJdec_si} the map $\ov \Psi'_s$ from \ref{newJdec} satisfies 
\[\ov\Psi'_s(\ocE(\GF,\cC)^{\spann<\gamma>}) = 
	\bigsqcup_{a\in A_0}\oUCh(C_a )^{\spa{\gamma'}}
 \text{ and } \ov\Psi'_s(\ocE(\GF,\cC)^{\spann<F_0,\gamma>}) = 
	\bigsqcup_{a\in A_0}\oUCh(C_a)^{\spann<F_0,\gamma'>}.\]
Let $\Gamma_{s,F}$ be the map from \Cref{cor3_13}. For $a\in A$ the map is given on $\UCh(C_a)$ by 
\[ \restr \Gamma_{s,F}|{\UCh(C_{a}) } =
    \restr(\ov \Psi'_s)^{-1}\circ\Pi_{\wt C_a}| {{\UCh(C_{a})}}.\]
Fix $a\in A$ and $\phi\in \UCh(C_{a})$ such that $\chi:=\Gamma_{s,F,F_0}(\phi)\in \ocE(\GF,\cC)^{\spann<F_0,\gamma>}$. 

Recall from the proof of \Cref{labelEGFFnull} that $\Gamma_{s,F,F_0}(\phi)=\Gamma_{s,F}(\phi)$. 
According to \Cref{newJdec_si} the fact that $\chi$ is $\spa{\gamma,F_0}$-invariant implies that $a\in \Cent_A(\gamma',F_0)=A_0$. By the equivariance properties of the bijection $\ov \Psi'_s$ given in \Cref{newJdec_si}, the character $\Pi_{\wt C_a}(\phi)$ is even $\spann<F_0,\gamma'>$-invariant. By the construction of $\Pi_{\wt C_a}$,
this ensures that $\phi$, $\phi^{F_0}$ and $\phi^{\gamma'}$ are contained in the same $\wt C_a$-orbit. Now $\wt C_a/C_a$ is naturally isomorphic to $A^F=A$ and the $\wt C_a$-action on $C_a$ coincides with the one of $A$ given in \Cref{rem3_6}. 

If $\{\phi,\phi^{F_0},\phi^{\gamma'}\}$ are contained in the same $A$-orbit, then
\[\phi\in \bigcup_{\substack{b\in A \\ c\in A}} \UCh(C_a)^{\spann<v_b F_0,v_c \gamma'>}.\]
By the definition of $\UU(s,F,F_0)$ we see that $\phi\in\UU(s,F,F_0)$ and this character is $\spa{v_c\gamma'}$-invariant for some $c\in A$. This implies our first claim that
\[ \Gamma_{s,F,F_0 }^{-1}(\ocE(\GF,\cC)^{\spann<F_0 ,\gamma>})=\bigcup_{c\in A} \UU(s,F,F_0 )^{ \spann<v_c\gamma'>} .\]
The second stated equality follows from the definition of $\UU(s,F,F_0)$. 
\end{proof}

\begin{rem}  It would be clearly interesting to ask if similar labellings are possible whenever $\bG$ is a simply connected simple group of any type, at least when $E(\bG)$ is not cyclic.
\end{rem}

\section{Condition $\Ap$ for $\twDlsc(q)$} \label{sec4_2D}
The aim of this chapter is to prove Condition $\Ap$ (see \ref{Ainfty}) for $\GF=\twDlsc(q)$. Here $\bG=\tD_{l,\textrm{sc}}(\FF)$ ($l\geq 4$) and $F=F_p^{m}\gamma\in \uE^+(\bG)$ for $q=p^{m}$ and $p$ an odd prime.
	
 We assume as before \Cref{hyp_cuspD_ext} for all $4\leq l'<l$ and $1\leq m'$, so that we are in the situation where Condition $\Ap$ holds for all groups $\tDlsc(q')$ and prime powers $q'=p^{m'}$ by \Cref{thm_typeD1}. Note also that Condition $\Ai$ and Condition $\Ap$ are equivalent for $\GF=\twDlsc(q)$ since $E(\GF)$ is then cyclic.  In this chapter we prove the following.
	
\begin{theorem} \label{thm41}
Let $\GF=\twDlsc(q)$ with $l\geq 4$ and odd $q$, let $\wGF$ and $E(\GF)$ be associated to $\GF$ as in \ref{not}. Assume that \Cref{hyp_cuspD_ext} holds for all $4\leq l'<l$ and $m'\geq 1$. Then there exists some $E(\GF)$-stable $\wGF$-transversal in $\Irr(\GF)$. \end{theorem}
Note that whenever $p=2$ the group $\tw 2 \tDlsc(q)$ the statement holds since in this case $\wGF$ acts by inner automorphism on $\GF$. 
The fact that the groups $E(\GF)$ and $\Z(\GF)$ are then cyclic allows us to apply some ideas from \cite{CS17C} and \cite{CS18B}. According to \Cref{prop82}, \Cref{thm41} can be verified via a counting of characters. In \ref{sec4B}, the required equality of cardinalities of character sets is then proved using the parametrization of characters given in \Cref{labelEGFFnull} and \Cref{cor3_19}. 

\subsection{Proving \Cref{thm41} via counting} 	\label{not4_2}
We recall that $q$ is odd and $\Z(\bG)^{\spa{\gamma}}=\{ 1, h_0 \}$ is of order 2, see \Cref{not2_1}.

\begin{lem}\label{lemEcyclic}If 
$\GF=\twDlsc(q)$, then $E(\GF)$, $\cZ_F$ and $\Z(\GF)$ are cyclic. \end{lem}
\begin{proof} In $E(\GF)$ we have that $\gamma$ and some element of $\spann<\restr F_p|{\GF}>$ induce the same automorphism of $\GF$. This shows that $E(\GF)$ is cyclic. 
	If $l$ is odd, then $\Z(\bG)$ and both $ \Z(\bG)_F$ and $\Z(\GF)$ are cyclic. If $l$ is even, then $F_p$ acts trivially on $\Z(\bG)$ and $F(z)=\gamma(z)=h_0 z $ for every $z\in\Z(\bG)\setminus\spannh$. Then $|\Z(\bG)^F|=|\Z(\bG)_F|=2$.
	\end{proof}
Recall that ${m}$ denotes the integer with $q=p^{m}$ and $F=F_p^{m}\gamma$. We consider the decomposition $m=m_2m_{2'}$ where $m_2$ is the maximal $2$-power dividing $m$. Each $2$-subgroup of $\EGF$ has a generator of the form $F_p^{i }$ with $m_{2'}| i$ and $i| m$.
	\begin{prop}\label{prop82}
		Assume $\GF=\twDlsc(p^{m})$. Let 
		$$\II Fcal@{\cF}:= \{F_p^{i } \mid \text{ where } m_{2'}| i\text{ and } i| m \}\subseteq \uE^+(\bG).$$ The following are equivalent: 
		\begin{asslist}
			\item \Cref{thm41} holds for $\GF$;
			\item $|\Irr(\GF)^{\spann<\wGF, F_0>} | =|\Irr(\GFnull)^{\spann< \wGFnull,F>}|$ for every $F_0 \in \cF$.
		\end{asslist}
	\end{prop}

Note that the groups $ \GFnull$ and $\wGFnull$ are of untwisted type, which makes the analysis of the twisted and untwisted types impossible to separate with our approach.

	\begin{proof} 
		According to \Cref{lemEcyclic} the group $\cZ_F$ is cyclic. Hence we can choose $t\in \wGF$ such that $\GF\Z(\wGF)\spann<t>=\wGF$. Recall the isomorphism $\cZ_F=\wbG^F/(\GF\Z(\wbG^F))\cong \Z(\bG)_F$ allowing to relate diagonal outer automorphisms of $\GF$ with elements of $\Z(\bG)_F$, see (\ref{cZF}) in \ref{not}.

We prove the proposition in two steps. First we show that \Cref{thm41} holds for $\GF$ if and only if
		\begin{align} \label{eq4_0}
			|\Irr(\GF)^{\spann<\wGF, F_0>} |&
			= |\Irr(\GF)^{\spa{tF_0}}|
			\text{ for every }F_0\in \cF.\end{align} 
		This implies the statement when additionally
		\begin{align}\label{eq4_1}
			|\Irr(\GF)^{\spa{tF_0}}|&=|\Irr(\bG^{F_0})^{\spann< \wGFnull,F>}| \text{ for every }F_0\in \cF,\end{align} 
		which we ensure in the second step.

For completeness let us first show that \eqref{eq4_0} holds whenever \Cref{thm41} holds. In this case every $\wGF$-orbit in $\Irr(\GF)$ contains a character $\chi$ with $(\wGF E(\GF))_\chi=\wGF_\chi E(\GF)_\chi$. Let $\chi'\in \Irr(\GF)^{\spa{tF_0}}$ be $\wGF$-conjugate to $\chi$. If $(\wGF E(\GF))_{\chi'}=(\wGF E(\GF))_\chi$, then we get $\chi^t=\chi=\chi^{F_0}$ and hence $\chi'\in \Irr(\GF)^{\spa{\wGF,F_0}}$. If $(\wGF E(\GF))_{\chi'}\neq(\wGF E(\GF))_\chi$, then $(\wGF E(\GF))_{\chi'}/(\GF \Z(\wGF))$ is not normal in $(\wGF E(\GF))/(\GF \Z(\wGF))$, which is isomorphic to $\Z(\GF)\rtimes E(\GF)$. Computations in this group show that this is only possible if $|\Z(\GF)|=4$ and $E(\GF)_\chi\not \leq \Cent_{E(\GF)}(\Z(\GF))$. We notice that the $\Z(\GF)$-conjugates of $E(\GF)$ in $\Z(\GF)\rtimes E(\GF)$ never contain an element $zF_0$ where $z \in \Z(\GF)\setminus\spannh$. Correspondingly we observe that $\chi'$ cannot be $t F_0$-invariant. This clearly leads us to the conclusion that $\chi$ and $\chi'$ are $\spa{t,F_0}$-invariant.

Assume first that $|\Z(\GF)|=2$. Then according to the considerations in the proof of \cite[Prop.~2.2]{CS17C} 
it is sufficient for Condition $\Ap$ (see \ref{Ainfty}) to check the equation
\[ |\Irr(\GF)^{\spa{tF_0}} |=|\Irr(\GF)^{\spann<\wGF, F_0>} |
\forevery e \text{ with } e\mid |E(\GF)| \und F_0=F_p^e.\]
 As $|\Z(\GF)|=2$ and hence $|\wGF/(\GF \Z(\wGF)) |=2$, the  group $E(\GF)$ acts trivially on $\wGF/(\GF \Z(\wGF))$. 

The group $E(\GF)$ has a cyclic Sylow $2$-subgroup that is generated by the restriction to $\GF$ of $F_p^{m_{2'}}$. It is sufficient to prove $ |\Irr(\GF)^{\spa{tF_0}} |=|\Irr(\GF)^{\spann<\wGF, F_0>} |$, where $\spa{\restr F_0|{\GF}}$ runs over the $2$-subgroups of $E(\GF)$. By the definition of $\cF$ it is sufficient to check the equation for $F_0\in \cF$. Hence whenever $|\Z(\GF)|=2$, \Cref{thm41} holds for $\GF$ if and only if
		\begin{align*}
			|\Irr(\GF)^{\spa{tF_0}} |&=|\Irr(\GF)^{\spann<\wGF, F_0>} |\text{ for every }F_0\in \cF,\end{align*} 
		where $t\in \wGF$ is an element inducing a diagonal automorphism associated to some element in $\Z(\bG)\setminus\spannh$, see (\ref{cZF}). 

		It remains to consider the case where $|\Z(\GF)|=4$ and hence $\Z(\GF)$ is a cyclic group of order $4$, see \Cref{lemEcyclic}. Then $q\equiv 3 \mod 4$ and $\spann<(F_p)^{m_{2'}}>$ is the Sylow $2$-subgroup of $E(\GF)$. Note that $F_p^{m}$ and $\gamma$ coincide as automorphisms of $\GF$. Recall $\Out(\GF)=\wGF E(\GF)/(\GF\Z(\wGF))\cong \Z(\GF)\rtimes E(\GF)$. The Condition $\Ap$ holds for $\GF$ if and only if every $\wGF$-orbit in $\Irr(\GF)$ contains some $\chi\in \Irr(\GF)$ with $(\wGF E(\GF))_\chi=\wGF_\chi E(\GF)_\chi$. 
  As in \cite[Prop.~2.2]{CS17C} we see that this holds for a $\wGF$-orbit in $\Irr(\GF)$ containing a character $\chi$ with orbit sum $\ov\chi:=\Pi_{\wGF}(\chi)$ if and only if 
		\begin{align*}
	\ov \chi^\gamma=\ov \chi &\quad \Rightarrow\quad \chi^\gamma \in\{\chi, \chi^{[t,\gamma]}\}.\end{align*}
	It is sufficient to ensure that for a character $\chi$ with $\ov \chi^\gamma=\ov \chi $ we have $\chi^\gamma\in \{\chi,\chi^{t^2}\}$. Clearly $\chi^\gamma \in \{ \chi,\chi^t,\chi^{t^2}, \chi^{t^3}\}$, whenever $\ov \chi^\gamma=\ov \chi $.

We have seen that $\gamma$ acts on $\GF$ as an element of $ \cF$. 
The equalities from \eqref{eq4_0} imply 
$\Irr(\GF)^{\spann<\wGF,\gamma>} =\Irr(\GF)^{\spa{t\gamma}}= \Irr(\GF)^{\spa{t^3\gamma}}$ since 
$\Irr(\GF)^{\spann<\wGF,\gamma>} \subseteq \Irr(\GF)^{\spa{t\gamma}}$ and the sets 
$\Irr(\GF)^{\spa{t\gamma}}$ and $ \Irr(\GF)^{\spa{t^3\gamma}}$ are $\gamma$-conjugate. 

This implies that every   $\chi \in \Irr(\GF)^{\spa{t\gamma}}\cup \Irr(\GF)^{\spa{t^3\gamma}}$ is also $\gamma$-stable, whenever \eqref{eq4_0} holds. 
Then we obtain 
$$ \chi^{t\gamma}= \chi\text{ or }\chi^{t^3\gamma}= \chi \quad \Longleftrightarrow \quad \chi^\gamma=\chi=\chi^t. $$ 
\noindent Altogether we see that \Cref{thm41} holds if 	$|\Irr(\GF )^{\spa{t\gamma}}|= |\Irr(\GF)^ {\spann<\wGF,\gamma>}|$ and $|\Z(\GF)|=4$. This completes our first step.

Now we verify Equation (\ref{eq4_1}). In all cases, $F_0\in \cF$ is a Frobenius endomorphism of $\bG$ commuting with $F$. Hence we can apply the ideas of descent equalities, see the proof of \Cref{prop_ext_wG}. Recall that given a Frobenius endomorphism $\si$ there exists an equivalence relation $\sim_\si$ on $\wbG$ given by $g\sim_\si g'$ if and only $g=yg'\si(y^{-1})$ for some $y\in\wbG$. The equivalence classes are then called {\it $\si$-classes}, see \cite[Sect.~3.A]{CS18B}.
		Then the map $N_{F/F_0}$ given by $ y^{-1} F_0(y) \mapsto F(y) y^{-1}$ gives a well-defined bijection between the $F$-classes on $\wbG^{F_0}$ and the $F_0$-classes on $\wGF$, see \cite[Thm~3.2]{CS18B}.
		According to \cite[Equ.~(3.3.2)]{CS18B} we know 
		\begin{align}\label{eq_44}
			|\Irr(\GF)^{\spa{tF_0}}| &=|\Irr(\bG^{F_0})^{\spa{t_0F}}|,
		\end{align}
		where $t_0\in\wbG^{F_0}$ corresponds to $t\in\wGF$ via $N_{F/F_0}$. (Note that this part of the proof of \cite[Thm 3.3]{CS18B} applies in our situation of two commuting Frobenius endomorphisms, without assuming that one is a power of the other.)
Accordingly $t_0$ induces a diagonal automorphism of $\GFnull$ corresponding via (\ref{cZF}) to some $[\Z(\bG),F]$-coset contained in $\Z(\bG)\setminus \spa{h_0^{}}$.

Note first that $F$ induces $\gamma$ on $\GFnull$. 
By the definition of $\cF$, the group $\GFnull$ is of untwisted type and hence Condition $\Ap$ holds for $\GFnull$ thanks to \Cref{thm_typeD1}, since we assume \Cref{hyp_cuspD_ext} in lower ranks. We use the statement of $\Ap$ in terms of stabilizers (see \cite[Lem~2.4(ii)]{TypeD1}). Condition $\Ap$ then tells us that for every $\chi_0 \in \Irr(\bG^{F_0})^{\spa{t_0F}}$ there exists some $\wGFnull$-conjugate of $ \chi_0$ that is $t_0$-invariant or $t_0^\gamma$-invariant. 
Let $X$ be the subgroup of $\calZ_{F_0}$ corresponding to $\wGFnull_{\chi_0}$ via the isomorphism \eqref{cZF}.
By the above $t_0\in \wGFnull_{\chi_0}$ or 
$t_0^\gamma\in \wGFnull_{\chi_0}$, hence $X\cap \{ z [\Z(\bG),F_0] \mid z \notin\spannh \} \neq \emptyset$. Note that $(\wGFnull)_ {\chi_0}$ is $t_0F$-stable and the group $X$ is $F$-stable as well. The only subgroup of $\calZ_{F_0}$ with this property is $\calZ_{F_0}$ itself. This shows that $\chi_0$ is $\spann<t_0,F>$-invariant, hence $\chi_0\in \Irr(\GFnull)^{\spann<\wGFnull,F>}$. 
We obtain $\Irr(\bG^{F_0})^{\spa{t_0F}} =\Irr(\bG^{F_0})^{\spann< \wGFnull,F>}$. Together with (\ref{eq_44}) we obtain Equation (\ref{eq4_1}), namely
\begin{align*}
	|\Irr(\bG^F)^{\spa{tF_0}}|&=|\Irr(\bG^{F_0})^{\spann< \wGFnull,F>}| \text{ for every }F_0\in\cF. 
\end{align*}
This finishes the proof. 
\end{proof}
We now make use of the partitioning of $\Irr(\GF)$ and $\Irr(\GFnull)$ into geometric Lusztig series $\cE(\GF,\cC)$ and $\cE(\GFnull, \cC)$ for $\cC\in\Cl_{\textrm{ss}}(\bH)$, that are $F$-stable or $F_0$-stable, respectively, see \ref{series}. Instead of proving $|\Irr(\bG^{F_0}) ^{\spann< \wGFnull,F>}| = |\Irr(\GF) ^{\spann< \wGF,F_0>}|$ for $F_0 \in \cF$ we compare the intersection of those sets with geometric Lusztig series associated to the same semisimple $\bH$-conjugacy class $\cC$. 

	According to (\ref{rem_ntt}) every $\chi\in\cE(\GF,\cC)^{\spa{F_0}}$ is contained in a rational series $\cE(\GF,[s])$ with $s\in\cC^F$ and $F_0$-stable $[s]_{\HF}$. 
	Analogously every $\chi_0\in\cE(\GFnull ,\cC)^{\spa F}$ is contained in $\cE(\GFnull,[s_0])$ with some $F$-stable $[s_0]_\HFnull$ and $s_0\in\cC^{F_0}$. 
	
	We relate semisimple conjugacy classes of $\bH$ with those two properties in the following statement. 
	
\begin{lem}\label{lem95} Let $F_0\in \uE^+(\bG)$ (see Notation~\ref{not2_1}) and $F=F_0^i\circ \gamma$ for some $i\geq 1$. Let $\cC\in \Cl_{\textrm{ss}}(\bH)$. Then 
		\begin{thmlist}
			\item The following are equivalent: 
			\begin{asslist}
				\item $\cC$ contains an $F_0$-stable $\HF$-class $[s]_\HF$ for some $s\in \HF$; 
				\item $\cC$ contains an $F$-stable $\HFnull$-class $[s_0]_\HFnull$ for some $s_0\in \HFnull$. 
			\end{asslist}
\item Let $\calC$ and $s\in \cC^F$ such that $F_0(s)\in [s]_\HF$. Let $F_0'\in\HF F_0$ with $F_0'(s)=s$ and let $\wc A(s)$ be $\spa{F,F_0'}$-stable with $\Cent_{ \bH}(s)=\Cent_\bH^\circ(s)\rtimes \wc A(s)$ (given by  \Cref{cor_2Becht}(b) with $F_0$ as $\si'$). Then there exist some $g\in \bH$ and an isomorphism $\II iotass0@{\iota}: \bH \rtimes  \uE(\bH) \lra \bH\rtimes \uE(\bH)$ given by $x\mapsto x^g$, such that $\iota(F _0')=F_0$. If $s_0:=\iota(s)$ and $F':=\iota(F)$, then $s_0\in \cC^{F_0}$ and  $\wc A(s_0):=\iota(\wc A(s))$ is   $\spa{F_0,F'}$-stable with $\Cent_\bH(s_0)=\Cent_\bH^\circ(s_0)\rtimes \wc A(s_0)$.
\end{thmlist}
\end{lem}
\begin{proof}  The fact that (ii) implies (i) and the point (b) are covered by \cite[Cor. 2.6]{CS22}. It provides the existence of $s$ from the one of $s_0$ and then the isomorphism $\iota$ with all properties required in (b). In \cite[Cor. 2.6]{CS22}, which is a bit more general than the above, $F$ and $F_0$ play symmetric roles and (ii) is clearly equivalent to (i).
\end{proof}

By the above statement it makes sense to denote by $\II ClssHF@{\Cl_{\textrm{ss}}(\bH)^{(F_0, F) }}$ the $\bH$-classes in $\Cl_{\textrm{ss}}(\bH)$ that contain an $F_0$-stable $\HF$-class, or equivalently an $F$-stable $\HFnull$-class whenever $F=F_0^i\circ \gamma$. As mentioned above the counting required by \Cref{prop82} is done by comparing the intersections with geometric Lusztig series.
\begin{cor}\label{cor_96}
Let $(\bG,F)$ be as in \ref{not} with $\GF=\twDlsc(q)$ and $\cF$ from \Cref{prop82}. \Cref{thm41} holds for $(\bG,F)$, if 
\begin{align}\label{eqcor95}
|\cE(\GF,\calC)^{\spann<\wGF,F_0>}|&= |\cE(\GFnull,\calC) ^{\spann<\wGFnull,F>} |
\end{align}
for every $F_0\in\cF$ and $\calC\in \Cl_{\textrm{ss}}(\bH)^{(F_0, F)}$.
\end{cor}
\begin{proof}
According to \Cref{prop82}, \Cref{thm41} can be proved by showing 
\[|\Irr(\GFnull)^{\spann<\wGFnull,F>} |=|\Irr(\GF)^{\spann<\wGF, F_0>} | \forevery F_0\in \cF .\]
		The partitions $\Irr(\GF)=\bigsqcup_{\calC\in\Cl_{\textrm{ss}}(\bH)^F}\cE(\GF,\cC)$ and 
		$\Irr(\GFnull)=\bigsqcup_{\calC_0\in\Cl_{\textrm{ss}}(\bH)^{F_0}}\cE(\GFnull,\cC_0)$ into geometric Lusztig series from \ref{series} lead to 
		$$|\Irr(\GFnull)^{\spann<\wGFnull,F>} |=\sum_{\calC\in \Cl_{\textrm{ss}}(\bH)^{F_0}} |\cE(\GFnull,\calC)^{\spann<\wGFnull,F>}|$$
	and $$|\Irr(\GF)^{\spann<\wGF, F_0>} | =\sum_{\calC\in \Cl_{\textrm{ss}}(\bH)^F} |\cE(\GF,\calC)^{\spann<\wGF,F_0>}|.$$ Only the summands corresponding to some $\cC\in \Cl_{\textrm{ss}}(\bH)^{(F_0,F)}$ can be non-zero, see \Cref{lem95} and (\ref{rem_ntt}). The assumption now shows that the summands associated to the same conjugacy class $\calC$ are equal. 
	\end{proof}

\subsection{Labellings of $\cE(\GF,\calC)^{\spann<\wGF,F_0>}$ and $\cE(\GFnull,\calC)^{\spann<\wGFnull,F>}$ via unipotent characters}\label{sec4B}
Our aim is now to verify Equation \eqref{eqcor95} of \Cref{cor_96}. On the one hand we apply the labelling from \ref{sec3D}, on the other we use the following bijection from \cite{CS18B}. For a given $\cC\in \Cl_{\textrm{ss}}(\bH)$ the arguments used depend on the action of $F_0\in\cF$  (see \Cref{prop82}) and $F$  on $B(s)=\omega_s(\AHs)$ for $s\in \cC$: In a first step we assume that $F$ or $F_0$ acts trivially on $B(s)$ and then study the case where $F$ and $F_0$ act non-trivially on $B(s)$.

\begin{prop}[{\cite[Cor.~3.6]{CS18B}}]\label{propCSB}
Let $s\in \bH_{\textrm{ss}}$ and $\bC:=\Cent_\bH^\circ(s)$. If $F'$ and $F''$ are commuting Frobenius endomorphisms of $\bH$ fixing $s$, then there exists a bijection 
\begin{equation*}
			\II fsF'F''@{f_{s,F',F''}}:
			\UCh(\bC ^{F'})^{\spa{F''}}\lra \UCh(\bC ^{F''})^{\spa{F'}}.
\end{equation*}
 which is equivariant for algebraic automorphisms of $\bC$ commuting with $F'$ and $F''$. In particular $f_{s,F',F''}$ is $\Cent_{\bC^{F'}}(F'')$-equivariant. 
\end{prop}
This map is applied to prove the equation
$|\cE(\GF,\calC)^{\spann<\wGF,F_0>}|=|\cE(\GFnull,\calC)^{\spann<\wGFnull,F>} | $
 required in \Cref{cor_96}.
   We make use of the considerations from \ref{sec3D} where $\ocE(\GF,\cC)^{\spa{F_0}}$, the $F_0$-stable $\wGF$-orbit sums in a geometric series were studied. 
 
 \begin{notation}\label{not47}
We fix for the following $F_0\in\cF$ and  $\calC\in \Cl_{\textrm{ss}}(\bH)^{(F_0, F)}$ (defined after \Cref{lem95}).

According to \Cref{cor_2Becht}(b),  there exists $s\in \cC^F$ with $F_0$-stable $[s]_\HF$, a  $\spa{F,F_0'}$-stable group $\wc A(s)$  for some $F_0'\in\Cent_{\HF F_0}(s)$, such that $\Cent_\bH(s)=\Cent_\bH^\circ(s)\rtimes \wc A(s)$. 
Recall $B(s)=\omega_s(\AHs)$ and let  $v'_b\in \wc A(s)^F$ ($b\in \AHs$) be defined as in \Cref{cor_CS22}. Then there exists a natural isomorphism 
\[ \II{omegachecks}@{\protect{\wc \omega} _s}: \B(s)\lra \wc A(s)\] 
inverse to $\restr\omega_s|{\wc A(s)}$. By  \Cref{cor_CS22} this map is  $\spa{F,F_0'}$-equivariant.

Let now $s_0\in \cC^{F_0}$, $F'$, $\wc A(s_0)$ and the isomorphism $\iota$ be given as in \Cref{lem95}. 
Note that then $\wc A(s_0)=\iota(\wc A(s))$. As $\B(s_0)=\omega_{s_0}(A_\bH(s_0))$ we observe $\B(s)=\B(s_0)$. There exists a natural $\spa{F_0,F'}$-equivariant isomorphism 
\[\wc \omega_{s_0}: \B(s_0)\lra \wc A(s_0)\] 
defined as above. Then $\iota \circ \wc \omega_{s}=\wc\omega_{s_0}$. Let us denote the element of $\wc A(s_0)$ representing $a\in A_\bH(s_0)$ by $\II{v'0a}@{v'_{0,a}}$, see \Cref{cor_CS22} and 
the element of $\wc A(s_0)_{F_0}$ representing $a\in A_\bH(s_0)_{F_0}$ by $\II{v0a}@{v_{0,a}}$

For $a\in \AHs_F$ we have chosen a representing element $v_a\in \wc A(s)$ in \Cref{noteps}. 
In case $[\AHs,F]\neq 1$ and hence $[\Z(\bH_0),F]=\spannh$, the elements of $B(s)_F$ can be identified with elements of $\Z(\bH_0)/\spannh$ and then $v_a$ ($a\in \AHs_F$) can be chosen such that $\omicron(\B(s)_F)=\{ \omega_{s}(v_{a}) \mid a\in \AHs_{F}\}$, where $\omicron: \Z(\bH_0)/\spannh \lra \Z(\bH_0)$ is the section from \Cref{noteps}. Analogously when $[\B(s_0),F_0]\neq 1$, we choose the elements $\{v_{0,a}\mid a \in A_\bH(s_0)_{F_0}\}\subseteq \wc A(s_0)$ such that $\omicron(b)= \omega_{s_0}(v_{0,b})$ for every $a\in A_\bH(s_0)_{F_0} $ and  $b\in \B(s_0)_{F_0}$ with $\omega_{s_0}(a)=b$.
\end{notation}

Whenever $t\in \bH_{\textrm{ss}}$, $F_1$ is  a Frobenius endomorphism of $\bH$ with $F_1(t)=t$, and $\kappa$ induces an automorphism of $\Cent_\bH^\circ(t)^{F_1}$, we set
$$\II{UhattFkappa}@{\protect{\wh \rmUU}(t,F_1, \kappa)}:=\{ \phi \in \rmUU(t,F_1,\kappa) \mid (A_\bH(t)^{F_1})_\phi=1\},$$ where $\protect{\rmUU}(t,F_1, \kappa):=\UCh(\Cent_\bH^\circ(t)^{F_1})^{\spa{\kappa}}$, as before \ref{cor3_19}.
\begin{lem} \label{lem4_B} Keep $s$, $F$, $F_0'$, $F'$ as in \Cref{not47}.
Let $z,z',z''\in \B(s)$. Then 
	\begin{thmlist}
		\item 	$|\wh \rmUU(s,\wc\omega_s(z)F,\wc\omega_s(z')F_0')|= 
		|\wh \rmUU(s,\wc\omega_s(z')F_0', \wc\omega_s(z)F)| $;
		\item $|\wh \rmUU(s,\wc\omega_s(z)F,\wc\omega_s(z')F_0')| =
		|\wh \rmUU(s,\wc\omega_s(z[z'',F^{-1}])F, \wc\omega_s(z'[z'',{F_0'}^{-1}])F_0')| $;
		\item $|\wh \rmUU(s,\wc\omega_s(z)F_0',\wc\omega_s(z')F)| =
		|\wh \rmUU(s_0,\wc\omega_{s_0}(z)F_0, \wc\omega_{s_0}(z')F')| $.
	\end{thmlist}
\end{lem}
\begin{proof} Recall the notation ${\rmUU(s, F', \kappa)}:=\UCh(\Cent_\bH^\circ (s)^{F'})^{\spa \kappa}$ from \ref{sec3D}.
	\Cref{propCSB} gives us a bijection  $$f_{s,\wc\omega_s(z)F,\wc\omega_s(z')F_0'}\colon \rmUU(s,\wc\omega_s(z)F,\wc\omega_s(z')F_0')\lra\rmUU(s, \wc\omega_s(z')F_0',\wc\omega_s(z)F).$$ Let  $A:=\AHs$. Via $\Cent_{\wc A(s)}(\spa{F,F_0'} )$ the group $\Cent_{ A}(\spa{F,F_0'})$ acts on both character sets and the bijection is $\Cent_{ A}(\spa{F,F_0'})$-equivariant. 
	
	Part (a) follows from \[f_{s, \wc\omega_s(z)F,\wc\omega_s(z')F_0'} (\wh \rmUU(s,\wc\omega_s(z)F,\wc\omega_s(z')F_0'))=\wh 
	\rmUU(s, \wc\omega_s(z')F_0',\wc\omega_s(z)F).\]
	This equality is clear, if $ A^{F}= A^{F_0}$.
	Otherwise consider  $\phi\in \rmUU(s,\wc\omega_s(z)F,\wc\omega_s(z')F_0')$ and $\phi':= f_{s,\wc\omega_s(z)F,\wc\omega_s(z')F_0'}(\phi)$:
	Let $Z:=\omega_s({A^{F}}_\phi)$, 
	$Z':=\omega_s((A^{F_0'})_{\phi'})$ and
	 $Z''=\Cent_{\omega_s(A)} ( 
	 \spa{F,F_0' } )$. 
	As $f_{s,\wc\omega_s(z)F,\wc\omega_s(z')F_0'}$ is $\Cent_{ A}(\spa{F,F_0'})$-equivariant, $Z\cap Z''=Z'\cap Z''$. Since $Z$ and $Z'$ are by this construction $\spa{F,F_0'}$-stable, an easy check in $\Z(\bH_0)$ shows that $Z=Z'=1$, whenever $Z$ or $Z'$ are trivial. 
	This implies that $(A^{F})_\phi=1$ and $(A^{F_0'})_{\phi'}=1$ are equivalent. 
	This leads to \[f_{s, \wc\omega_s(z)F,\wc\omega_s(z')F_0'} (\wh \rmUU(s,\wc\omega_s(z)F,\wc\omega_s(z')F_0'))=
	\wh \rmUU(s, \wc\omega_s(z')F_0',\wc\omega_s(z)F),\]
	hence part (a).
	
	Part (b) follows from conjugating the set $\wh \rmUU(s,\wc\omega_s(z)F, \wc \omega_s(z')F_0')$ with $\wc \omega_s(z'')$: 
	\begin{align*}
	\wh \rmUU(s,\wc\omega_s(z)F, \wc \omega_s(z')F_0')^{\wc \omega_s(z'')}&=
\wh \rmUU(s,\wc\omega_s(z)F^{{\wc \omega_s(z'')}}, \wc \omega_s(z')(F_0')^{\wc \omega_s(z'')})=\\
&=\wh \rmUU(s, 
\wc\omega_s(z [z'',F\inv]) F, 
\wc\omega_s(z' [z'',{F_0'}^{-1}]) F_0' ).
\end{align*}
Recall that $\wc A(s)$ is abelian, $F(\wc \omega_s(z))=\wc \omega_s(F(z))$ and $F_0'(\wc \omega_s(z'))=\wc \omega_s(F_0'(z'))$. This leads to the equality in (b)

Part (c) follows from an application of  the isomorphism $\iota$ from \Cref{lem95}. We have $\iota(s)=s_0$, $\iota(F)=F'$, $\iota(F_0')=F_0$ and $\iota(\wc A(s)^{F_0'})=\wc A(s_0)^{F_0}$: The map $\iota$ satisfies $\iota\circ \wc \omega_s=\wc \omega_{s_0}$ and induces accordingly a bijection between  $\wh \rmUU(s,\wc \omega_s(z) F_0', \wc \omega_s(z')F)$ and $\wh \rmUU(s,\wc \omega_{s_0}(z) F_0, \wc \omega_{s_0}(z')F')$. 
This proves the equality in part (c).
\end{proof}
If $F$ or $F_0'$ acts trivially on $\B(s)$, the set $\protect{\wh\UU}(s,F,F_0')$ from \Cref{lem_whU} rewrites as 
\[\II{UsFtau}@{\protect{\wh\UU}(s,F,F_0')}=
\bigsqcup_{\substack{a\in \Cent_{\AHs_F}(F_0') \\ b\in \AHs^F }} \wh \rmUU(s, v_a F, v'_b F_0')
=\bigsqcup_{\substack{z\in Z \\ z'\in \B(s)^F }} \wh \rmUU(s, \wc \omega_s(z) F, \wc \omega_s(z') F_0'),
\] 
where $Z:=\omega_s(\{v_a\mid a\in (\AHs_F)^{F_0'}\})$.
If $F$ and $F_0'$ act non-trivially on $\B(s)$, then the assumptions of \ref{cor3_19}  are satisfied for $\tau'_\HF=F_0$ and $\tau=F_0'$, which allows us to define the set 
\[ \wh\UU(s,F,F_0'):= \bigsqcup_{\substack{a\in \Cent_{\AHs_F}(F_0 ) \\ b\in \AHs^F }} \wh \rmUU(s, v_a F, v_a v'_b F_0')
=\bigsqcup_{\substack{z\in Z \\ z'\in \B(s)^F }} \wh \rmUU(s, \wc \omega_s(z) F, \wc \omega_s(z z') F_0'),\] 
where $Z:=\omega_s(\{v_a\mid a\in (\AHs_F)^{F_0'}\})$. 
Let $\wc \UU(s_0,F_0,F')$ be defined analogously using $\wc A(s_0)$ and $\{v_{0,a}\mid a \in A_\bH(s_0)_{F_0}\}$. According to \Cref{lem_whU} and \ref{cor3_19} the set $\ocE(\GF,\calC)^{\spa{F_0}}\cap \Irr(\GF)$ is parametrized by the set $\wh {\UU}(s,F,F_0')$. It leads to the following.
\begin{lem}\label{lemdesc}
Let $\calC\in \Cl_{\textrm{ss}}(\bH)^{(F_0,F)}$, $s_0\in \calC^{F_0}$ and $s\in \cC^F$ as above.
Then:
\begin{thmlist}
    \item $|\wh \UU (s,F,F'_0)|= |\AHs^F| \,\cdot \, |\ocE(\GF,(s))^{\spa{F_0}}\cap \Irr(\GF)^{\spa{F_0}}|$, and 
    \item $ |\wh \UU (s_0,F_0,F')|= |A_\bH(s_0)^{F_0}| \,\cdot \, |\ocE(\GFnull,(s_0))^{F}\cap \Irr(\GFnull)|$.
\end{thmlist}
\end{lem}

\begin{prop}\label{cor96}
Let $F_0\in\cF$, $\calC\in \Cl_{\textrm{ss}}(\bH)^{(F_0, F)}$ and $s\in \cC^F$. If  $F$ or $F_0$ acts trivially on $\B(s)=\omega_s(\AHs)$, then 
\begin{align}\label{eqcor96}|\cE(\GF,\calC)^{\spann<\wGF,F_0>}|&=|\cE(\GFnull,\calC)^{\spann<\wGFnull,F>} | .
\end{align}
\end{prop}

\begin{proof}
In a first step we prove for $A:=\AHs$,
$B:=\begin{cases} 
	\omicron(\B(s)_{F_0'})& \text{, if }[\B(s), F_0']\neq 1,\\
	\B(s)^{F}& \otw, \end{cases}$
and 
\[\II{hat Us,F_0',F}@{\protect{\wh \UU(s,F_0',F)}}:= \bigsqcup_{ \substack{z\in B \\ z'\in \B(s)^{F}} } \wh\rmUU(s, \wc \omega_s(z) F_0', \wc \omega_s(z')F) \]
the equality 
\begin{align} \label{eq_UU}
|[A,F]|\ \cdot \ |\wh \UU (s,F,F_0')|= |[A,F_0']|\, \cdot \, |\wh \UU (s,F_0',F)|.
\end{align}
We give a detailed argument in the case where $[A,F]=1$. The other case, where $\omega_s([A,F])=\spa{h_0}$, follows from similar considerations and uses $[A,F_0']=1$.  

Assuming $[A,F]=1$, we have $A/[A,F]=A=A^F$ and  hence \Cref{lem_whU} applies with the set
\begin{align}\label{eq4_10}
{\wh \UU(s,F,F_0')}&=\bigsqcup _{\substack{ z\in \B(s)^{F_0'} \\ z'\in \B(s)^F}} \wh \rmUU(s, \wc\omega_s(z) F, \wc\omega_s(z') F_0')
\end{align} 
by its definition in \ref{lem_whU}. 

\Cref{lem4_B}(b) implies in combination with $[A,F]=1$  the equality
\[| \wh \rmUU(s,\wc \omega_s(z)F, \wc \omega_s(z')F_0')|= |\wh \rmUU(s,\wc \omega_s(z)F, \wc \omega_s(z'[z'',F_0'])F_0')|,\] 
for every  $z,z',z''\in \B(s)$. Note that by definition $B$ is a set of representatives of the $F$-stable $[\B(s),F_0']$-cosets in $\B(s)$ and hence  
\[ |\wh \UU(s,F,F_0')|= 
 \left |    \bigsqcup_{\substack{ z\in \B(s)^{F_0'} \\ z'\in  \B(s)}}  
 \wh \rmUU(s, \wc \omega_s(z) F, \wc \omega_s(z') F_0')
\right |
=|[\B(s),F_0']| \ \cdot \  \left |    \bigsqcup_
{\substack{ z\in\B(s)^{F_0'} \\ z'\in  B}}  \wh \rmUU(s, \wc \omega_s(z) F, \wc \omega_s(z') F_0')
 \right |.\]

Applying \Cref{lem4_B}(a) leads to 
\[ |\wh \UU(s,F,F_0')|= 
 |[\B(s),F_0']| \ \cdot \  \left |    \bigsqcup_
{\substack{ z\in B\\ z'\in\B(s)^{F_0'} }}  \wh \rmUU(s, \wc \omega_s(z) F_0', \wc \omega_s(z') F) \right |. \]
This ensures Equality \eqref{eq_UU} in the case where $[\B(s),F]=1$ or equivalently ${[A,F]}=1$. Otherwise, $[A,F]\neq 1$ and hence $[\B(s),F_0]=1$. Then analogous considerations prove the equality since $F$ and $F'_0$ play relatively similar roles.

Next we prove \begin{align}
| \wh \UU(s,F_0',F) |
=| \wh \UU(s_0,F_0,F') |.\label{UU_410}
\end{align}
Let now $s_0$ and $\iota$ be as in \Cref{lem95}. According to \Cref{lem4_B}(c) the isomorphism $\iota$ shows the equality 
\[\left |\wh \rmUU(s, \wc \omega_s(z') F_0', \wc \omega_s(z) F) \right |=\left | 
\wh \rmUU(s_0, \wc \omega_{s_0}(z') F_0, \wc \omega_{s_0}(z) F') \right |\]
for every $z, z'\in \B(s)$. 
We see 
\[   \left |    \bigsqcup_
{\substack{ z\in B\\ z'\in\B(s)^{F_0'} }}  \wh \rmUU(s, \wc \omega_s(z) F_0', \wc \omega_s(z') F) \right |=
 \left |    \bigsqcup_{\substack{ z\in B\\ z'\in\B(s)^{F_0'} }}  \wh \rmUU(s_0, \wc \omega_{s_0}(z) F_0, \wc \omega_{s_0}(z') F') \right |. \]
We observe that $B=\omega_{s_0}(\{v_{0,a}\mid a \in (A_\bH(s_0)_{F_0'})^{F} \})$ and hence
\[   \left |    \bigsqcup_
{\substack{ z\in B\\ z'\in\B(s)^{F_0'} }}  \wh \rmUU(s_0, \wc \omega_{s_0}(z) F_0, \wc \omega_{s_0}(z') F') \right |= \left | \wh \UU(s_0,F_0,F') \right| . \] 
Altogether, this ensures the equality \eqref{UU_410}.

A character $\chi\in\Irr(\GF)$ is $\wGF$-invariant if and only if $\chi = \Pi_{\wGF}(\chi)$. This leads to 
\begin{align}
\cE(\GF,\cC)^{\spann<\wGF,F_0>}&= \ocE(\GF,\calC)^{\spa{F_0}}\cap \Irr(\GF).\label{ocEcapIrr}\end{align}
Analogously $\cE(\GFnull,\calC)^{\spann<\wGFnull,F>}= \ocE(\GFnull,\calC)^{\spann<F>}\cap \Irr(\GFnull)$. 

We prove our claim by combining the above results successively. Recall that $\cC=(s)=(s_0)$.
We obtain the following equalities:
\begin{eqnarray*}
	|\cE(\GF,\cC)^{\spann<\wGF,F_0>}|&=& |\ocE(\GF,\calC)^{\spa{F_0}}\cap \Irr(\GF)|= \frac{|\wh \UU(s,F,F_0')|}{|\AHs^F|}\\
    &=& \frac{ | [\AHs, F_0']| \  \cdot  | \wh \UU(s,F_0',F)| }
	{ | [\AHs, F]|\ \cdot \ |\AHs^F|   }
    \\
	&=&  \frac{ | [\B(s), F_0']| \  \cdot  | \wh \UU(s,F_0',F)| }
	{ | [\B(s), F]|\ \cdot \ |\B(s)^F|}\\
			&=& \frac{|\wh \UU(s,F_0',F)| } {|\B(s)^{F_0'}| }=
			\frac{|\wh \UU(s_0,F_0,F')| }{|B(s_0)^{F_0} | }\\
			&=& |\ocE(\GFnull,\cC)^{\spa F}\cap \Irr(\GFnull)|
			=|\cE(\GFnull,\cC)^{\spann<\wGFnull,F>}|.
		\end{eqnarray*} 
The first equality and the last follow from \eqref{ocEcapIrr} and its analogue. The second equality is a consequence of \Cref{lemdesc}(a). Afterwards, we apply \eqref{eq_UU}. Then we can replace $\AHs$ by $B(s)$. 
Next we use that  $|\B(s)^F| \, \cdot \, |[\B(s),F]|=|\B(s)|=|\AHs|$ leads to 
 $\frac{ | [\B(s), F_0']| }
	{ | [\B(s), F]|\ \cdot \ |\B(s)^F|}= \frac {1}{|B(s)^{F_0'}|}$. 
Another ingredient is the equality $| \wh \UU(s,F_0',F) |
=| \wh \UU(s_0,F_0,F') |$ from \eqref{UU_410}. This implies the required equality. 
\end{proof}
\noindent It remains to consider the case where $F$ and $F_0$ both act non-trivially on $\B(s)$.
\begin{prop}\label{prop_nontrivFF0action} Let $F_0\in\cF$, $\calC\in \Cl_{\textrm{ss}}(\bH)^{(F_0, F)}$ and $s\in \cC^F$.
Assume that $F$ and $F_0$ both act non-trivially on $\B(s)$. Then 
\[ |\cE(\GF,\cC)^{\spa{\wGF,F_0}}| = |\cE(\GFnull,\cC)^{\spa{\wGFnull,F}}|.\]
\end{prop}
\begin{proof}
By definition $\B(s)\leq \Z(\bH_0)$ and $\B(s)=\B(s_0)$. The assumption that $F$ and $F_0$ act non-trivially on $\B(s)$ implies $|\B(s)|=|\B(s_0)|=|A_\bH(s_0)|=|\AHs|=4$. 
In this case $\B(s)^F=\spannhHnull=\B(s_0)^{F_0}$ and hence $\AHs^F=\AHs^{F_0'}$.
The groups $\B(s)/[\B(s),F]$ and $\B(s_0)/[\B(s_0),F_0]$ coincide with $\Z(\bH_0)/\spannhHnull$. Hence $F_0$ and $F$ act trivially on $\B(s)/[\B(s),F]$. Via $\omega_s$ and $\omega_{s_0}$ this implies $(\AHs_F)^{F_0'}=\AHs_F$, analogously $A_\bH(s_0)^F=A_\bH(s_0)^{F_0'}$ and $(A_\bH(s_0)_{F_0})^{F'}=A_\bH(s_0)_{F_0}$.
Let  $\III{Z^\diamond}=\spannhHnull$ and $\III{Z_\diamond}=\omicron(\Z(\bH_0)/Z^\diamond)$ where $\omicron$ is the map from \Cref{noteps}. 

According to \Cref{lemdesc} we have 
\[ |\cE(\GF,\cC)^{\spa{\wGF,F_0}}| = \frac{| \wh \UU(s,F,F_0')|}{|\AHs^F| }  
\und |\cE(\GFnull,\cC)^{\spa{\wGFnull,F}}|=\frac{|\wh \UU(s_0,F_0,F')|}{|A_\bH(s_0)^{F_0}| }, \]
where $\wh \UU(s,F,F_0')$ and $\wh \UU(s_0,F_0,F')$ are the sets defined there. 
Hence we have to show $| \wh \UU(s,F,F_0')| =|\wh \UU(s_0,F_0,F')|$.

Recall that $\omicron: \Z(\bH_0)/Z^\diamond \ra \Z(\bH_0)$ is a section with $\omicron(1)=1$ and we assume that the elements $v_a\in \wc A(s)$ are chosen such that   
$\wc \omega_s {(\omicron(z))}= v_a$, whenever $\omega_s(a)=z$ for $a\in \AHs_F$. This leads to 
$\omega_s(v_a)=\omicron(a)$ for every $a\in \AHs_F$, see before \Cref{lemdesc}. 
We then get $\{v_a\mid a\in \AHs_F\}= \{\wc \omega_s(\omicron(z))\mid z\in Z_\diamond\}$ and 
$\{v'_a\mid a\in \AHs^F\}= \{\wc \omega_s(z')\mid z'\in Z^\diamond\}$.
The disjoint union $\wh \UU(s,F,F_0')$ can be rewritten as 
\[ 
\wh \UU(s,F,F_0')= \bigsqcup_{\substack{ a\in (\AHs_F) ^{F_0'} \\ b\in \AHs^F}} \wh \rmUU(s,v_a F, v_av'_b F_0')=
 \bigsqcup_{\substack{ z\in Z_\diamond \\ z'\in Z^\diamond}} \wh \rmUU(s,\wc \omega_s(z) F, \wc \omega_s(zz') F_0').
\]
Analogously we see 
\[ 
\wh \UU(s_0,F_0,F')= \bigsqcup_{\substack{ a\in (A_\bH(s_0)_{F_0}) ^{F'} \\ b\in A_\bH(s_0)^{F_0}}} \wh \rmUU(s_0,v_a F_0, v_a v'_b F')=
\bigsqcup_{\substack{ z\in Z_\diamond \\ z'\in Z^\diamond}} \wh \rmUU(s_0,\wc \omega_{s_0}(z) F_0, \wc \omega_{s_0}(zz') F').
\]

In order to prove $|\wh \UU(s,F,F_0')|=|\wh \UU(s_0,F_0,F')|$ it is therefore sufficient to prove for every $z\in Z_\diamond$ and $ z'\in Z^\diamond$ the equality 
\[|\wh \rmUU(s,\wc \omega_s(z) F, \wc \omega_s(zz') F_0')|
=|\wh \rmUU(s_0,\wc \omega_{s_0}(z) F_0, \wc \omega_{s_0}(z z') F')|.\]

\Cref{lem4_B}(a) leads to 
\[
|\wh \rmUU(s, 
\wc \omega_s(z) F, \wc \omega_s(zz') F_0' )|=|\wh \rmUU(s, 
\wc \omega_s(zz') F_0',
\wc \omega_s(z) F)|.
\] 
If $z'=1$ then $\wh \rmUU(s, 
\wc \omega_s(zz') F_0',
\wc \omega_s(z) F)=\wh \rmUU(s, 
\wc \omega_s(z) F_0',
\wc \omega_s(zz') F)$. 
If $z'\neq 1$, then $z'=h_0^{(\bH_0)}$  and $z'=[z'',F_0]$ for some $z''\in \Z(\bH_0)\setminus Z^\diamond$. In combination with \Cref{lem4_B}(b) we obtain 
\begin{align*}
|\wh \rmUU(s, 
\wc \omega_s(zz') F_0',
\wc \omega_s(z) F)^{v''}|&=
|\wh \rmUU(s, \wc \omega_s(zz' z') F_0',
\wc \omega_s(zz') F)|\\
&= |\wh \rmUU(s, 
\wc \omega_s(z) F_0',
\wc \omega_s(zz') F)|.
\end{align*}
In all cases we see that 
\[|\wh \rmUU(s, 
\wc \omega_s(zz') F_0',
\wc \omega_s(z) F)|= 
|\wh \rmUU(s, 
\wc \omega_s(z) F_0',
\wc \omega_s(zz') F)|.
\]
The sets $\wh \rmUU (s,\wc \omega_s(z) F_0', \wc \omega_s(zz') F)$ and
$\wh  \rmUU(s_0,
\wc \omega_{s_0}(z) F_0,
\wc \omega_{s_0}(zz') F')$ have the same cardinality according to \Cref{lem4_B}(c).

Taking all those equalities together we now have 
\[|\wh \rmUU(s,\wc \omega_s(z) F, \wc \omega_s(zz') F_0')|
=|\wh \rmUU(s_0,\wc \omega_{s_0}(z) F_0, \wc \omega_{s_0}(z z') F')|.
\]
As explained in the beginning this implies the stated equality.
\end{proof}
 
\begin{proof}[Proof of \Cref{thm41}]
Equation \eqref{eqcor95} from \Cref{cor_96} holds according to \Cref{cor96} and \Cref{prop_nontrivFF0action}. This implies \Cref{thm41}.
\end{proof} 

	\begin{rem}\label{rem4x}
		In combination with \Cref{thm_typeD1}, \Cref{thm41} proves Condition $\Ap$ for $^2\tD_{l,\mathrm{sc}}(q)$ assuming \Cref{hyp_cuspD_ext} for groups of type $\tD_{l'}$ ($l'<l$). This suggests that it is not possible with our method to separate the proof of condition $\Ap$ for type $^2\tD$ from the one for type $\tD$.
	\end{rem}
	\section{Counting invariant characters with no extensions} \label{sec_ext_D}
	The aim of this chapter is to prepare the proof of Condition $\Ai$ from \ref{Ainfty} for groups of simply connected type $\tD_l$ by induction on $l$. 
	As before we assume \Cref{hyp_cuspD_ext} for groups of smaller rank, i.e., Condition $\Ai$ for $\cusp(\tDlprimesc(p^{m'}))$ for all $4\leq l'<l$ and $1\leq m'$.

	Let us point out the surprising fact that for the proof of \Cref{thm1}, which leads to a description of the action of $\Aut(\GF)$ on $\Irr(\GF)$, it is so far indispensable to determine the Clifford theory \wrt $\GF \unlhd \GF E(\GF)$ by proving the full Condition $\Ai$ for smaller ranks and cuspidal characters. 
	\begin{theorem}\label{Ainfty_D}
		Let $l\geq 4$, $q$ a prime power, $(\bG,F)$ with $\GF=\tDlsc(q)$, and $\wbG$ and $E(\GF)$ be associated to $\GF$ as in \ref{not}. 
		Assume \Cref{hyp_cuspD_ext} for all $l'$ with $4\leq l'< l$ and $1\leq m'$.
		Then there exists some $E(\GF)$-stable $\wGF$-transversal $\TT$ in $\Irr(\GF)$ such that every $\chi\in\TT$ extends to $\GF E(\GF)_\chi$. 
	\end{theorem} 
   \Cref{hyp_cuspD_ext} implies that Condition A$(\infty)$ holds for the set of cuspidal characters of $\tD_{l',sc}(p^{m'})$ with $4\leq l'< l$.  This in turn  ensures that an $E(\GF)$-stable $\wGF$-transversal $\TT$ in $\Irr(\GF)$ exists according to  \Cref{thm_typeD1}. The purpose of this chapter and the next is to show that $\TT$ can be chosen such that every $\chi\in\TT$ extends to $\GF E(\GF)_\chi$, that is to prove the above \Cref{Ainfty_D}. The final proof is given in section \ref{ssec5D}. 

	In \Cref{7A} we make some simplifying remarks. First we show that even in the case of type $\tD_4$ only 2-subgroups of $E(\GF)$ need to be  considered. With \Cref{cor79} we obtain a criterion that implies \Cref{Ainfty_D} and allows us to prove it via a counting argument. 
	
	Whenever $D$ is a non-cyclic $2$-group $D\leq E(\GF)$ and $F_0\in D$ denotes a certain Frobenius endomorphism restricted to $\GF$ with $D=\spann<F_0,\gamma>$, the number of $D$-invariant characters with an extension to $\GF D$ can be determined using Shintani descent and is $\frac 1 2 |\Irr(\GF)^D|+\frac 1 2 |\Irr(\GFnull)^{\spann<\gamma>}|$, see \Cref{thm77}. (Thereby $F_0$ is a well-chosen  Frobenius endomorphism.) On the other hand $\frac 1 2 |\Irr(\GF)^D|-\frac 1 2 |\Irr(\GFnull)^{\spann<\gamma>}|$ characters of $\GF$ are $D$-invariant and do not extend to $\GF\rtimes D$. \Cref{cor79}
    suggests how these characters are distributed over the geometric Lusztig series, in particular Lusztig series associated to $D$-stable geometric conjugacy classes $\cC\in \Cl_{\textrm{ss}}(\bH)$. 
	
	If the $D$-stable geometric conjugacy class $\cC\in \Cl_{\textrm{ss}}(\bH)$ contains no $\gamma$-stable $\HFnull$-conjugacy class we introduce in \ref{sec8B} the set $\EE(\cC)$ as a $\wGF$-transversal of $\cE(\GF,\cC)^D$.

	In \ref{sec5C} we assume $\cC\in\Cl_{\textrm{ss}}(\bH)$ to contain some $\gamma$-stable $\HFnull$-conjugacy class. 
    In this situation we construct a set $\wt\EE(\cC)$ and  a bijection between $\cE(\GF,\cC)^D\setminus\wt\EE(\cC)$
    and $\cE(\GFnull,\cC)^{\spa{\gamma}}$ via relating both with unipotent characters, see \Cref{thm_desc}(b).
    The study of $\wt\EE(\cC)$ is then continued in Chapter \ref{SecEC}.

\subsection{More basic considerations and assumptions}\label{7A}
As said above, thanks to \Cref{hyp_cuspD_ext} and  \Cref{thm_typeD1} we know that there exists some $E(\GF)$-stable $\wGF$-transversal $\TT$ in $\Irr(\GF)$. It remains to prove that $\TT$ can be chosen such that every $\chi\in\TT$ extends to $\GF \rtimes E(\GF)_\chi$. 
Note that a Sylow $\ell$-subgroup of $E(\GF)$ is non-cyclic if and only if $\ell=2$ or $\ell=3$ and $\GF$ is of type $\tD_4$. 
We see at the end of this section that it is sufficient to find enough $D$-invariant characters of $\GF$ that extend to $\GF D$, where $D$ is a non-cyclic $2$-subgroup of $E(\GF)$ contained in  $\spa{\gamma,F_p}$. 
The next lemma essentially solves the problem specific to type $\tD_4$ with regard to the non-cyclic Sylow $3$-subgroup of $E(\GF)$. 

\begin{lem}\label{lem73} Keep $\GF$ as in \Cref{Ainfty_D}. Let $\ell$ be an odd prime and $E'$ a Sylow $\ell$-subgroup of $E(\GF)$, then every $\chi\in \Irr(\GF)$ extends to its stabilizer in $\GF E'$. \end{lem}

\begin{proof} 
Recall that any element of $\Irr(\wGF)$ restricts to $\GF$ without multiplicity according to \cite[Thm~15.11]{CE04}. The same property is true for restriction to every $G_1$ with $\GF\leq G_1\leq \wGF$ and this means that maximal extendibility then holds \wrt $G_1\unlhd \wGF$, see \ref{not11}.	Since $\Z(\GF)$ is a $2$-group, we can consider $\GF\times \Z(\wGF)_{2'}$ as a subgroup of $\wGF$ and the group $\wGF/(\GF\times \Z(\wGF)_{2'})$ is a $2$-group. Maximal extendibility then holds with respect to $\GF\times \Z(\wGF)_{2'}\unlhd \wGF$ and as a consequence the cardinality of $\Irr(\wGF |\chi \times 1_{\Z(\wGF)_{2'}})$ is a power of $2$. As $E'_{\chi}$ is an $\ell$-group acting on this set with $\ell\neq 2$, there exists $\wt \chi\in \Irr(\wGF |\chi \times 1_{\Z(\wGF)_{2'}})^{E'_{\chi}}$. The character $\wt \chi$ extends to $ \wGF E'_{\chi}$ by \Cref{prop_ext_wG}. Hence $\chi$ extends to $\GF E'_{\chi}$. \end{proof}

The above justifies  that we consider extendibility only with regard to 2-subgroups of $\uE(\GF)$ from  \Cref{not2_1}. Here are some useful properties of non-cyclic 2-subgroups of $\uE(\GF)$.

\begin{lem}\label{rem89}
Let $D\leq \uE(\GF)$ be a non-cyclic $2$-subgroup. Then there exists some $F_0\in \uE(\bG)^+$ with $F\in \spa{F_0^2}^+$ (see \ref{not2_1}), $[\Z(\bG),F_0]=1$ and $D=\spann<\restr F_0|{\GF},\gamma>$. 
\end{lem}
\begin{proof} A non-cyclic subgroup of $\uE(\GF)$ has to contain $\gamma$, since otherwise it would be isomorphic to its image in the cyclic group $\uE(\GF)/\spann<\gamma>$. Recall the integer $m\geq 1$ with $q=p^{m}$, so that $\uE(\GF)\cong \Cy_{m}\times \Cy_2$ which is non-cyclic only if $m$ is even. Let $m_{2'}$ be the maximal odd divisor of $m$. The group $\spann<(\restr F_p|{\GF})^{m_{2'}},\gamma>$ is the Sylow $2$-subgroup of $\uE(\GF)$ and hence $D=\spann<\restr F_p^i|\GF ,\gamma>$ for some divisor $i$ of $m$ with $m_{2'}\mid i$ and $i<m$, so that $2i\mid m$. If $F_p^i$ acts trivially on $\Z(\bG)$, we can take $F_0=F_p^i$. Otherwise it means that $l$ and $i$ are odd, see \Cref{not2_1}. Then $F_p^i\gamma$ acts trivially on $\Z(\bG)$ and we can take $F_0=F_p^i\gamma$ since $m$ is even. In all cases we also have $F\in \spa{F_0^2}^+$. \end{proof}

\begin{prop}\label{propE4} Assume: 
\begin{asslist}
    \item there exists some $\EGF$-stable $\wGF$-transversal $\TT$ in $\Irr(\GF)$; and
    \item for every non-cyclic 2-subgroup $D\leq \uE(\GF)$ and $\ov \chi\in \Pi_{\wGF}(\Irr(\GF))^D$ one of the following holds
    \begin{enumerate}
	\item[(ii.1)] every $\chi \in\Irr(\ov \chi)^D$ extends to $\GF D$; 
	\item[(ii.2)] $|\Irr(\ov \chi)^D|=2$ and some  $\chi \in\Irr(\ov \chi)^D$ extends to $\GF D$.
    \end{enumerate}
\end{asslist}
Then \Cref{Ainfty_D} holds.
\end{prop}
	Assumption (i) is ensured by \Cref{hyp_cuspD_ext} in lower ranks and \Cref{thm_typeD1} but stated here separately for clarification. 
\begin{proof}
We prove that for each $\ov \chi \in\Pi_{\wGF}(\Irr(\GF))$ there exists some $\chi_0\in\Irr(\ov \chi)$ such that $(\wGF \EGF)_{\chi_0}=\wGF_{\chi_0} \EGF_{\chi_0}$ and $\chi_0$ extends to $\GF \EGF_{\chi_0}$. By the proof of \cite[Lem.~2.4]{TypeD1} this implies the existence of a $\wGF$-transversal with the additional properties. 
On the other hand by conjugacy we see that it is sufficient to verify this statement for some $\ov \chi$ in each $E(\GF)$-orbit in $\Pi_{\wGF}(\Irr(\GF))$. 

We choose $\ov \chi \in \Pi_{\wGF}(\Irr(\GF))$ such that some Sylow $2$-subgroup of $E(\GF)_{\ov \chi}$ is contained in $\uE(\GF)$. Again by \cite[Lem.~2.4]{TypeD1}, Assumption (i) implies that there exists some $\chi\in\Irr(\ov \chi)$ with $(\wGF \EGF)_\chi=\wGF_\chi \EGF_\chi$ and therefore $\EGF_{\ov\chi} =\EGF_\chi$. Let $E_2$ be a Sylow $2$-subgroup of $\EGF_\chi$ such that $E_2\leq \uE(\GF)$, which exists by our choice of $\ov \chi$ and $\chi$. By \cite[Cor.~11.22]{Isa} and \Cref{lem73}, it suffices to show that there is $\chi_0\in\Irr(\ov\chi)$ with $E(\GF)_\chi =E(\GF)_{\chi_0}$ and that  extends to $\GF E_2$. If $E_2$ is cyclic the claim is clear with $\chi_0:=\chi$. 

We now assume that $E_2$ is non-cyclic and we apply Assumption (ii) to $D=E_2$ and $\ov\chi$. If (ii.1) holds for $\chi$, then  $\chi$ extends to $\GF D$ and again $\chi$ can be chosen as $\chi_0$. 

Next we assume that the statement in (ii.2)  holds and $|\Irr(\ov \chi)|= 2$. Recall that by \Cref{rem89} $\gamma\in D$. Then  $\wGF_\chi $ being normalized by $E(\GF)_\chi$ is a $\gamma$-stable subgroup of $\wGF$. By the action of $E(\bG)$ on $\Z(\bG)$ and hence on $\wGF/(\GF \Z(\wGF))\cong \Z(\bG)_F$ we see that $\wGF_\chi$  corresponds to  the cyclic group $Z_0={\Z(\bG)}^{\spa{\gamma}}$ of order $2$ and $E(\GF)_\chi\leq \uE(\GF)$, see \Cref{not2_1}. Assumption (ii.2) implies that some $D$-invariant $\wGF$-conjugate $\chi''$ of $\chi$ extends to $\GF D$. By the information on $\wGF_\chi$ and $E(\GF)_\chi$ from above we see that $(\wGF E(\GF))_\chi$ is normalized by $\wGF$ and hence 
\[(\wGF \EGF)_{\chi''}=(\wGF \EGF)_{\chi}=\wGF_{\chi} \EGF_{\chi}= \wGF_{\chi''} \EGF_{\chi''}.\] 
As above we can argue that $\chi''$ extends to $\GF E(\GF)_{\chi''}$. Hence $\chi''$ can be chosen as $\chi_0$ and has all required properties. This concludes the proof. 
\end{proof}

Recall that $\III{m}$ is the integer with $\III{q=p^{m}}$. We have $\uE(\GF)\cong \Cy_{m}\times \Cy_2$ so  \Cref{propE4} shows that \Cref{Ainfty_D} holds for odd $m$. We therefore assume the following from now on.

\begin{hyp}\label{hyp76}
    We assume $\GF =\tDlsc(p^{m})$ with $p\neq 2$, $2\mid m$ and $l\geq 4$.
    In the following, $D$ denotes a non-cyclic 2-subgroup of  $\uE(\GF)$. Let $F_0\in \uE^+(\bG)$ be the element from \Cref{rem89} with $D=\spa{\restr F_0|{\GF},\gamma}$,  $F\in\spa{F_0^2}^+$ and  $[\Z(\bG),F_0]=[\Z(\bG),F]=1$.
\end{hyp}

In the rest of the paper we assume the above and study $\cE(\GF,\calC)$ for some $\calC\in\Cl_{\textrm{ss}}(\bH)^{\spa{ F_0,\gamma}}$.

\subsection{Proving \Cref{Ainfty_D} via counting}\label{5Bnew}

	 In order to establish the assumptions of \Cref{propE4} we use again a method seen in the proof of \Cref{prop_ext_wG}. 
	\begin{prop} \label{thm77}
	Assume \Cref{hyp76} on $\bG$, $F$, $F_0$, and $D$.
		\begin{thmlist}
			\item There exist exactly $\frac 1 2 \big(|\Irr(\GF)^{D}|+|\Irr(\bG^{F_0})^{\spann<\gamma>}|\big)$ characters in $\Irr(\GF)^D$ that have an extension to $\GF D$.
			\item If some subset $\II E0@{\EEnull}\subseteq \Irr(\GF)^{D}$ satisfies
			\begin{asslist}
				\item $ | \EEnull| =\frac 1 2 \big(|\Irr(\GF)^{D}| -  |\Irr(\bG^{F_0})^{\spann<\gamma>}|\big)$, and 
				\item no $\chi\in \EEnull$ extends to $\GF D$,
			\end{asslist}
			then every $\chi\in \Irr(\GF)^{D}\setminus \EEnull$ extends to $\GF D$. 
		\end{thmlist} 
	\end{prop}
 
\begin{proof} 
Let $\wh \EE\subseteq \Irr(\GF)^D$ be the subset of all irreducible characters that have no extension to $\GF D$. 

Define $F'_0=\restr F_0|{\GF}$.
Recall from the last three paragraphs of the proof of \Cref{prop_ext_wG}, the notions of central functions on cosets in the semi-direct product $\GF\rtimes\spa{F'_0}$, Shintani maps and their application to our situation. The automorphism $\gamma$ induces linear maps $$\gamma_1: \CF_{\GF\spann<F'_0>}(\GF F'_0) \lra \CF_{\GF\spann<F'_0>}(\GF F'_0)\text{ and }\gamma_0: \CF_{\GFnull}(\GFnull)\lra \CF_{\GFnull}(\GFnull)$$ while the Shintani norm map gives the commutative diagram \eqref{comdiag_Shint} seen in the proof of \Cref{prop_ext_wG} implying that \begin{equation}\label{Trga0}
\Tr(\gamma_1)=\Tr(\gamma_0)=|\Irr(\GFnull) ^{\spann<\gamma>}|.
\end{equation}

In order to compute $\Tr(\gamma_1)$ we define an extension map $\Lambda$ \wrt $\GF\unlhd \GF\spann<F'_0>$ for $\Irr(\GF)^{\spann<F'_0>}$, see end of \ref{not11}. 
We set  $\Lambda_{}(\chi)$ to be any extension of $\chi$ to $\GF\spa{F'_0}$, whenever  $\chi\in \Irr(\GF)^D\setminus\wh \EE$. Since such $\chi$ extends to $\GF D$ and $D$ is abelian, we see that $\Lambda(\chi)$ is $\gamma$-invariant. For $\chi\in \Irr(\GF)^{\spann<F'_0>} \setminus \Irr(\GF)^D$, we define  $\Lambda(\chi)$ and $\Lambda(\chi^\gamma)$ to be $\gamma$-conjugate extensions of $\chi$ and $\chi^\gamma$, respectively. 
Now $\Lambda$ so defined is $\gamma$-equivariant on $\Irr(\GF)^{\spa{F'_0}}\setminus\wh \EE$. 
If $\chi\in \wh \EE$, then $\chi$ has no extension to $\GF D$ and the character $\Lambda (\chi)$ is not $\gamma $-invariant. More precisely, $\Lambda_{}(\chi)^\gamma$ is obtained from $\Lambda(\chi)$ by multiplying with the linear character of $\GF\spann<F'_0>/\GF $ of order $2$.

Concerning $\CF_{\GF\spann<F'_0>}(\GF F'_0)$ and $\gamma_1$, the restrictions $\restr{\Lambda(\chi)}|{\GF F'_0}$ ($\chi\in \Irr(\GF)^{\spa{F'_0}}$) form a basis of the vector space $\CF_{\GF\spann<F'_0>}(\GF F'_0)$, see \cite[Prop. 3.1]{CS18B}. If $\chi \in \Irr(\GF)^D\setminus \wh \EE$, the class function $\restr\Lambda(\chi)|{\GF F'_0}$ is an eigenvector of $\gamma_1$ for the eigenvalue $1$. 
		If $\chi \in \wh \EE$, the function $\restr\Lambda(\chi)|{\GF F'_0}$ is an eigenvector of $\gamma_1$ for the eigenvalue $-1$. Otherwise the functions $\restr{\Lambda(\chi)}|{\GF F'_0}$ and 
		$\restr{\Lambda(\chi^\gamma)}|{\GF F'_0}$ 
		are permuted by $\gamma_1$ and $\gamma_1$ has accordingly the trace $0$ on the subspace they generate. We get $$\Tr(\gamma_1)=|\Irr(\GF)^D\setminus \wh \EE|- |\wh \EE|= |\Irr(\GF)^D|- 2|\wh \EE|,$$ see also the proof of \cite[Thm 3.8]{CS18B}. Combined with (\ref{Trga0}) this implies $|\Irr(\GF)^D|- 2|\wh \EE| =|\Irr(\GFnull)^{\spann<\gamma>}|$.

Because of $|\Irr(\GF)^D\setminus \wh \EE|= \frac 1 2 |\Irr(\GF)^D | + \frac 1 2 |\Irr(\GFnull)^{\spann<\gamma>}|$, we see that $|\wh \EE|= \frac1 2 |\Irr(\GF)^D |-\frac 1 2 |\Irr(\GFnull)^{\spann<\gamma>}|$ and that (exactly) $ \frac 1 2 {|\Irr(\GF)^D |} + \frac 1 2 |\Irr(\GFnull)^{\spann<\gamma>}|$ characters extend to $\GF D$. This proves the statement in (a).

Part (b) follows from part (a) if $\EE=\wh \EE$. The set $\EEnull$ satisfies $\EEnull=\wh \EE$ as $|\EEnull|=|\wh \EE|$ and $\EEnull\subseteq \wh \EE$ by the assumptions. This implies that every character in $\Irr(\GF)^{D}\setminus \EEnull$ has a $\gamma$-invariant extension to $\GF\spann<F'_0>$ and hence extends to $\GF\spann<F'_0,\gamma>=\GF D$.
\end{proof}
The assumptions of \Cref{thm77} are verified in the following by showing a reformulation in terms of geometric Lusztig series. Recall that $\Cl_{\textrm{ss}}(\bH)$ denotes the set of semisimple conjugacy classes of $\bH$. 

\begin{defi} \label{frakC}  We keep 
	 \Cref{hyp76}.
Let $\II{Cfrak}@{\frakC} :=\Cl_{\textrm{ss}}(\bH)^{\spann<  F_0{},\gamma>}$ be the set of $\calC \in \Cl_{\textrm{ss}}(\bH)$ with $F_0(\cC)=\gamma(\cC)=\cC$.

Let $\II{Cfrak0}@{\frakC_0}:=\{ \cC\in \frakC \mid \gamma(s)\in [s]_{\HFnull} \text{ for some } s\in \cC^{F_0} \}$
and  $\II{Cfrak1}@{\frakC_1}:=\frakC\setminus\frakC_0$. Note that then $\frakC_1=\{ \cC\in \frakC \mid \gamma(s)\notin [s]_{\HFnull} \text{ for every } s\in \cC^{F_0} \}$.
\end{defi}

\begin{cor}\label{cor79} Assume that for every $\calC\in \frakC$ there exists a set $\EEnull (\calC)\subseteq \cE(\GF,\calC)^{D}$ such that
	\begin{asslist}
		\item $|\EEnull (\calC)| =\frac{1} 2 |\cE(\GF,\calC)^{D}| - \frac 1 2 |\cE(\bG^{F_0},\calC)^{\spann<\gamma>}|$, and 
		\item no $\chi \in \EEnull (\calC)$ extends to $\GF D$.
	\end{asslist} 
Then every $\chi \in \Irr(\GF)^D\setminus \EEnull$ extends to $\GF D$, where $\EEnull:=\bigcup_{\calC\in\frakC} \EEnull(\calC)$
\end{cor}
\begin{proof}
The sets $\Irr(\GF)$ and $\Irr(\bG^{F_0})$ are disjoint unions of the geometric Lusztig series $\cE(\GF,\calC)$ and $\cE(\bG^{F_0},\calC_0)$, respectively, where $\calC\in\Cl_{\textrm{ss}}(\bH)$ is $F$-stable and $\calC_0\in\Cl_{\textrm{ss}}(\bH)$ is $F_0$-stable, see \ref{series}. 

If $\cE(\GF,\calC)^{D}\neq \emptyset$ then $\calC\in\frak C$ and even more $\calC$ has to contain some $\spa{F_0,\gamma}$-stable $\HF$-conjugacy class according to (\ref{rem_ntt}). 
Since we assume that $F\in \spann<F_0> ^+$ in $E(\bH)$, every $F_0$-stable class is $F$-stable. We have the following partitions of $\Irr(\GF)^{D}$ and $\Irr(\GFnull)^{\spann<\gamma>}$ via $\frakC$:
\begin{align*}
\Irr(\GF)^D=\bigsqcup_{\cC\in\mathfrak C} \cE(\GF,\ \cC)^{D} \und 
\Irr(\GFnull)^{\spann<\gamma>}=\bigsqcup_{\cC\in\frakC} \cE(\GFnull,\cC)^{\spann<\gamma>}.
\end{align*}
This implies 
$$|\Irr(\GF)^D|=\sum_{\cC\in \frakC} |\cE(\GF,\ \cC)^{D}| \und
|\Irr(\GFnull)^{\spa\gamma}|=\sum_{\cC\in \frakC} |\cE(\GFnull,\ \cC)^{\spa\gamma}|.$$
Note $\EEnull$ is the disjoint union of sets $\EEnull(\cC)$ ($\cC\in\mathfrak C $) and hence $|\EEnull|=\sum_{\cC\in\frakC} |\EEnull(\cC)|$. Then Assumption (i) implies 
\begin{align*}
			|\EEnull|&= \sum_{\cC\in\frakC} |\EEnull(\cC)|=
			\sum_{\cC\in \mathfrak C} \left( \frac 1 2 |\cE(\GF,\ \cC)^{D}|- \frac 1 2 |\cE(\GFnull,\ \cC)^{\spann<\gamma>}|\right)\\
			&= \frac 1 2 |\Irr(\GF)^{D}| -  \frac 1 2 |\Irr(\GFnull)^{\spann<\gamma>}|.
\end{align*}
This is Assumption \ref{thm77}(i).
By construction no character of $\EEnull$ has an extension to $\GF D$ and hence Assumption \ref{thm77}(ii) holds.  \Cref{thm77} then gives our claim.
\end{proof}

The remainder of the chapter is dedicated to finding the sets $\EE(\cC)\subseteq \cE(\GF,\cC)^D$ of \Cref{cor79}. .
In our investigations of $\cE(\GF,\calC)^D$ we distinguish the cases of $\cC\in \frak C_0$ or $\cC\in\frak C_1$, see \Cref{frakC}. In both cases $\EE(\cC)$ will be defined as a transversal of a set $\wt \EE(\cC)$ constructed in \Cref{lemfrakC1} and \Cref{thm_desc} for $\cC\in \frak C_1$ and $\cC\in \frak C_0$, respectively. We study the sets $\wt \EE(\cC)$ separately for the two cases but obtain a common list of properties that will allow us to construct some $\wGF$-transversal $\EE(\cC)$ of $\wt \EE(\cC)$, whose characters all have no extension to $\GF D$. 
    
\subsection{Comparing $\cE(\GF,\calC)^ {\spa{F_0,\gamma}}$ with $\cE(\bG^{F_0},\calC) ^{\spa{\gamma}}$ for $\calC\in \frakC_1$} \label{sec8B}
In this section we assume $\cC\in\frakC_1$, see \Cref{frakC}. We study $|\cE(\GF,\calC)^D|$ and $|\cE(\GFnull,\calC)^{\spann<\gamma>}|$ appearing in Assumption \ref{cor79}(i) and find a character set with the cardinality $\frac1 2 |\cE(\GF,\calC)^D| - \frac 1 2 |\cE(\GFnull,\calC) ^{\spann<\gamma>}|$.

\begin{lem}\label{lem5_11}
If $\cC\in\frakC_1$, then $\cE(\GFnull,\cC)^{\spa \gamma}=\emptyset$.
\end{lem}
\begin{proof}
Since no $\HFnull$-class contained in $\cC$ is $\gamma$-stable by the definition $\frakC_1$, (\ref{rem_ntt}) implies that no rational Lusztig series contained in $\cE(\GFnull,\cC)$ is $\gamma$-stable. This implies $\cE(\GFnull,\cC)^{\spa \gamma}=\emptyset$ by \eqref{SplitSer}. \end{proof}
So for later application of \Cref{cor79} we need a set $\EE(\calC)\subseteq \cE(\GF,\cC)^D$ with $\frac 1 2 |\cE(\GF,\cC)^D|$ elements. 

For the necessary study of the characters in $\cE(\GF,\cC)^D$ via Jordan decomposition the following statement about $\cC$ is key. Note that $\cC\in\frakC_1$ does not satisfy the assumptions made in \ref{ssec4C}, since by definition of $\frakC_1$ no $s\in\cC^{F_0}$ satisfies $\gamma(s)\in [s]_{\bH^{F_0}}$ as required in \Cref{not3_15}. Recall the simply connected covering $\pi\colon\bH_0\to\bH$ and the map $\omega_{s}\colon A_\bH(s)\to \Z(\bH_0)$ from \ref{ssclasses}.

\begin{prop}[{\cite[Cor.~3.2]{CS22}}]\label{lem511} Let $\cC\in\frakC_1$, $s\in\cC^{F}$ and $s_0\in\pi\inv(s) \subseteq \bH_0$ where $\pi$ is from (\ref{piH0}). Then there exist $F_0'\in\Cent_{\HF F_0}(s)$, $\gamma'\in\Cent_{\HF \gamma}(s)$ and an $\spa{F,F_0',\gamma'}$-stable subgroup $\wc A(s)$ of $\Cent_\bH(s)$ such that 
\begin{asslist}
\item $\Cent_\bH(s)= \Cent^\circ_\bH(s)\rtimes \wc A(s) $;
 \item $\omega_s(\wc A(s))=\spannhHnull$ and $|\AHs|=|\wc A(s)|=2$;
\item $[s_0,F'_0]\in \Z(\bH_0)\setminus \spannhHnull$, $[s_0,\gamma ']\in \spannhHnull$, and
\item $[F_0', \gamma']=\wc a$, where $\wc a$ is the generator of $\wc A(s)$.
\end{asslist}
\end{prop}
We describe properties of $\chi\in\cE(\GF,\cC)^D$ for example via applying the Jordan decompositions $\JJ$ and $\wt \JJ$ from \Cref{thm66} and \Cref{JorTilde} to $\chi$ and $\wt \chi\in \Irr(\wGF\mid \chi)$.
\begin{lem}\label{lemfrakC1}
Assume that Condition $\Ap$ holds for $\GF $. Let $\cC\in\frakC_1$ and $\chi\in\cE(\GF,\cC)^D$. Then:
\begin{thmlist}
\item $h_0\in\ker(\chi)$, and 
\item $\chi$ is not $\wGF$-invariant but stable under the diagonal automorphisms associated to $\spannh$ via (\ref{cZF}); in particular $\chi$ lies in a $\wGF$-orbit of length $2$ consisting of two $D$-invariant characters. 
\item Let $\wt \EE(\cC):=\cE(\GF,\cC)^D$ and $\II EC@{ \EE(\cC)}$ be a set of distinct representatives in $\wt \EE(\cC)$ for $\wGF$-conjugacy. Then 
\[ |\EE(\cC)|= \frac 1 2 | \cE(\GF,\cC)^D| -\frac 1 2 |\cE(\GFnull,\cC )^{\spa{\gamma}}|.\]
\end{thmlist}
\end{lem}
By part (b) the set $\cE(\GF,\cC)^D$ is a union of $\wGF$-orbits of length $2$ and is $\wGF$-stable as well.
\begin{proof} Let $\{\nu\}=\Irr(\restr \chi|{\Z(\GF)})$. Since $\chi$ is $\gamma$-stable, $\nu$ is $\gamma$-stable as well. Because of $[\Z(\GF),\gamma]= \spannh$, this implies $h_0\in\ker(\nu)\leq \ker(\chi)$ as claimed in (a). Let $s\in \cC^F$, $\wc A(s)$, $F_0'$ and $\gamma'$ be associated to $s$ as in \Cref{lem511}, i.e., $F_0'\in \HF F_0$, $\gamma'\in \HF \gamma$ with $F_0'(s)=\gamma'(s)=s$ and $\wc A(s)$ is $\spa{F_0',\gamma'}$-stable. Additionally $\wc a:=[F_0', \gamma']\in \wc A(s)\setminus \{1\}$ with $|\AHs|=|\wc A(s)|=2$ by \Cref{lem511}.

Let $\JJ:\Irr(\GF)\ra \Jor(\GF)$ be the Jordan decomposition from \Cref{thm66} and \Cref{cor_Jordandec2}. By our assumption Condition $\Ap$ holds and the Jordan decomposition $\JJ$ is $\calZ_F\rtimes \EGF$-equivariant, where $\calZ_F:=\wGF/(\GF\Z(\wGF))$, see (\ref{cZF}). Recall $\chi\in\cE(\GF,[s]_\HF)$ and hence $(s,\wt\phi)\in \JJ(\chi)$ for some  $\wt\phi\in\UCh(\Cent_\HF(s))$. Let $\phi\in \Irr(\restr \wt \phi|{\Cent^\circ _\bH(s)^F})\subseteq \UCh({\Cent^\circ _\bH(s)^F})$. 

The equivariance of $\JJ$ implies that the $\HF$-orbit of $(s,\wt \phi)$ is $\spa{F_0',\gamma'}$-stable. By construction $F_0'(s)=\gamma'(s)=s$ and hence $\wt \phi$ is $\spa{F_0', \gamma'}$-invariant. Recall that $\Cent_\bH(s)^F=\Cent_\bH^\circ(s)^F\rtimes \wc A(s)$. By Clifford theory, $\phi$ is therefore $aF_0'$- and $a'\gamma'$-invariant for some $a,a'\in \wc A(s)$. Now $[a F_0', a'\gamma']=[ F_0', \gamma']=\wc a$ since $F'_0$ and $\gamma'$ act trivially on $\wc A(s)$. Since $\wc A(s)=\spa{\wc a}$, this implies that $\phi$ is $\Cent_\HF(s)$-invariant and $ \phi=\restr \wt \phi|{\Cent^\circ_\bH(s)^F}$. As $\wt\phi$ is an extension of $\phi$, there exists some $g\in \wGF$ such that $\chi^g\neq \chi$, as $\JJ$ is $\calZ_F$-equivariant and $\calZ_F$ acts by multiplying with non-trivial characters of $\AHs^F$, see \Cref{lem3_7}. Since $\phi$ is $\wc A(s)$-invariant and $|\AHs|=2$, we have $|\wGF:\wGF_\chi|=|\AHs_\phi|=2$ and hence every $\wt \chi\in\Irr(\wG\mid \chi)$ satisfies $ \wt \chi(1)=2 \chi(1)$. As $\chi$ is $\gamma$-invariant the group $\wGF_\chi$ is $\gamma$-stable. This implies that $\chi$ is invariant under the diagonal automorphism associated to $h_0$ and hence $\chi$ is $\wh G$-invariant for $\wh G= (\bG/\spannh)^F$. This implies part (b).

For the equation in part (c) observe that  $|\cE(\GFnull,\cC)^{\spa{\gamma}} |=0$ by \Cref{lem5_11} and part (b) leads to  $|\EE(\cC)|=\frac 1 2 |\cE(\GF,\cC)^D |$. Taking this together we obtain the stated equality.  \end{proof}

Recall that ${\wt\pi: \wt \bH\ra \bH}$ denotes the natural epimorphism extending the epimorphism $\pi$ from (\ref{piH0}).

\begin{prop}\label{prop5_13}
In the situation of \Cref{lemfrakC1} let $\wt \chi\in\Irr(\wGF\mid \chi)$, $s_0\in\bH_0\cap \wt\pi\inv(s)$ and $\wt s\in \wt \pi\inv(s)\cap \wHF$ such that $(\wt s, \wc \phi)\in \wt \JJ(\wt \chi)$ for some $\wc\phi\in \UCh({\Cent _{\wt\bH}(\wt s)^F})$. Then there exists some $F_0'\in \Cent_{\HF F_0}(s)$ such that 
$F_0'.(\wt s,\wc\phi)=( F_0'( \wt s), \wc\phi)$ and $[s_0,F_0']\in\Z(\bH_0)\setminus \spa{ h_0^{(\bH_0)}}$. 
\end{prop}
\begin{proof} 
Let $\ov \chi:=\Pi_\wGF(\chi)$ and let $\ov \JJ$ be the map from \Cref{Jordandec}. Then $(s,\phi)\in \ov \JJ(\ov \chi)$ according to Theorem~\ref{thm66_c} with the notation from the proof of 
\Cref{lemfrakC1}. Recall $\wt \JJ$ the Jordan decomposition from \Cref{JorTilde}. By the definition of $\ov \JJ$ using $\wt \JJ$ we see that $(\wt s, \wc \phi)\in \wt \JJ(\wt \chi)$, where we obtain $\wc \phi$ from $\phi$ via the epimorphism $\Cent_\wbH(\wt s)^F\ra \Cent_\bH^\circ (s)^F$ induced by $\wt \pi$. We have seen in the proof of
\Cref{lemfrakC1} that $\phi$ and hence $\wc \phi$ are $\spa{F_0', \gamma'}$-invariant and $F_0'(s)=\gamma'(s)=s$. This implies that $F_0'.(\wt s, \wc \phi)=(F_0'(\wt s),\wc\phi)$. By construction there is some $z'\in \Z(\wbH)$ with $\wt s=s_0 z'$. Now according to \Cref{lem511} we have $[s_0,F_0']\in \Z(\bH_0)\setminus\spannhHnull$. This completes the proof.
\end{proof}
 \subsection{Comparing $\cE(\GF,\cC)^{\spa{F_0,\gamma}}$ with $\cE(\GFnull,\cC)^{\spa{\gamma}}$ for $\cC\in\frakC_0$ }
 \label{sec5C}
In the following we assume $\cC\in\frakC_0$ (see \Cref{frakC}). 
For a later application of \Cref{cor79} we investigate the character sets $\cE(\GF,\calC)^D$ and $\cE(\GFnull,\cC)^{\spa{\gamma}}$, their cardinalities and define a first candidate of $\EE(\cC)$ in \Cref{prop818} that satisfies Assumption \ref{cor79}(i) on the cardinality, namely  $|\EE(\cC)|=\frac 1 2 |\cE(\GF,\calC)^D| -\frac 1 2 |\cE(\GFnull,\cC)^{\spa{\gamma}}|$. 

The aim of this chapter is to define a set $\wt \EE(\cC)\subseteq \cE(\GF,\cC)^D$ together with a bijective map 
\[ \desc': \cE(\GF,\cC)^D\setminus \wt \EE(\cC) \lra 
\cE(\GFnull, \cC)^{\spa{\gamma}},\]
see \Cref{thm_desc}.
It turns out that $\wt \EE(\cC)$ is a union of $\wGF$-orbits of length $2$ and hence any $\wGF$-transversal $\EE(\cC)$ in $\wt \EE(\cC)$ has the required cardinality, see \Cref{prop818}. 

For the construction of $\desc'$ and $\wt \EE(\cC)$ various tools introduced earlier are combined: Like before characters of $\cE(\GF,\calC)^D$ are studied via $\ocE(\GF,\calC)^D$ using the map $\Pi_{\wGF}$ from \Cref{def31}, analogously we also apply $\Pi_{\wGFnull}$ giving a surjective map between $\cE(\GFnull,\calC)^{\spa{\gamma}}$ and  $\ocE(\GFnull,\calC)^{\spa{\gamma}}$. For a fixed element $s\in \cC^{F_0}$ with $\gamma$-stable $\HFnull$-conjugacy class we make use of the maps $\Gamma_{s,F,F_0}$ and $\Gamma_{s,F_0,F}$ from \Cref{labelEGFFnull}, which associate to a unipotent character an orbit sums. A third ingredient in the proof is the descent map from \Cref{propCSB}, which leads to the construction in \Cref{cor814}.

Many of our maps between character sets are not bijective and we investigate the cardinalities of their fibres.
\begin{notation}\label{not5_10_0}
We keep $\bG$, $F$, $F_0$, and $D$ satisfying \Cref{hyp76}.
\end{notation}
In the following we assume \Cref{hyp_cuspD_ext} for any $4\leq l'<l$,  $1\leq m'$, and consequently Condition $\Ap$ holds for $\GF =\tD_{l,\mathrm{sc}}(q)$, see \Cref{thm_typeD1}. Therefore 
a $\wGF$-orbit sum $\ov \chi$ is $D$-invariant if and only if $\ov \chi$ has some $D$-invariant constituent.
Analogously, Condition $\Ap$ holds for $\GFnull$ if $\GFnull$ is untwisted and according to \Cref{rem4x} and \Cref{thm41}, in case $\GFnull$ is of twisted type. Hence every $\gamma$-invariant $\wGFnull$-orbit sum has a $\gamma$-invariant constituent.
\begin{lem}\label{lem817}
Let $\chi\in\Irr(\GF)$ and $\ov \chi:=\Pi_{\wGF}(\chi)$. Assume \Cref{hyp76} for $\tD_{l,\mathrm{sc}}(q)$ and that \Cref{hyp_cuspD_ext} holds for all $l'$, $4\leq l'<l$. Then:
\begin{thmlist}
	\item If $\ov \chi$ is $D$-invariant, then 
	\begin{align*}
	|\Irr(\ov \chi)^D|= \begin{cases} 2 & \text{ if } |\Irr(\ov\chi)|=4 \und [\Z(\GF),D]\neq 1 ,\\
	|\Irr(\ov \chi)| &\text{ } \otw .
			\end{cases}\end{align*}
	\item Let $\chi_0\in\Irr(\GFnull)$ with a  $\gamma$-invariant $\wGFnull$-orbit sum $\ov \chi_0$. Then \begin{align*}
				|\Irr(\ov \chi_0)^{\spann<\gamma>}|= \begin{cases}
					2 & \text{ if }\ \  |\Irr(\ov\chi_0)|=4 \und [\Z(\GFnull),\gamma]\neq 1 ,\\
					|\Irr(\ov \chi_0)| &\text{ } \otw .
			\end{cases}
		\end{align*}
	\end{thmlist} 
\end{lem}
\begin{proof}
Part (a) follows from the fact that according to \Cref{thm_typeD1}, Condition $\Ap$ holds for $\GF$. Accordingly there exists some $E( \GF)$-stable $\wGF$-transversal in $\Irr(\GF)$. By Lemma 1.4 of \cite{TypeD1} there exists some $\wGF$-conjugate $\chi'$ of $\chi$ that satisfies $(\wGF \EGF)_{\chi'}=\wGF_{\chi'} E(\GF)_{\chi'}$. As $\ov \chi$ is $D$-invariant, this implies $E(\GF)_{\chi'}\geq D$.
		
		Recall $D=\spann<\restr F_0|{\GF},\gamma>\leq \uE(\GF)$ with $[F_0,\Z(\bG)]=1$. If $\chi '$ is $D$-invariant, the group $\wGF_{\chi '}$ is also $D$-stable. Via the isomorphism between $\wGF/(\GF \Z(\wGF))$ and $\Z(\bG)_F=\Z(\bG)$ from (\ref{cZF}), the $D$-stable group $\wGF_{\chi'}/(\GF \Z(\wGF))$ corresponds to $1$, $\Z(\bG)$ or $\spannh$. If we identify the quotient $(\wGF D)/(\GF\Z(\wGF))$ with $\Z(\GF)\rtimes D$, the group $(\wGF D)_{\chi'}/(\GF\Z(\wGF))$ corresponds to the group $D$,  $\Z(\GF)\rtimes D$ or $\spannh\rtimes D$. Using this we compute the number of $D$-invariant characters in $\Irr(\ov \chi)=\Irr(\ov {\chi '})$.
		If $|\Irr(\ov \chi)|=4$, then $(\wGF D)_{\chi'}= \GF\Z(\wGF) D$. 
		The $\wGF$-conjugates of $\GF\Z(\wGF) D$ are $\GF\Z(\wGF)D$ and $\GF\Z(\wGF)\spann<F_0,\gamma t_0>$, where $t_0\in\wGF$ induces on $\GF$ a diagonal automorphism associated to $h_0$ via (\ref{cZF}). This shows $|\Irr(\ov \chi)^D|=2$. In all the other cases the $\wGF$-conjugates of $(\wGF D)_{\chi'}/(\GF\Z(\wGF))$ still contain $D$. This implies part (a).

By the definition of $F_0$, the group $\GFnull$ can be of twisted or untwisted type. In all cases, \Cref{hyp_cuspD_ext} and \Cref{thm_typeD1} imply that Condition $\Ap$ holds for $\GFnull$. Hence considerations from above lead to the statement in part (b) for the untwisted case, while \Cref{thm41} can be used in the twisted case.  
\end{proof}

\begin{notation}\label{not5_10_1}
Let $\calC \in\frakC_0$, i.e. there exists some $\gamma$-stable $\HFnull$-class contained in $\cC$.
Let $s\in \calC^{F_0}$ with $\gamma(s)\in [s]_\HFnull$. If $|\AHs|=4$ then we can choose $s\in \pi(\bH_0^{F_0})$, see \cite[Thm B]{CS22}. According to \Cref{cor_2Becht} there exist $\gamma'\in \Cent_{\bH^{F_0} \gamma}(s)$ and some $\spann<F_0,\gamma'>$-stable $\wc A(s)\leq \bH$ with $\Cent_\bH(s)=\Cent_\bH^\circ (s)\rtimes \wc A(s)$. 
Note that $A:=\AHs$ satisfies $[A,F]=[A,F_0]=1$ since $[\Z(\bH_0),F]=1$, $[\Z(\bH_0),F_0]=1$, and the map $\omega_s: A\ra \Z(\bH_0)$ from \ref{ssclasses} is $\spann<\gamma',F_0>$-equivariant, see also \Cref{cor_CS22}. Let $\II A0@{A_0:=\Cent_{A}(\gamma')}$. 

As in \Cref{cor_CS22} and \Cref{noteps}, we denote by $v'_{a'}\in \wc A(s)$ an element representing $a'\in A$ and use an element $v_a\in\wc A(s)$ representing $a\in \AHs_F$. 
Since $[A,F]=1$, we have $v_a=v'_a$, see again \Cref{noteps}. 
\end{notation}

We construct here a candidate for $\EEnull(\calC)$ of \Cref{cor79} by investigating the sets $\cE(\GF,\calC)^{D}$ and $\cE(\bG^{F_0},\calC)^{\spa{\gamma}}$ with the given fixed $s\in \cC^{F_0}$. We compare the cardinality of those two sets of characters using \Cref{labelEGFFnull}. Recall that  $[F_0,\Z(\bG)]=[F,\Z(\bG)]=1$ by \Cref{hyp76}. With the abbreviation $A=\AHs$, we use $\{ v'_a=v_a \mid a\in A\}$, $\gamma '$ and $A_0:=\Cent_A(\gamma')$ from \Cref{not5_10_1}. Note that $A$ acts via the elements $v'_b$ on the group $\Cent_\bG^\circ (s)$ and $\UCh(\Cent^\circ_\bH(s)^{v_a F})$ and $\UCh(\Cent^\circ_\bH(s)^{v_a F_0})$, see also \Cref{rem3_6}. 

The sets $\rmUU(s,v_aF, v_b F_0)$ and $\rmUU(s,v_aF_0, v_b F)$ are defined as in \Cref{lem717} with $F_0$ as $\sigma$.   
\begin{prop} \label{prop5_12}
Let $\Gamma_{s,F,F_0}$ and $\Gamma_{s,F_0,F}$ be the surjective maps from \Cref{labelEGFFnull}. 
For $a\in A_0$ and $b\in A$ let $$\II MsFF0 @{M_{s,F,F_0}(a,b)^{[\gamma']}}:= \bigcup_{c\in A}\UCh(\Cent_\bH^\circ(s)^{v_aF})^{\spann<v_b F_0,v_c\gamma'>}\subseteq  \rmUU(s,v_aF, v_b F_0)\ \ \ \und\ \ $$
		
\[	\II MsF0F @{M_{s,F_0,F}(a,b)^{[\gamma']}}:= \bigcup_{c\in A}\UCh(\Cent_\bH^\circ(s)^{v_aF_0})^{\spann<v_b F,v_c\gamma'>}\subseteq \rmUU(s, v_aF_0, v_bF),\]
    see after \Cref{lem_whU} for the definition of the sets $\rmUU(s,F',\kappa)$.  

	As in \Cref{prop616} we set 
$$ \UU(s,F,F_0)^{[\gamma']} := \bigsqcup _{\substack{a\in A_0 \\ b\in A} } 
		M_{s,F,F_0}(a,b)^{[\gamma']}
\ \ 		\text{ and }\ \ 
		\UU(s,F_0,F)^{[\gamma']}
		:= \bigsqcup _{\substack{a\in A_0 \\ b\in A} } 
		M_{s,F_0,F}(a,b)^{[\gamma']}.$$
		
		Then
		\begin{thmlist}
			\item 
			 $\Gamma_{s,F,F_0}^{-1}(\ocE(\GF,(s))^{\spann<F_0,\gamma>})=
			\UU(s,F,F_0)^{[\gamma']}$ \ \ and
			\item 
			 $\Gamma_{s,F_0,F}^{-1} (\ocE(\GFnull,(s))^{\spa{\gamma}})= \UU(s,F_0,F)^{[\gamma']}$.
		\end{thmlist}
	\end{prop}
	\begin{proof}
		(a) follows from \Cref{prop616} by applying it with $F$ and $F_0$ as such and using $\gamma$ as there. \Cref{prop616} allows to exchange the roles of $F$ and $F_0$, so we also get part (b). 
	\end{proof}
	We apply the bijection from \Cref{propCSB} to those character sets, more precisely we relate the subset $M_{s,F,F_0}(a,b)^{[\gamma']}$ with $M_{s,F_0,F}(b,a)^{[\gamma']}$, where $(a,b)\in A_0\times A_0$. 
	The first set is only well-defined when $(a,b)\in A_0\times A$ and the second if $(b,a)\in A_0\times A$. Hence a bijection mapping $M_{s,F,F_0}(a,b)^{[\gamma']}$ to $M_{s,F_0,F}(b,a)^{[\gamma']}$ can only exist if $(a,b)\in A_0\times A_0$.
	This leads to the definition of the following disjoint unions $\UU'(s,F,F_0)^{[\gamma']}$ and ${{\UU'(s,F_0,F)}^{[\gamma']}}$.

	\begin{cor}\label{cor814} In the notation of \Cref{prop5_12} let 
		$$\II U'sFF0@{\UU'(s,F,F_0)^{[\gamma']}}
		:= \bigsqcup_{\substack{a\in A_0 \\ b\in A_0}} M_{s,F,F_0} (a,b)^{[\gamma']} \und 
		\II U'sF0F @{{\UU'(s,F_0,F)}^{[\gamma']}}
		:= \bigsqcup_{\substack{a\in A_0 \\ b\in A_0}} M_{s,F_0,F}(a,b) ^{[\gamma']}.$$ There exists a bijection \[ f:\UU'(s,F,F_0)^{[\gamma']} \lra {\UU'(s,F_0,F)}^{[\gamma']}\]
		such that $A_\phi=A_{f(\phi)}$ for every $\phi\in{\UU'(s,F,F_0)}^{[\gamma']}$.
	\end{cor}
Note that $\UU'(s,F,F_0)^{[\gamma']} \neq \UU(s,F,F_0)^{[\gamma']}$ in general whenever $A\neq A_0$ by definition. 
\begin{proof} Let $(a,b)\in A_0\times A_0$. According to \Cref{hyp76} that $F_0$ and $F$ act trivially on $\Z(\bG)$ and $A$. Hence the elements $v_a, v_b$ from \Cref{not5_10_1} are fixed by $F_0$ and $F$, see \Cref{cor_CS22}. 

	As $[v_aF,v_bF_0]=1$ in $\End(\Cent_\bH^\circ(s))$, there exists the bijection 
		\begin{equation*}
			f_{s,v_aF,v_bF_0}: M_{s,F,F_0}(a,b) \lra M_{s,F_0,F}(b,a) 
		\end{equation*} 
		from \Cref{propCSB}. Then $A_{\phi}= A_{\phi'}$ for every $\phi\in M_{s,F,F_0}(a,b)$ and $\phi':=f_{s,v_aF,v_bF_0}(\phi)$. 

		If $(a,b)\in A_0\times A_0$, then $v_a$ and $v_b$ are $\gamma'$-fixed by \Cref{cor_CS22}. 
		 Hence $f_{s,v_aF,v_bF_0}$ is $\spa{\gamma'}$-equivariant and induces a bijection 
		\[ f:{\UU'(s,F,F_0)}^{[\gamma']} \lra {\UU'(s,F_0,F)}^{[\gamma']}\] 
		with $ A_{\phi}= A_{f(\phi)}$ for every $\phi\in {\UU'(s,F,F_0)}^{[\gamma']}$. 
	\end{proof}
Note that by construction any character $\phi\in\UU(s,F,F_0)^{[\gamma']}$ lies in some $M_{s,F,F_0}(a,b)$ with $a\in A_0$ and $b\in A$. We determine the size of the fibres of $\Gamma_{s,F,F_0}$.  Recall that $A$ acts on elements of $\UCh(\Cent_\bH^\circ(s)^{v_a F})$ for every $a\in A$ according to \Cref{rem3_6}. We also use the map 
$\omega_s: A=\AHs \ra \Z(\bH_0)$ for $s\in \bH_{\textrm{ss}}$ from \ref{ssclasses}.
\begin{lem}\label{lem815}
 \begin{thmlist}
 \item Let $\ov\chi\in \ocE(\GF,(s))^{D}$, $a\in A_0$, $b\in A$ and $\phi\in\UCh(\Cent_\bH^\circ(s)^{v_a F})^{\spa{v_b F_0}}$, such that $\ov\chi= \Gamma_{s,F,F_0}^{}(\phi)$. 
	Then
	$$ |\Gamma_{s,F,F_0}^{-1}(\ov\chi)\cap {\UU'(s,F,F_0)}^{[\gamma']} |= \begin{cases}
				2&\text{ if } |A_\phi|=4,\\
				0&\text{ if } \omega_s(A_\phi)=\spann<h_0^{(\bH_0)}>\und b\in A\setminus A_0 ,\\
				|A|&\text{ } \otw .
			\end{cases}$$
	In particular if  $\ocE(\GF,(s))^{D}\neq \Gamma_{s,F,F_0}(  {\UU'(s,F,F_0)}^{[\gamma']} )$, then $|A|=4$.
 \item 
 Let $\phi\in \UU(s,F_0,F)^{[\gamma']}$ and $\ov \chi:=\Gamma_{s,F_0,F}(\phi)$. Then 
\[ |\Gamma_{s,F_0,F}^{-1}(\ov \chi)\cap {\UU'(s,F_0,F)}^{[\gamma']} |= \begin{cases}	2&\text{ if } |A_\phi|=4,\\
	|A|&\text{ } \otw .   \end{cases}\]
Moreover, $\ocE(\GFnull,(s))^{\spa{\gamma}}= \Gamma_{s,F_0,F}(\UU'(s,F_0,F)^{[\gamma']})$. 
\end{thmlist}
\end{lem}
\begin{proof}
For part (a) let $\chi\in \Irr(\ov \chi)^D$. Note that $|\Gamma_{s,F,F_0}^{-1}(\ov \chi)|=|A|$ according to \Cref{labelEGFFnull}(i) and $\Gamma_{s,F,F_0}^{-1}(\ov \chi)$ forms an $A$-orbit. By definition $b\in A$ satisfies $\phi^{v_b F_0}=\phi$. Additionally $\phi^{v_{b'} F_0}=\phi$ if and only if $b'\in bA_\phi$. As the $A$-orbit of $\phi$ is $\gamma'$-stable, $A_\phi$ is $\gamma'$-stable and hence $\omega_s(A_\phi)\in \{\{1\},\spannhHnull, \Z(\bH_0) \}$. 

The character $\phi$ appears $|(A_0)_{\phi}|$-times in $\UU'(s,F,F_0)^{[\gamma']}$ if $|bA_\phi \cap A_0|\neq 0$. In that case every $A$-conjugate of $\phi$ appears $|(A_0)_\phi|$-times in $\UU'(s,F,F_0)^{[\gamma']}$, as well. This implies the statement if $|A_\phi|=4$ or $|bA_\phi \cap A_0|\neq 0$.

If $bA_\phi\cap A_0=\emptyset$, then $\phi\notin \UU'(s,F,F_0)^{[\gamma']}$. This is only possible if $b\in A\setminus A_0$ and hence $|A|=4$. Now $\ov \chi^\gamma=\ov \chi$ implies $\phi^{v_c\gamma'} =\phi$ for some $c\in A$ and hence $\phi$ is also $[v_b F_0,v_c\gamma']$-invariant. Note that $v_{[b,\gamma']} =[v_b F_0,v_c\gamma']$ by \Cref{cor_CS22}. Because of $b\in A\setminus A_0$, we see $\omega_s([b,\gamma'])=h_0^{(\bH_0)}$. Hence $\omega_s(A_\phi)\geq \spa{h_0^{(\bH_0)}}$ or equivalently $A_\phi\geq A_0$. Together with $|bA_\phi \cap A_0|= 0$ this shows that $\Gamma_{s,F,F_0}^{-1}(\ov \chi) \cap \UU'(s,F,F_0)^{[\gamma']}=\emptyset$ is equivalent to $A_\phi=A_0$ and $b\in A\setminus A_0$.

For part (b) observe that any $\phi\in\Gamma^{-1}_{s_0,F_0,F}(\ov \chi ) \cap \UU(s,F_0,F) ^{[\gamma']}$ is a character of $\Cent_\bH^\circ(s)^{v_aF_0}$ and hence $v_aF_0$-invariant. Since $F\in \spa{F_0^2}$ by \Cref{hyp76}, we see $v_{a'}F\in \spa{(v_aF_0)^2}^+$ in $\End(\Cent_\bH^\circ(s))$ for some $a'\in\spann<a^2>\leq A_0$. This leads to $\phi\in \UU'(s,F_0,F)^{[\gamma']}$. 
This shows that $\Gamma_{s,F_0,F}^{-1}(\ov \chi)\cap {\UU'(s,F_0,F)}^{[\gamma']}\neq \emptyset$. 
Now we can compute $|\Gamma_{s,F_0,F}^{-1}(\ov \chi)\cap {\UU'(s,F_0,F)}^{[\gamma']}|$ with considerations similar to the above.  
\end{proof}
We summarize the above results and can finally compare $|\ocE(\GF,\cC)^D|$ with $|\ocE(\GFnull,\cC)^{\spann<\gamma>}|$. 
\begin{prop}\label{cor816}
Let $\II Eoverlines@{{\ovEE} (\cC)} :=\ocE(\GF,\cC)^D\setminus \Gamma_{s,F,F_0}({\UU'(s,F,F_0)}^{[\gamma']})$
and $\II Ugamma'exc@{\UU^{[\gamma']}_{exc}}:=\Gamma_{s,F,F_0}^{-1}(\ovEE(\cC))$.
 Then 
		\begin{align*}
			\UU^{[\gamma']}_{exc}&= \{\phi \in \UU(s,F,F_0)^{[\gamma']} \setminus \UU'(s,F,F_0)^{[\gamma']} \, \,
			\left | \, \, |A_\phi|=2 \right \}.
		\end{align*}
	In particular, each $\phi\in\UU^{[\gamma']}_{exc}$ satisfies $\phi\in M_{s,F,F_0}(a,b)$ for some $a\in A_0$, $b\in A\setminus A_0$.
	
Moreover, $\UU^{[\gamma']}_{exc}\neq\emptyset$ only if $|A| = 4$.
\end{prop}
Note that from the definition it is not clear that $\ov\EE(\calC)$ depends only on $\calC$, since it is defined using $s$ and other particular choices. The proof of \Cref{Ainfty_D} in \ref{ssec5D} will show that $\ov \EE(\cC)$ is the set of orbit sums in $\ov\cE(\GF,\cC)$, which have a $D$-invariant constituent that does not extend to $\GF D$. As such, it is uniquely defined by $\cC$. 
\begin{proof}[Proof of \Cref{cor816}]
\Cref{lem815} implies the inclusion
\begin{align*}
	\Gamma_{s,F,F_0}^{-1}(\ovEE(\calC))&\subseteq
	\left\{ \left. \phi \in \UU(s,F,F_0)^{[\gamma']} \setminus \UU'(s,F,F_0)^{[\gamma']}  \ 
	\right| \, |A_\phi|=2 \right\}.
\end{align*}
Conversely, let $\phi \in \UU(s,F,F_0)^{[\gamma']} \setminus \UU'(s,F,F_0)^{[\gamma']} $ with $|A_\phi|=2$. Then $\phi\in M_{s,F,F_0}(a,b)$ for some $(a,b)\in A_0 \times (A\setminus A_0)$. From \Cref{prop5_12} we see that $\ov \chi:= \Gamma_{s,F,F_0}(\phi)\in \ocE(\GF,\calC)$. The other elements of $\Gamma_{s,F,F_0}^{-1}(\ov \chi)$ are $A$-conjugates and are contained in $M_{s,F,F_0}(a,b')$ with $b'\in b A_\phi=b A_0$, see \Cref{cor3_13} and hence none is contained in $\UU'(s,F,F_0)^{[\gamma']}$. This proves $\Gamma_{s,F,F_0}(\phi)\in\ovEE(\cC)$.
	\end{proof}
In the following we deduce from the map $f$ defined in \Cref{cor814} a bijection between $ \ocE(\GFnull,\cC)^{\spann<\gamma>}$ and a subset of $ \ocE(\GF,\cC)^D$. The existence of this bijection is  remarkable. Clearly the maps $\Gamma_{s,F,F_0}$ and $\Gamma_{s,F_0,F}$ relate the unipotent $\wGF$-orbit sums with unipotent $\wGFnull$-orbit sums. The following diagram with non-injective maps can't be made commutative since there exist characters $ \phi,\phi' \in \UU'(s,F,F_0)$ with $\Gamma_{s,F,F_0}(\phi)=\Gamma_{s,F,F_0}(\phi')$ and $\Gamma_{s,F_0,F}(f(\phi))\neq \Gamma_{s,F_0,F}(f(\phi'))$. 
	\begin{align}
		\xymatrix{
			\UU'(s,F,F_0)^{[\gamma']} \ar[rr]^f \ar[d]_{\Gamma_{s,F,F_0}} 
			&& \UU'(s,F_0,F)^{[\gamma']} \ar[d]^{\Gamma_{s,F_0,F}}
			\\
			\ocE(\GF,\cC)^{D}&&\ocE(\GFnull,\cC)^{\spann<\gamma>}
		}
	\end{align}
	Recall that $\UU'(s,F,F_0)$ is a disjoint union and hence $\phi$ and $\phi'$ can be chosen to be the same character but seen as two different elements of $\UU'(s,F,F_0)$. As $\UU'(s,F,F_0)$ is a disjoint union we find some character $ \phi'$ that can be seen as an element of $M_{s,F,F_0}(a,b)$ and one of $M_{s,F,F_0}(a,b')$ for some $b,b'\in A$ with $b\neq b'$. 
	If we write $\phi_b \in M_{s,F,F_0}(a,b)$ and $\phi_{b'}\in M_{s,F,F_0}(a,b')$ for the character $\phi'$ in the specific subsets, then $f(\phi_b)$ and $f(\phi_{b'})$ are characters of two different groups $\bGos^{v_bF_0}$ and $\bGos^{v_{b'}F_0}$, and hence they are not contained in the same $A$-orbit. Accordingly, $\Gamma_{s,F_0,F}$ sends them to orbit sums in two different rational series.

Nevertheless, the maps $f$, $\Gamma_{s,F,F_0}$ and $\Gamma_{s,F_0,F}$ allow for the construction of two bijections, which we call ${\operatorname {desc}}$ and $\desc'$ because of their similarity with Shintani descent in relating representations of $\GF$ and $\bG^{F_0}$.

The following diagram relates the maps $f$, $\Gamma_{s,F,F_0}$ and $\Gamma_{s,F_0,F}$, as well as the maps $\operatorname{desc}$, $\operatorname{desc}'$ and the set $\wt\EE(\calC) $ introduced in the next theorem.

\begin{align}\label{diadesc}	
\xymatrix{
\UU(s,F,F_0)^{[\gamma']}\ar[d]_{\Gamma_{s,F,F_0}}& \supseteq&
\UU'(s,F,F_0)^{[\gamma']} \ar[d]_{\Gamma_{s,F,F_0}} \ar[r]^f_\sim&
		\UU'(s,F_0,F)^{[\gamma']}\ar[d]^{\Gamma_{s,F_0,F}}	\\
			\ocE(\GF,\calC)^{\spa{F_0,\gamma}} &\supseteq&
			\ocE(\GF,\calC)^{\spa{F_0,\gamma}} \setminus \ovEE(\calC) \ar@/^-2pc/@{.>}[r]^{\desc}_\sim		&
			\ocE(\GFnull,\calC)^{\spann<\gamma>}\\
			\cE(\GF,\calC)^{\spa{F_0,\gamma}} \ar[u]^{\Pi_{\wGF}}&\supseteq
            &\cE(\GF,\calC)^{\spa{F_0,\gamma}} \setminus \wt\EE(\calC) \ar@/^-2pc/@{.>}[r]^{\desc'}_\sim \ar[u]^{\Pi_{\wGF}}
            &\cE(\GFnull,\calC)^{\spann<\gamma>}\ar[u]_{\Pi_{\wGFnull}}
}
\end{align}

\begin{theorem}\label{thm_desc} 
    We keep $\bG$, $F$, $F_0$, and $D$ satisfying  \Cref{hyp76}.
\begin{thmlist}
\item For the set $\ovEE(\cC)$ from \Cref{cor816} there exists a bijection  
\[ \desc:\ocE(\GF,\cC)^D\setminus \ovEE(\cC)\stackrel\sim\lra \ocE(\GFnull,\cC)^{\spann<\gamma>}, \]
    such that 
    \[|\Irr(\ov \chi)|= |\Irr({\operatorname {desc}}(\ov \chi))|\forevery \ov\chi\in \ocE(\GF,\cC)^D\, \setminus \, \ovEE(\cC).\]
    \item Let $\wt \EE(\cC)=\bigcup_{\ov \chi\in\ov \EE(\cC)}\Irr(\ov \chi)^D$. Then there exists a bijection  
\[ \desc': \cE(\GF,\cC)^D\setminus \wt \EE(\cC) \stackrel{\sim}{\lra} \cE(\GFnull,\cC)^{\spa \gamma}.\] 
\end{thmlist}    
\end{theorem}

\begin{proof}
Before giving the proof of part (a) by constructing $\desc$ we first show how $\desc'$ can be obtained. Clearly,  $\desc$ yields a correspondence between $D$-stable  $\wGF$-orbits in $\cE(\GF,\cC) \setminus \wt \EE(\cC)$ and the $\gamma$-stable  $\wGFnull$-orbits in $\cE(\GFnull,\cC)$ such that corresponding orbits have the same length.  
Since we know that $\GF$ and $\GFnull$ satisfy Condition $\Ap$ through \Cref{hyp_cuspD_ext} for lower ranks by \Cref{thm_typeD1} and \Cref{thm41}, this leads to a bijection $\desc'$ as claimed: 
 If $\ov\chi\in \ocE(\GF,\cC)^D\setminus\ov\EE(\cC)$ and $\ov \chi_0=\desc(\ov\chi)$, the equality $|\Irr(\ov\chi)|=|\Irr(\ov\chi_0)|$ implies  \[ |\Irr(\ov \chi)^D|=|\Irr(\ov \chi_0)^{\spa{\gamma'}}|,\]
see \Cref{lem815}. Hence we can define $\desc'$ as a bijection mapping $\Irr(\ov \chi)^D$ to $\Irr(\ov\chi_0)^{\spa{\gamma}}$ where $\desc(\ov\chi)=\ov\chi_0$. 

Now we construct the bijection $\desc$.  If we set 
\[ \UU_{A'}:=\{\phi\in {\UU'(s,F,F_0)}^{[\gamma']} \mid A_\phi=A' \} \und \UU_{0,A'}:=\{\phi\in {\UU'(s,F_0,F)}^{[\gamma']} \mid A_\phi=A' \} \]
for every $\gamma$-stable $A'\leq A$, we see that 
\[ 
\UU'(s,F,F_0)^{[\gamma']}=\bigsqcup_{A'} \UU_{A'} \und 
\UU'(s,F_0,F)^{[\gamma']}= \bigsqcup_{A'}\UU_{0,A'},\]
where $A'$ runs over the $\gamma$-stable subgroups of $A$. By the properties of $\Gamma_{s,F,F_0}$ and $\Gamma_{s,F_0,F}$ we also have 
\[ \ocE(\GF,\calC)^D\setminus \ovEE(\calC)=\bigsqcup_{A'} \Gamma_{s,F,F_0}(\UU_{A'})   \und 
\ocE(\GFnull,\calC)^{\spa{\gamma'}}=\bigsqcup_{A'} \Gamma_{s,F_0,F}(\UU_{0,A'}), \]
where the sets $\Gamma_{s,F,F_0}(\UU_{A'})$ are disjoint for different groups $A'$ and analogously the sets $\Gamma_{s,F_0,F}(\UU_{0,A'})$ are disjoint.
The required bijection $\desc$ can be obtained by choosing bijections between $\Gamma_{s,F,F_0}(\UU_{A'})$ and $\Gamma_{s,F_0,F}(\UU_{0,A'})$ for each $A'$ provided  $|\Gamma_{s,F,F_0}(\UU_{A'})|=|\Gamma_{s,F_0,F}(\UU_{0,A'})|$ for all $\gamma'$-stable subgroups $A'\leq A$. 

Let us fix a $\gamma'$-stable $A'\leq A$ and let $f: {\UU'(s,F,F_0)}^{[\gamma']} \lra {\UU'(s,F_0,F)}^{[\gamma']}$ be the bijection from \Cref{cor814}. By its construction $f$ is $A$-equivariant on each set $M_{s,F,F_0}(a,b)$ for every $a\in A_0$ and $b\in A$. Every $\phi \in \UU_{A'}$ satisfies accordingly $A_{f(\phi)}=A'$. This implies 
\[f(\UU_{A'})=\UU_{0,A'}\] 
and hence $|\UU_{A'}|=|\UU_{0,A'}|$, as $f$ is bijective. The maps can be seen in the following diagram.
\begin{align*}
			\xymatrix{
				&\UU_{A'}\cap M_{s,F,F_0}(a,b) \ar@{^{(}->}[d]\ar[rr]^{f_{s,v_aF,v_bF_0}} && \UU_{0,A'}\cap M_{s,F_0,F}(b,a) \ar@{^{(}->}[d]\\
				\ocE(\GF,\calC)^D\setminus \ovEE(\calC)
				\ar@/_3pc/[rrrr]^{\desc}
				&
				\ar[l]_{\quad \Gamma_{s,F,F_0}}\UU_{A'} \ar[rr]^{\restr f|{\UU_{A'}}}&& \UU_{0,A'}\ar[r]^{\Gamma_{s,F_0,F}}&\ov\cE(\GFnull,\calC)^{\spann<\gamma'>}
			}
\end{align*}
Note that the lower part is in general not commutative, with possibly $\Gamma_{s,F_0,F}\circ \restr f|{\UU_{A'}}\neq \desc\circ \Gamma_{s,F,F_0}$.
Let $\mathfrak{a}:=|A'|$ if $A'\neq A$ and $2$ otherwise. By \Cref{cor816} 
$$|\Gamma_{s,F,F_0}^{-1}(\ov \chi )\cap \UU_{A'}|=\mathfrak{a} \forevery \ov \chi\in \Gamma_{s,F,F_0}(\UU_{A'})$$
and
\[|\Gamma_{s,F_0,F}^{-1}(\ov \chi_0)\cap \UU_{0,A'}|=\mathfrak{a} \forevery \ov \chi_0\in \Gamma_{s,F_0,F}(\UU_{0,A'}).\]
Hence
\[|\Gamma_{s,F,F_0}(\UU_{A'})|= \frac{|\UU_{A'}|}{\mathfrak{a}}= \frac{|\UU_{0,A'}|}{\mathfrak{a}}= |\Gamma_{s,F_0,F}(\UU_{0,A'})|.\]
We define $\desc$ on $\Gamma_{s,F,F_0}(\UU_{A'})$ as an arbitrary bijection between $\Gamma_{s,F,F_0}(\UU_{A'})$ and $\Gamma_{s,F_0,F}(\UU_{0,A'})$. That way we obtain the bijection $\desc$. 

It remains to verify $|\Irr(\ov \chi)|=|\Irr(\desc(\ov\chi)|$ for every $\ov \chi\in\ocE(\GF,\cC)^D\setminus\ov\EE(\cC)$. Given $\ov\chi$ there exists  some $\gamma'$-stable subgroup $A'\leq A$ such that $\ov\chi\in \Gamma_{s,F,F_0}(\UU_{A'})$. Then $\ov\chi_0:=\desc(\ov \chi)$ is contained in $\Gamma_{s,F_0,F}(\UU_{0,A'})$. Hence the characters $\ov \chi$ and $\ov\chi_0$ satisfy
\[ |\Irr(\ov \chi)|=\frac{|A|}{|A'|} \und |\Irr(\ov \chi_0)|=\frac{|A|}{|A'|},\]
see \Cref{labelEGFFnull}(ii).
\end{proof}
For a later application of \Cref{cor79} we construct from $\ovEE(\calC)$ a set having the cardinality required in \Cref{cor79}(i). 
\begin{prop} \label{prop818}
Assume \Cref{hyp_cuspD_ext} for all $4\leq l'<l$ and $1\leq m'$. Let $\II EEcalC@{\wt\EE(\calC)}:= \bigcup _{ \ov \chi\in \ovEE(\calC)} \Irr(\ov \chi)^D$ 
and $\EEnull(\calC)$ be a $\wGF$-transversal in $\wt\EE(\calC)$. Then
$$|\EEnull(\calC)|= \frac 1 2 |\cE(\GF,\calC)^D| -\frac 1 2 |\cE(\GFnull,\calC)^{\spann<\gamma>}|.$$
Note that $|\AHs|=4$ for every $s\in \cC$, whenever $\wt\EE(\calC)\neq \emptyset$.
\end{prop}
\begin{proof}
The bijection $$\desc': \cE(\GF,\calC)^D\setminus\wt\EE(\calC) \lra \cE(\GFnull,\calC)^{\spann<\gamma>}$$ 
from \Cref{thm_desc}(b) implies the equality  \begin{align*}
|\cE(\GF,\calC)^D\setminus\wt \EE(\calC)|&= 
|\cE(\GFnull,\calC)^{\spann<\gamma>}|
\end{align*} 
and hence 
\begin{align}\label{eq517}
|\cE(\GF,\calC)^D|-|\wt \EE(\calC)|&= |\cE(\GFnull,\calC)^{\spann<\gamma>}|.
\end{align}
Let $\ov \chi\in \ovEE(\calC)$ and $\phi \in\Gamma_{s,F,F_0}\inv(\ov \chi)$. Then $\omega_s(A_\phi)=\spann<h_0^{(\bH_0)}>$ according to \Cref{lem815}(a) and therefore  $|\Irr(\ov \chi)|=2$ by \Cref{labelEGFFnull}(i). Then both characters in $\Irr(\ov \chi)$ are $D$-stable. As $\EEnull(\calC)$ is a $\wGF$-transversal in $\wt\EE(\calC)$ we see $2|\EEnull(\calC)|=|\wt \EE(\calC)|$. In combination with \eqref{eq517} this proves the stated equality.

Whenever $\wt\EE(\calC)\neq \emptyset$ the set $\UU_{exc}^{[\gamma']}$ is also non-empty. This in turn implies $|A|=4$, see \Cref{cor816}.
\end{proof}

\subsection{Properties of $\wt \EE(\cC)$ for $\cC\in \frakC_0$}\label{5E}
Let us  keep $\bG$, $F$, $F_0$, and $D$ satisfying  \Cref{hyp76}.
We conclude the chapter with more properties of $\chi\in\wt\EE(\cC)$ that will lead to non-extendibility properties in the next chapter.
Recall that $\wt \JJ$ denotes the Jordan decomposition of $\Irr(\wGF)$ from \Cref{JorTilde}. The following is an analogue of \Cref{prop5_13} for $\frakC_0$.
\begin{prop}\label{prop522}
Let $\cC\in\frakC_0$ and $\chi\in\wt \EE(\cC)$, see \Cref{prop818}. Then there exists some $\si'\in \HF F_0\subseteq \HF \uE(\bH)$ and some $\wt \chi\in\Irr(\wGF\mid \chi)$ such that 
	\begin{thmlist}
		\item $h_0\in\ker(\chi)$;
		\item $\restr \wt\chi|{\GF}\neq \chi$ and $\chi$ is invariant under diagonal automorphisms associated to $\spannh$ via (\ref{cZF});
		\item some $(\wt s,\varphi)\in \wt \JJ(\wt \chi)$, $\wt s \in \wt \bH^F$, $s_0\in \bH_0\cap \wt s\Z(\wt \bH)$ and $z'\in \Z(\wbH)$ satisfy $\si '.(\wt s, \varphi)=( \wt s z', \varphi)$ and $[s_0,\si ']\in\Z(\bH_0)\setminus\spa{ h_0^{(\bH_0)}}$. 
	\end{thmlist}
\end{prop}
For what follows we need to fix some more notation. 
\begin{notation}\label{not5_24}
Let $\cC\in \frakC_0$, $\chi\in\wt \EE(\calC)$ and $\ov\chi:=\Pi_\wGF(\chi)\in\ovEE(\calC)$. According to the definition of $\ovEE(\calC)$ in \Cref{cor816}, we see $\ov \chi\in \ov \cE(\GF,\calC)^D$, hence $\ov \chi$ lies in the image of the map $\Gamma_{s,F,F_0}$ from \Cref{labelEGFFnull}, see also \Cref{prop5_12}. 
Recall that $\wt\EE(\calC)\neq \emptyset $ implies $|\AHs|=4$ for every $s\in \cC$ according to \Cref{prop818}.

The set $\wt \EE(\cC)$ is defined after choosing an element $s\in \cC^{F_0}$ in \Cref{not5_10_1}. According to \cite[Thm~B]{CS22} there exists some $s'\in \cC\cap \pi(\bH_0^{F_0})$ with $\gamma(s')\in [s']_\HFnull$, where $\pi: \bH_0 \ra \bH$ is from (\ref{piH0}). We assume $s$ to have those properties.
\end{notation}

The first two properties of $\chi$ given in \Cref{prop522} imply that every $\wGF$-conjugate $\chi'$ of $\chi\in\wt\EE(\calC)$ is $D$-invariant by the structure of $\Out(\GF)$ and hence $\chi'\in\wt\EE(\calC)$.
\begin{proof}[Proof of \Cref{prop522}]
	Let $\ov \chi=\Pi_{\wGF}(\chi)$ and $\phi\in \Gamma^{-1}_{s,F,F_0}(\ov \chi)$. Notice that $\ov \chi\in \ovEE(\calC)$. By the definition of $\wt\EE(\calC)$ we observe that $\phi\in \UU^{[\gamma']}_{exc}$ (see \Cref{cor816}) and hence by \Cref{lem815} the character $\phi$ satisfies $\omega_s(A_\phi)=\spannhHnull$. Hence $|\wGF:\wGF_\chi|=2$. We can assume that Condition $\Ap$ holds and then some character $\chi'\in \Irr(\ov \chi)$ is $D$-invariant. As $\chi'$ is $\gamma$-invariant, the group $\wGF_{\chi'}$ is $\gamma$-stable and hence the diagonal automorphisms stabilizing $\chi'$ correspond to $\spannh$. This is part (b). 
	
	Since $\Out(\GF)_\chi \unlhd \Out(\GF)$ (unless $\bG$ is of type $\tD_4$), we see that every character in $\Irr(\ov \chi)$ is also $D$-invariant. When $\bG$ is of type $\tD_4$, an easy calculation in ${\mathcal{Z}}_F\rtimes E(\bG^F)$ shows that any $\wGF$-conjugate of $\chi $, for instance $\chi '$, is $D$-invariant as well.

Recall that $|\Z(\GF)|=4$ and note that by the assumption on $D$ we have $[\Z(\GF),D]=\spannh$. Let $\{\nu \}=\Irr(\restr \chi|{\Z(\GF)})$. As $\chi$ and hence $\nu$ is $D$-invariant, we have $\nu ([\Z(\GF),D] )=1$ and therefore part (a).

We now focus on the verification of part (c). 
Recall that as pointed in \Cref{not5_10_1} some $s_0'\in \bH_0^F$ exists with $\pi(s_0')=s$. For further considerations we again make use of the group  $\wc A(s)\leq \bGs$ and $\gamma'\in \Cent_\HFnull(s)\gamma$ associated to $s$ from \Cref{cor_2Becht} fixed already in \Cref{not5_10_1}. This group is by construction $\spa{F_0,\gamma'}$-stable and hence it satisfies 
$\Cent_\bH(s)=\Cent^\circ_{\bH}(s) \rtimes \wc A(s)$.  

Recall $A=\AHs$ and $A_0:=\Cent_A(\gamma')$. 
In the following we denote by $v_a$ the element of $\wc A(s)$ corresponding to $a\in A$.  Then there exist $a\in A_0$ and $b\in A\setminus A_0$ with $\phi\in M_{s,F,F_0}(a,b)$ and $|A_\phi|=2$ by \Cref{cor816}. Hence the $\wc A(s)$-orbit of $\phi$ contains two characters and 
$$\phi\in\UCh(\Cent^\circ_\bH(s)^{v_a F})^{\spa{v_b F_0}}.$$

Let $\phi'\in \oUCh(\Cent^\circ_\bH(s)^{v_a F})$ be the $\wc A(s)$-orbit sum, see \Cref{not3_4}. 
This leads to $(s, \phi') \in \ov \JJ(\ov \chi)$ for the Jordan decomposition $\ov \JJ$ from \Cref{Jordandec} if $a=1$.

Let $s_a\in \cC$ and $\iota_a$ be defined as in \Cref{cor_CS22} and \Cref{78} using the element $g_a\in \bH$ with $F^{g_a}=v_a F$ from \Cref{noteps}.
Via the morphism $\iota_a$ the character $\phi$ defines $\phi_a\in \UCh(\Cent^\circ_\bH(s_a)^F)$. Then we write $\phi'_a$ for its $\Cent_\bH( s_a)^F$-orbit sum. By this construction the $\HF$-orbit of $(s_a,\phi'_a)$ belongs to $\ov\Jor(\GF)$ and $(s_a,\phi'_a)\in \ov \JJ(\ov\chi)$. 

Let $\III{v_{a,0}\in \pi\inv(v_a)}$, $\III{v_{b,0}\in \pi\inv(v_b)}$ and $g_{a,0}\in \pi\inv (g_a)$ such that $F^{g_{a,0}}=v_{a,0} F$. 
Further let $e$ be some integer with $(v_{a,0} F)^e= F^e$ and consider $\wh E:=\spa{\restr F_p|{\bH^{F^e}},\gamma}$ from \Cref{78} as subgroup of $\Aut(\bH^{F^e})$. Let 
$$\II{iotaa0}@{\iota_{a,0}: \wbH \rtimes \wh E \lra \wbH \rtimes \wh E 
	\text{ given by }x\mapsto x^{g_{a,0}}},$$ 
$\III{F'_{0,a}:=\iota_a\inv(F_0)}$ and 
$\II{gammaa}@{\gamma_{a}:=\iota_a\inv(\gamma')}$. We set $\III{s_{a,0}:=\iota_{a,0}\inv(s_0')}$ and 
$\II {stildea}@{\wt s_a}\in s_{a,0}\Z(\wbH)\cap \wbH^F$. Then $\wt \pi(\wt s_{a})=s_a$ and hence 
$\Cent^\circ_\bH( s_{a})=\wt\pi (\Cent_\wbH(\wt s_{a}))$.
Therefore $\phi_a$ defines a character $\wc \phi_a$ of $\UCh(\Cent_\wbH(\wt s_{a})^F)$. The $\wt\bH^F$-orbit of $(\wt s_a, \wc\phi_a)$ is contained in $\Jor(\wGF)$ and $\pi^*(\ov{(\wt s_a, \wc\phi_a)})= \ov{(s_a,\phi'_a)}$ for the map $\pi^*$ defined before \Cref{Jordandec}. Let $\wt \chi\in\cE(\wbG^F,[\wt s_a])$ be the character corresponding to $ (\wt s_a,\wc \phi_a)$ via $\wt \JJ$. Then $\wt \chi\in \Irr(\wGF\mid \chi)$, see \Cref{Jordandec}.

Let $\si':= \iota_a\inv ( v_{b,0} F_0 )$. Note that $F_0(v_b)=v_b$ as $\wc A(s)$ is $F_0$-stable and $F_0$ acts trivially on $\Z(\bH_0)$ and hence on $\AHs$. According to Lemma~\ref{lem3_15d} we see that $\si'\in \HF F_0$.  Additionally we see that $\si'(\wt s_a)\in \wt s_a\Z(\wbH_0)$. As $\phi$ is $v_b F_0$-invariant, the character $\phi_a$ is $\si'$-stable. Note that $\si'\in \HF F_0$. This leads to $\si'.(\wt s_a,\wc\phi_a)=(\si'(\wt s_a),\wc\phi_a)$. For $\wt s:=\wt s_a$, $\varphi:=\phi_a$ and $s_0:=s_{a,0}$  we have $(\wt s,\varphi)\in \wt \JJ(\wt\chi)$ such that $\si'.(\wt s,\varphi)=(\si'(\wt s),\varphi)$. Since by the construction $s$ is $\si'$-fixed, $\wt s$ satisfies $\si'(\wt s)\in \Z(\wt \bH^F) \wt s$.

Finally we have to compute $[\si', s_{a,0}]$, namely
\begin{align*}
 \si'(s_{a,0})= 
 \iota_a\inv \big( ^{v_{b,0}}(F_0 (s_{0})) \big)= \iota_a\inv ( ^{v_{b,0}}(s_{0}) )=\iota_a\inv ( \omega_s(v_{b}) s_0)= \omega_s(v_{b} )
 \iota_a\inv ( s_0)=\omega_s(v_{b} ) s_{a,0}.
\end{align*}
We see that $[\si',s_{a,0}]=\omega_s(v_b ) $. Recall $b\in A\setminus A_0$ and $\omega_s $ is $\gamma'$-equivariant. Hence $[\si',s_{a,0}]\in \Z(\bH_0)\setminus\spannhHnull$. This proves our last claim since $[ s_{0},\si']=[ s_{a,0},\si']=[\si', s_{a,0}]\inv$. 
\end{proof}
\section{The extendibility part of Condition $\Ai$}\label{SecEC}
We continue our considerations towards a proof of \Cref{Ainfty_D}. Hence, the next two sections are in the same framework as the preceding chapter, in particular $\GF=\tD_{l,\textrm{sc}}(p^{m})$ for some odd prime $p$, $m\geq 1$ and $l\geq 4$. We also keep $F_0$, and $D$ satisfying  \Cref{hyp76}.

In the last step of the proof of \Cref{Ainfty_D} we now ensure that the subsets $\EE(\cC)$ of $\cE(\GF,\cC)^D$ introduced in \ref{sec8B} and \ref{5E} (see \Cref{lemfrakC1} and \Cref{prop818}) can be chosen to contain no character extending to $\GF D$. This is derived from the properties established in \Cref{prop5_13} and \Cref{prop522}.  
Illustrating an elementary alternative discussed in \Cref{prop820}, we can establish in \Cref{cor822} that the sets $\EE(\cC)$ satisfy the requirements introduced in \Cref{cor79} and fulfil the desired condition of not extending to $\GF D$.  The intermediate quotient $\bG/\spa{h_0}=\SO_{2l}(\FF)$ and how it inserts in the duality between $\bG$ and $\bH$ plays an important part. 

This allows us in \ref{ssec5D} to conclude our proof of \Cref{Ainfty_D} and the main theorems of the paper.

\subsection{Characters extending to $\SO_{2l}(q) $}\label{sec8C}

Our goal here is  \Cref{prop821} showing a non-extendibility in the situation of \Cref{prop522}. \Cref{lem520} is a slight generalization of  \cite[Prop. 2.18]{TypeD1} that leads naturally to a discussion around the duality between $\bG$ and $\bH$. 

The following general result on the extension of characters will be key and motivates the considerations afterwards. 
\begin{lem}\label{prop820}
Let $X\unlhd \wh X$ be finite groups with natural epimorphism $\eps: \wh X\lra \wh X/X$ and $c , f, d \in \wh X$ such that 
		\begin{itemize} 
			\item $ \eps(c)$ and $\eps(d)$ are involutions of $\wh X/X$,
			\item $\wh X/X$ is an abelian $2$-group, and 
			\item $\wh X/X=\spa{\eps(d)}\times \spa{\eps(f)}\times \spa{\eps(c)}$.
		 \end{itemize} 
		Let $\chi\in \Irr(X)^{\wh X}$ such that for some (hence all) $\delta\in\Irr(\chi^{\spann<X,d>} )$ one has $\delta^f\neq \delta$. 
		Then one of the following holds: 
		\begin{asslist} 
			\item $\chi$ extends to $\spann<X,f,dc>$ and not to $\spann<X,f,c>$, or 
			\item $\chi$ extends to $\spann<X,f,c>$ and not to $\spann<X,f,dc>$.
		\end{asslist}
	\end{lem}
	\begin{proof}
		We consider an extension $\phi$ of $\chi$ to $\spann<X,f>$. Let $\nu\in \Irr(\spann<X,f>)$ be the linear character of order $2$ with $\ker(\nu)\geq X$. 
		Since $\delta$ is not $f$-invariant and $\Irr(\chi^{\spann<X,d>} )$ has two elements, $\Irr(\chi^{\spann<X,d>} )^{\spa{f}}=\emptyset$ and $\chi$ does not extend to $\spann<X,d,f>$. 
		Hence $\phi\neq \phi^d=\phi\nu$.

		If $\phi$ is not $c$-invariant, $\phi^c=\phi\nu$ as $\eps(c)$ is an involution. Then $\phi$ and $\chi$ have no extension to $\spa{X,f,c}$. This shows $$\phi^{dc}=(\phi^d)^c=(\phi\nu)^c=\phi^c\nu=\phi\nu \nu=\phi.$$ 
		Accordingly $\phi$ extends to $\spann<X,f,dc>$ as $\spann<X,f,dc>/\spann<X,f>$ is cyclic. In this case $\chi$ satisfies the statement in (i). 

		If $\phi$ is $c$-invariant, then $\chi$ extends to $\spann<X,f,c>$. On the other hand $\phi$ is not $dc$-invariant since $\phi^{dc}=(\phi^d)^c=(\phi \nu)^c=\phi \nu$. Accordingly $\chi$ satisfies the statement in (ii). 
	\end{proof}
	Recall $ {{Z_0}}:=\langle{h_0^{(\bG)}}\rangle =\Z(\bG)^{\spa{\gamma}  }    $ (see \Cref{not2_1}), let $\II Ghat@{\protect{\widehat \bG}}:=\bG/Z_0$, $\II Ghat@{\protect{\wh G}}:=(\bG/Z_0)^F$ and note that $D$ acts on $\wh G$ by definition. Then $X :=\GF/Z_0$ and $\wh X:= \wh G \rtimes D$ satisfy the group-theoretic assumptions of the above statement with $d\in \wh G\setminus X$, $f:=\restr F_0|{\wh G}$ and $c:=\gamma$ (all seen as automorphisms of $\wh G$). 
	\Cref{prop820} can be applied if some and hence every extension $\wh \chi$ to $\wh G$ of the considered $\wh G$-invariant character $\chi$  satisfies $\wh \chi^{F_0}\neq \wh\chi$. (Then $\wh\chi$ can serve as $\delta$ in the application of \Cref{prop820}.) The following proposition gives a way to prove this property. 
	
Recall that $\wt \JJ$ denotes the Jordan decomposition for $\Irr(\wGF)$ from \Cref{JorTilde}, with $\wt \bH:=\wbG^*\geq \bH_0=[\wt\bH,\wt\bH]$. Let us recall ${\Z(\bH_0)}^{\spa{\gamma}}= \langle{h_0^{(\bH_0)}}\rangle$, see again \Cref{not2_1}.
\begin{prop}\label{prop821}
Let $\si'\in E(\bG)$, $\chi\in\Irr(\GF)^{\spa{\si'}}$ and some $\wt \chi\in\Irr(\chi^{\wGF} )$. Assume 
\begin{asslist}
\item $Z_0\leq\ker(\chi)$;
\item $\restr \wt\chi|{\GF}\neq \chi$ and $\chi$ is $\wh\bG^F$-invariant;
\item there exist some $(\wt s,\phi)\in \wt \JJ(\wt \chi)$, $s_0\in \bH_0\cap \wt s\Z(\wbH)$, $\si \in\HF(\si ')^*\subseteq \HF E(\bH)$ (hence acting on $\wt\bG$, see \ref{not}, and therefore $\wt\bH =\wt\bG^*$) with $\si.(\wt s, \phi)=( \wt s z, \phi)$ for some $z\in \Z(\wt\bH)$ and $[s_0,\si]\in\Z(\bH_0)\setminus \langle{h_0^{(\bH_0)}}\rangle$. 
\end{asslist}
Then no $\wh \chi\in{\Irr}(\wh\bG^F)$ with $\wh \chi(x Z_0)=\chi(x)$ for all $x\in \GF$, is  $\si'$-invariant.
\end{prop}
The proof of this statement is completed at the end of the present chapter. We show first how characters of $\wh \bG^F$ can be studied via characters of $\wGF$ and our extendibility question reduced to checking certain linear characters. 
 Recall that $\calL\colon \wt\bG\to\wt\bG$ is the Lang map $x\mapsto x\inv F(x)$.

\begin{lem}\label{lem520} Let $\tau\colon \wt\bG\to\wt\bG$ be a bijective algebraic endomorphism commuting with $F$.
Let $\bK\leq \bG$ be a closed reductive $\spann<F,\tau>$-stable subgroup containing a maximal torus (so $\Z(\bG)\leq\bK$). Let $\wt\bK:=\bK \Z(\wt \bG)$, $\bK^F\unlhd\wh K:=\calL\inv(Z_0)\cap\bK$,  and $\wh Z:=\calL\inv (Z_0)\cap \Z(\wbG)$ (see \Cref{not2_1}). Let $\psi\in\Irr(\bK^F\mid 1_{Z_0})^{\spann<\wh\bK^F,\tau>}$ and $\wt\psi\in\Irr(\wKF\mid \psi)$. Let $\mu\in\Lin(\wt \bK^F/\bK^F)$ such that $\wt \psi^{\tau}=\wt \psi \mu$, and $\{\nu \}=\Irr( \restr \wt \psi|{\Z(\wGF)})$. Then the following are equivalent:
\begin{asslist}
		\item there exists some $\wh \psi \in \Irr(\wh \bK^F)^{\spa{\tau}}$ with $\restr\wh\psi|{\bK^F}=\psi$ ;
		\item there exists some $\wh \nu \in\Irr(\wh Z\mid \nu) $ such that $\mu(tx)=\wh \nu^{\tau} (x)\wh \nu(x^{-1})$ for every $t\in \wh K$ and $x\in \wh Z$ with $tx\in \wKF_\psi$.
\end{asslist}
\end{lem}
\begin{proof}
The statement is a slight variation of  \cite[Prop. 2.18]{TypeD1}. The proof is the same.
\end{proof}

When applying \Cref{lem520} to verify \Cref{prop821}, the characters $\mu$ and $\nu$ can be computed locally from characters of tori. More precisely in the following \Cref{lem5_25} we see that $\mu$ and $\nu$ can be computed using some $\wt\theta\in\Irr(\wt\bT^F)$, where $\wt\bT$ is an $F$-stable maximal torus of $\wt\bG$. Again $\theta$ defines a $\wh \bT^F$-invariant character $\ov \theta$ on $\bT^F/Z_0$, where $\wh\bT:=\bT/Z_0$. 
According to \Cref{prop5_24}, a statement on duality, we then get that $\theta$ has no $\si'$-invariant extension $\wh \theta\in \Irr(\wh \bT^F)$, thus allowing to use \Cref{lem520}.
\begin{align*}
	\xymatrix{\wh \chi &\wh \bG^F&&&\wGF\ar@{-}[dl] &\wt \chi \\
		&\ov \chi& \GF/Z_0\ar@{-}[ul]& \GF\ar[l] & \chi}
\end{align*}

\begin{align*}
	\xymatrix{
		\wh \theta & \wh\bT^F\ar@{-}[dr] &&& \wt\bT^F\ar@{-}[dl]& \wt \theta \\
		&\ov \theta&\bT^F/Z_0& \bT^F\ar[l] &\theta }
\end{align*}

For the study of the character $\theta\in\Irr(\bT^F)$ we use the following statement on duality. 
Recall that for a given torus $\bS$ we denote by $\II XS@{X(\bS)}$ its character lattice and by $\II YS@{Y(\bS)}$ its cocharacter lattice. 

As before we denote by ${h_0^{(\bH_0)}}$, the unique non-trivial $\gamma$-fixed element of $\Z(\bH_0)$. We recall $Z^\diamond :=\langle h_0^{(\bH_0)}\rangle \leq \Z(\bH_0)$ already introduced in the proof of \Cref{prop_nontrivFF0action}.

For pairs of groups in duality, like $(\bG,\bH)$, $(\wbG,\wbH)$ and below $(\bG/Z_0,\bH_0/Z^\diamond)$, we recall, in the case of say  $(\bG,\bH)$ the set $\ov{\frak X}(\bG,F)$ of $\GF$-conjugacy classes of pairs $(\bT,\theta)$ where $\bT$ is an $F$-stable maximal torus of $ \bG$ and $\theta\in\Irr(\bT^F)$ (see \cite[2.3.20]{GM}), and the set $\ov{\frak Y}(\bH,F)$ of $\HF$-conjugacy classes of pairs $(\bT^*,s)$ where $\bT^*$ is an $F$-stable maximal torus of $\bG^*=\bH$ and $s\in \bT^*{}^F$ (see \cite[2.5.12]{GM}). The relation $(\bT,\theta)\stackrel{\bG}{\longleftrightarrow}(\bT^*,s)$ recalled in (\ref{Ttheta}) is defined through the common parametrization of $F$-stable maximal tori in $\bG$ and $\bH$ by $F$-conjugacy classes of the Weyl group and establishes a bijection $\ov{\frak X}(\bG,F)\longleftrightarrow\ov{\frak Y}(\bH,F)$, see \cite[Cor.~2.5.14]{GM} or \cite[Sect. 2.6]{CS13}.
\begin{prop}\label{prop5_24} Recall $\wh \bG:=\bG/Z_0$ and set $\wh \bH:=\bH_0/Z^\diamond$. Let $\wh \pi:\wh \bH\lra \bH$ be the canonical epimorphism. 
	\begin{thmlist}
		\item For any $\rho\in E(\bG)$ with associated $\rho^*\in E(\bH)$ as in (\ref{sigma*}) and any relation $(\wbT,\theta)\stackrel{\wbG}{\longleftrightarrow}(\wbT^*,\wt t)$ as in (\ref{Ttheta}), we have $(\wbT,\theta)^\rho\stackrel{\wbG}{\longleftrightarrow}\,{}^{\rho^*}\!(\wbT^*,\wt t)$. 
		\item Recalling the duality between $(\bG,\bT_0,F) $ and $(\bH,\bS_0,F)$ given by $\delta:X(\bT_0)\lra Y(\bS_0)$ from \ref{duals}, there is a duality between $(\wh \bG, \wh\bT_0,F) $ and $(\wh \bH,\wh\bS_0,F)$ given by some $\wh \delta:X(\wh \bT_0)\lra Y(\wh \bS_0)$ for$\wh\bT_0=\bT_0/Z_0$ and $ \wh \bS_0=\wh\pi\inv(\bS_0)$ such that 
        \[
        \delta\circ \alpha=\beta^*\circ \wh \delta \ \ \ \text{ for }  \al :X(\wh \bT_0)\lra X(\bT_0) \und \beta^* :Y(\wh \bS_0)\lra Y(\bS_0)
        \] the natural morphisms associated with $\bT_0\to \wh\bT_0=\bT_0/Z_0$ and  $\wh \bS_0 \lra \bS_0$. 

		\item Let $\bT\leq\bG$ and $\bT^*\leq\bH$ be two $F$-stable maximal tori  and set $\wh\bT=\bT/Z_0$, ${\wh \bT^*:=\wh\pi\inv(\bT^*)}$. 
		Let $s\in (\bT^*)^ F $ and $\theta\in\Irr(\bT^F\mid 1_{Z_0})$ such that $(\bT,\theta)\stackrel\bG {\longleftrightarrow} (\bT^*,s)$ in the sense of (\ref{Ttheta}). Then there exist
	 $\wh s\in (\wh \bT^*)^{F}$ and $\wh \theta\in \Irr(\wh \bT^F)$ such that $(\wh \bT,\wh \theta)\stackrel{\wh\bG} {\longleftrightarrow} (\wh \bT^*,\wh s)$, $\wh \pi(\wh s)=s$ and $\wh \theta(tZ_0)=\theta(t)$ for every $t\in \bT^F$.
	\end{thmlist}
\end{prop}
\begin{proof} For (a), see Bonnaf\'e's theorem recalled in (\ref{rem_ntt}) and its proof in \cite[Sect.~2]{NTT}.  For (b) the usual proof of $\bG/Z_0\cong \bH_0/Z^\diamond\cong \SO_{2l}(\FF)$ via root data gives our claim noting also that $ \beta^* :Y(\wh \bS_0)\lra Y(\bS_0)$ can be defined functorially as $ Y(\restr \wh\pi|{\wh \bS_0})$.

For (c) assume that $\theta$ is obtained from $s$ via the standard construction recalled in \cite[Sect.~1.G]{Cedric}. In particular there exists a character $y\in X(\bT)$ such that $\theta=\kappa \circ \restr y|{\bT^F}$ via a given fixed group injection $\kappa: \FF^\times\hookrightarrow \ov \QQ_\ell^\times$. Without loss of generality we can assume that $\wh \theta= \kappa\circ \restr\wh y|{\wh \bT^F} $ where $\wh y\in X(\wh \bT)$ with $\al(\wh y)=y$. Now $\delta(y)= \beta^*(\wh \delta(\wh y))$. 
	By the construction of $s$ we see that the element $\wh s$ associated to $\wh \theta$ and constructed using $y$ satisfies $\wh \pi (\wh s)=s$. This then proves part (c).
\end{proof}

\begin{lem}\label{lem5_25}
	Let $\si '\in E(\bG)$, $\wt\chi\in\cE (\wt\bG^F, (\wt s))$, $z\in \Z(\wt\bH)$, $\chi\in \Irr(\restr{\wt\chi}|{\GF})$, $\si\in\HF(\si')^*$ as in \Cref{prop821}. Let
	$\Irr(\restr{\wt\chi}|{\Z(\wGF)})=\{\nu \}$ and let $\mu\in\Lin(\wGF)$ be associated to $z$ by duality. Then $\wt \chi^ {\si ' }=\wt \chi\mu$ and the following hold.
	
	\begin{thmlist}
		
		\item  There exists some $\si''\in\Cent_{\wt\bH}(\wt s)^F\si $, a $\spa{\si'',F}$-stable maximal torus $\wt\bT^*$ of $\Cent_\wbH(\wt s)$, some $g\in\wGF$ and some $(\wbT,\wt\theta)$ with $(\wbT,\wt\theta)\stackrel{\wbG}{\longleftrightarrow}(\wt\bT^*,\wt s)$ in the sense of (\ref{Ttheta}), $(\wt\bT,\wt\theta)^{g\si'}=(\wt\bT,\restr \mu|{\wbT^F}\wt\theta)$ and $\restr \wt\theta|{\Z(\wGF)}=\nu$.

	\item	Defining
		$\bT:=\wbT\cap \bG$, no extension of $\theta:=\restr \wt \theta|{\bT^F}$ to $\wh T=\calL\inv({Z_0})\cap\bT$ is $g\si'$-invariant.
	\end{thmlist}
\end{lem}
\begin{proof} If $(\wt s,\phi)\in\wt\JJ(\wt \chi)$, then $(\wt sz,\phi)\in\wt\JJ(\mu\wt \chi)$, see for instance \cite[Thm.~4.7.1(3)]{GM}. On the other hand by the equivariance of $\wt\JJ$ from \Cref{JorTilde}, we have $(\wt sz,\phi)=\si.(\wt s,\phi)\in\wt\JJ(\wt \chi^{\si '})$. This implies $\mu\wt\chi =\wt\chi^{\si '}$. 
	
	For (a) let $\wt \bT^*$ be a maximally split torus of $\Cent_{\wt\bH}(\wt s)$. As this group is reductive and $\spann<F, \si>$-stable, we see that $\si(\wt\bT^*)$ is also a maximally split torus of $\Cent_{\wt\bH}(\wt s)$. Accordingly $\wt\bT^*$ and $\si(\wt\bT^*)$ are $\Cent_{\wt\bH}(\wt s)^F$-conjugate. Hence $\wt\bT^*$ is $\spa{F,\si''}$-stable for some $\si''\in\Cent_{\wt\bH}(\wt s)^F\si $. Let $(\wt \bT, \wt \theta) $ be such that 
	$$(\wbT, \wt \theta)	\stackrel{\wbG}{\longleftrightarrow} (\wbT^*,\wt s)$$
	in the sense of  (\ref{Ttheta}). Using \Cref{prop5_24}(a) we deduce $(\wbT, \wt \theta)^{\si'}	\stackrel{\wbG}{\longleftrightarrow} {}^{(\si')^*}(\wbT^*,\wt s)$ with the latter also a $\wt\bH^F$-conjugate of ${}^{\si}(\wbT^*,\wt s)=({}^{\si}\wbT^*,\wt sz)$, which in turn is $\Cent_{\wt\bH}(\wt s)^F$-conjugate to $(\wbT^*,\wt sz)$ by what has been said before. We get $(\wbT, \wt \theta)^{\si'}	\stackrel{\wbG}{\longleftrightarrow} ({}^{ }\wbT^*,\wt sz)$.  We can multiply the left side by $\restr{\mu\inv}|{\wbT^F}$ while multiplying the right side by $z\inv$ (see proof of \cite[Prop.~2.5.21]{GM}), so that $(\wbT^{\si'}, \wt \theta^{\si'}\restr{\mu\inv}|{\wbT^F})	\stackrel{\wbG}{\longleftrightarrow} ({}^{ }\wbT^*,\wt s).$ Combining the latter with $(\wt\bT,\wt\theta)	\stackrel{\wbG}{\longleftrightarrow} ({}^{ }\wbT^*,\wt s)$, we get that $(\wbT^{\si'}, \wt \theta^{\si'}\restr{\mu\inv}|{\wbT^F})$ and $(\wt\bT,\wt\theta)$ are $\wGF$-conjugate, hence our claim that for some $g\in \wGF$, we have $(\wt\bT,\wt\theta)^{g\si'}=(\wt\bT,\restr \mu|{\wbT^F}\wt\theta)$.

It is well-known that for every $\chi\in \cE(\wGF,[\wt s])$, $\restr \chi|{\Z(\wGF)}$  is a multiple of $\restr \wt \theta|{\Z(\wGF)}$, see for instance \cite[Lem. 2.2]{MaH0}. 

For the proof of (b) recall that the relation $(\wt \bT,\wt \theta) \stackrel{\wt\bG}{\longleftrightarrow} (\wt \bT^*,\wt s)$ implies
$$(\bT,\theta)\stackrel{\bG}{\longleftrightarrow} (\wt\pi(\wt\bT^*),\wt\pi(\wt s)),$$ see \cite[Lem.~9.3]{Cedric}.
Define $\bT^*:=\wt\pi(\wt\bT^*)$, $\wh\bT=\bT/Z_0$, and $\wh\bT^*=\bT^*/Z^\diamond$. 
Let $\ov \theta\in \Irr(\bT^F/Z_0)$ be the character determined by $\theta$. According to \Cref{prop5_24}(c) there exists some extension $\wh \theta$ of $\ov \theta$ to $\wh \bT^F$ and $\wh s=s_0z Z^\diamond $ for some $z\in \Z(\bH_0)$ with 
$$ (\wh \bT,\wh \theta) \stackrel{\wh \bG}{\longleftrightarrow} (\wh \bT^*,\wh s).$$
Letting $\la$ be the linear character of ${\wh \bG}^F$ associated with $z Z^\diamond$, one also gets $$ (\wh \bT,\wh \theta') \stackrel{\wh \bG}{\longleftrightarrow} (\wh \bT^*,\wh s')$$ for $\wh \theta ':=\wh \theta\restr{\la\inv}|{\wh\bT^F}^{}$ and $\wh s':=\wh sz^{-1}=s_0Z^\diamond$, by using the same argument as in the above proof of (a).

This means that duality provides a natural isomorphism $X(\wh\bT^F)\xrightarrow{\sim}(\wh\bT^* )^F$ sending $\wh\theta'$ (seen as having values in $\FF^\times$, a subgroup of $\CC^\times$) to $\wh s'$. This map sends $ \wh \theta'{}^{g\si '} $ to  ${}^{\si ''}\wh s'{} $ by the equivariance recalled in (1.9) which applies to $g\sigma$ whose dual map differs from $\sigma^*$ by an element of $\bG^*{}^F$. So we have $\wh \theta'{}^{g\si '}=\wh \theta'{}$ if and only if ${}^{\si ''}\wh s'{} = {}^{}\wh s'$.
By the assumption (iii) of \Cref{prop821} on $\wt s$ and $s_0$ we have  $[\si,s_0]\in \Z(\bH_0)\setminus \langle h_0^{(\bH_0)}\rangle $. Then also $[\si '',s_0]=[\si,s_0]\in \Z(\bH_0)\setminus \langle h_0^{(\bH_0)}\rangle $ since $\sigma$ and $\sigma''$ differ by an element of $\bH^F$. Therefore $\si''(\wh s') \neq \wh s'$ and $\wh \theta'$ is not $g\si'$-invariant, while $\ov \theta$ is $g\si'$-invariant. This shows that no extension $\wh\theta$ of $\ov\theta$ is $g\si'$-invariant since there are only two extensions of $\ov \theta$. 
\end{proof}

We can now prove \Cref{prop821}. 
\begin{proof}[Proof of \Cref{prop821}:]
Let $\wh Z=\calL\inv({Z_0})\cap\Z(\wbG)$, $\wh G=\calL\inv( {Z_0})\cap\bG$ be defined as in \Cref{lem520}, let 
$\nu\in\Irr(\restr \chi| { \Z(\wGF)})$, $\mu\in\Lin(\wGF)$ with $\wt \chi^{\si'}=\wt \chi \mu$ and $(\wt\bT ,\wt\theta)$ be as in \Cref{lem5_25}. Part (a) of the latter tells us that we also have $\wt \theta^{g\si'}=\wt\theta\restr \mu|{\wbT^F}$ and $\nu \in \Irr(\restr\wt \theta|{\Z(\GF)})$. According to \Cref{lem5_25}(b) the character $\theta$ has no $g\si '$-invariant extension to $\wh T$ and hence by \Cref{lem520} with $\bK =\bT$ and $\tau =g\si'$, $\mu$ and $\nu$ do not satisfy the equivalent assumption (ii) of that lemma. In other words there exists no $\wh \nu\in\Irr(\wh Z\mid \nu)$ with $\mu(tx)=\wh \nu^ {\si'}(x) \wh \nu(x\inv) $ for every $t\in \wh T$ and $x\in \wh Z$ with $tx\in\wbT^F$ ($g\si'$ can be replaced with $\si'$ when conjugating elements of the center). 
Then every such $\wh \nu\in\Irr(\wh Z\mid \nu)$ satisfies $\mu(tx)\neq \wh \nu^ {\si'}(x) \wh \nu(x\inv) $ for at least some $t\in \wh T$ and $x\in \wh Z$ with $tx\in\wbG^F_\chi$, noting that $\wh G$ stabilizes $\chi$ by the assumption \ref{prop821}(ii) and $\wh\bG^F=\wh G/{Z_0}$. Applying now the other direction of  \Cref{lem520} with $\bK =\bG$, this implies that $\chi$ has no ${\si'}$-invariant extension to $\wh \bG^F$. This is our claim taking into account $\wh\bG^F=\wh G/{Z_0}$ again. 
\end{proof}
The above considerations lead also to the following statement which gives a criterion to ensure the existence of a $\si'$-invariant extension. This won't be used here but is given for completeness.
\begin{cor}\label{cor821}
	Let $\si'\in E(\bG)$, $\chi\in\Irr(\GF)^{\spa{\si'}}$ and some $\wt \chi\in\Irr(\wGF\mid \chi)$ such that 
	\begin{asslist}
		\item  $Z_0\leq\ker(\chi)$;
		\item $\restr \wt\chi|{\GF}\neq \chi$ and $\chi$ is $\wh\bG^F$-invariant;
		\item there exist $(\wt s,\phi)\in \wt \JJ(\wt \chi)$, $\si\in\HF(\si')^*$ with $\si.(\wt s, \phi)=( \wt s z, \phi)$ for some $z\in \Z(\wbH)$, and some $s_0\in \bH_0\cap \wt s\Z(\wbH)$ satisfies $[s_0,\si]\in \spa{ h_0^{(\bH_0)}}$. 
	\end{asslist}
	Then every $\wh \chi\in{\Irr}(\wh\bG^F)$ with $\wh \chi(xZ_0)=\chi(x)$ for all $x\in \GF$ is $\si'$-invariant.
\end{cor}
\begin{proof}
	This follows from the same arguments given in the proof of \Cref{prop821}. The assumption $[s_0,\si]\in \spa{ h_0^{(\bH_0)}}$ implies  that in the situation of \Cref{lem5_25}, $\theta$ has a $\si''$-invariant extension to $\wh \bT^F$. 
\end{proof}
   
\subsection{Extending characters of $\wt \EE(\cC)$} We go back to $D=\spann<\restr F_0|{\GF}, \gamma>$ 
satisfying  \Cref{hyp76}, $\cC\in\Cl_{\textrm{ss}}(\bH)^{\spa{F_0,\gamma}}$ and $\wt \EE(\cC)\subseteq \cE(\GF,\cC)^D$ from \Cref{prop818}.

With the help of \Cref{prop5_13}, \Cref{prop522} and \Cref{prop821} above, we are finally in the position to see that half the characters in $\wt \EE(\cC)$ do not extend to  $\GF D$. 
\begin{prop}\label{cor822}
The set $\wt \EE(\cC)$ is a union of  $\wGF$-orbits such that each $\wGF$-orbit consists of two $D$-invariant characters, where one character extends to $\GF D$ and the other one does not extend to $\GF D$. 
Consequently there is some $\wGF$-transversal $\EEnull(\calC)$ in $\wt\EE(\calC)$ such that no $\chi\in\EEnull(\calC)$ extends to $\GF D$, and moreover $\wt\EE(\calC)\setminus \EEnull(\calC)$ is also a $\wGF$-transversal such that  every $\chi\in\wt\EE(\calC)\setminus \EEnull(\calC)$ extends to $\GF D$.
\end{prop}
\begin{proof}
Let $\chi\in \wt\EE(\calC)$ as above. The character $\chi$ satisfies the assumptions of \Cref{prop821} for $\si '=F_0$ according to \Cref{lemfrakC1} and \Cref{prop5_13} when $\cC\in \frakC_1$ and \Cref{prop522} when $\cC\in \frakC_0$. 
Consequently $\chi$, seen as a character $\chi '$ of $\GF/\spannh$, has no $F_0$-invariant extension to $\wh G/\spannh$. 

As explained before \Cref{prop821} the groups and elements for \Cref{prop820} can be chosen such that the group-theoretic assumptions from \Cref{prop820} are satisfied with $X:=\GF/\spann<h_0>$, $\wh X:=\wh G\rtimes D$ with $d\in \wh G\setminus (\GF/\spann<h_0>)$, $f:=\restr F_0|{\wh G}$ and $c:=\gamma$.
Recall $\chi'$ the character of $\GF/\spannh$ that inflates to $\chi$. We observe that $\chi'$ is $\wh X$-invariant, see \Cref{lem817} and \Cref{prop522}, respectively. On the other hand no extension of $\chi'$ to $\wh G$ is $F_0$-invariant by \Cref{prop821}. 

 \Cref{prop820} implies that $\chi'$ extends either to $(\GF/\spannh ) \rtimes D$ or to $(\GF/\spannh)\rtimes \spann<f,dc>$. Accordingly $\chi$ extends to $\GF\rtimes D$ or to $\GF \spann<\gamma,\restr F_0|{\GF} d>$.
Let $h'\in \Z(\bG)\setminus \spannh$ and $d'\in  \bG$ with $ F(d')= d' h'$, or in other words $d'$ induces the diagonal outer automorphism of $\GF$ associated to $h'\in \Z(\bG)=\Z(\bG)_F$ by (\ref{cZF}). Then $[d',F_0]=1$ and $[d',\gamma]= d$ in $\Out(\GF)$ since $[F_0,\Z(\bG)]=1$ and $\gamma(h')=h_0h'$. We then get
    \[(\GF \rtimes D)^{d'}=\GF \spann<f,dc>.\]
If $\chi$ has no extension to $\GF \rtimes D$ and extends to $\GF \spann<f,dc>$, then $\chi^{d'}$ has no extension to $(\GF\rtimes D)^{d'}= \GF \spann<f,dc>$ and extends to $(\GF \spann<f,dc>)^{d'}=\GF\rtimes D$. 
 We can argue analogously in the case where $\chi$ extends to $\GF\rtimes D$. 

This proves that every $\wGF$-orbit in $\wt \EE(\cC)$ consists of one character that extends to $\GF D$ and another one that doesn't. Taking all characters in $\wt\EE(\cC)$ without an extension $\GF D$ defines a $\wGF$-transversal as required. 
	\end{proof}
Next we describe all characters of $\GF$ that extend to $\GF D$. 
\begin{cor}\label{cor_EE} 
Let $\frakC=\Cl_{\textrm{ss}}(\bH)^{\spann<  F_0{},\gamma>}$ be as in \Cref{frakC}, $\EE(\cC)$ as in \Cref{cor822} for every $\cC\in \frakC$ and let $\EEnull :=\bigcup _{\calC\in\frakC} \EEnull(\calC)$.
Then every $\chi\in\Irr(\GF)^D\setminus \EEnull$ extends to $\GF D$. 
\end{cor}
\begin{proof} For every $\cC\in \frakC$ the set  $\EE(\cC)$ contains no character that extends to $\GF D$ by \Cref{cor822} and 
$$ |\EEnull(\calC)| =\frac 1 2 
|\cE(\GF,\calC)^D|-\frac 1 2  |\cE(\GFnull,\calC)^ {\spann<\gamma>}|,$$
see \Cref{prop818} and \Cref{lemfrakC1}, respectively for $\calC\in \frakC_0$ or  $\calC\in \frakC_1$. The assumptions of \Cref{cor79} are ensured and we get our claim  by applying it. 
\end{proof}

The sets $\EE(\cC)$, $\wt\EE(\cC)$ and $\ov\EE(\cC)$ were defined using $s$ and $\wc A(s)$. They then seem to be not uniquely determined by $\cC$.
But by the above $\EE(\calC)$ is the set of characters in $\cE(\GF,\cC)^D$, that have no extension to $\GF D$. Hence this set only depends on $\cC$ (and of course $D$). This proves that the sets $\ov\EE(\calC)=\Pi_\wGF(\EE(\cC))$ introduced in \Cref{cor816} and $\wt \EE(\calC)$ deduced from it are actually uniquely determined by $\cC$, as well. 
 
Taking together the results from above we give a criterion to show that  $\EE(\cC)$ is empty for some conjugacy class $\cC$. This won't be used here but is stated for future reference.
\begin{cor}
Let $\cC\in \Cl_{\textrm{ss}}(\bH)$ be $D$-stable. 
\begin{thmlist}
    \item If $|\cE(\GF,\cC)^D|=|\cE(\GFnull,\cC)^{\spa \gamma}|$, then all $D$-invariant characters of $\cE(\GF,\cC)^D$ extend to $\GF \rtimes D$.
    \item If $\cC$ contains exactly two $\HF$-conjugacy classes and at least one $\gamma$-stable $\HFnull$-conjugacy class, then every character in $\cE(\GF,\cC)^D$ extends to $\GF \rtimes D$.
\end{thmlist}
\end{cor}
\begin{proof}
According to \Cref{cor_EE} all characters in $\cE(\GF,\cC)\setminus\EE(\cC)$ extend to $\GF D$ while the characters in $\EE(\cC)$ have no extension to $\GF D$.
Hence in both parts it is sufficient to prove that $\EE(\cC)=\emptyset$.

By construction the set $\EE(\cC)$ defined above satisfies the equation \[|\EE(\cC)|= \frac 1 2 |\cE(\GF,\cC)^D| - \frac 1 2 |\cE(\GFnull,\cC)^{\spa \gamma}|,\]
see \Cref{prop818} and \Cref{lemfrakC1}. Hence the assumed equation $|\cE(\GF,\cC)^D|=|\cE(\GFnull,\cC)^{\spa \gamma}|$ implies $|\EE(\cC)|=0$ in part (a). The statement follows. 

In the situation of (b) we see that  $\cC\in\frakC_0$, see \Cref{frakC}. On the other hand the set $\EE(\cC)$ is empty unless $|\AHs|=4$ for some $s\in \cC$, see \Cref{prop818}.
\end{proof}

\begin{rem} There exists classes $\cC$ such that $\EE(\cC)\neq \emptyset$ for infinitely many $l\geq 4$. For instance in \cite[Rem. 3.6]{CS22} one can find a description of an element $s$ of $\frakC_1$ (see \Cref{frakC}) in rank 4. The construction can be transferred to higher ranks. Taking any unipotent character of $\Cent^\circ_\bH(s)^F$ and using Jordan decomposition leads to $\EE(\cC)\neq \emptyset$

Another challenge is to find a conjugacy class $\cC\in \frakC_0$ such that $\EE(\cC)\neq \emptyset$. Following the considerations in \Cref{sec5C} it is sufficient to find a character in the set $\UU_{exc}^{[\gamma']}$ from \Cref{cor816} associated to some $s\in \cC$. 

Let $q_0$ be a prime power with $4\mid (q_0-1)$ and let $q=q_0^2$. Let $\varpi\in \FF^\times$ with $\varpi^2=-1$ and let $\xi\in\FFtimes$ such that $\xi^{q}= \xi$ and $\xi^{q_0}\in \spa{\varpi} \xi$. We let $a$ be an odd integer and $d\geq 4$. Let $l=2d+4a+4$ and let $\bG=\tD_{l,\text{sc}}(\FF)$, $\bH=\tD_{l,\text{ad}}(\FF)$ with $F=F_q$. Let $s$ be the image in $\bH$ of the element $\prod_{i=1}^l \h_{e_i}(\zeta_i)\in \bG$ in the parametrization of \cite[Not. 3.3]{TypeD1}. We assume the coefficients $\zeta_i$ belong to $\{1,\varpi, \xi, \varpi \xi, \xi^{q_0}, \varpi\xi^{q_0}\}$ and the multiplicities are chosen so that $\Cent^\circ_{\bH}(s)$ is of type $\tD_{d}\times \tD_{d}\times \tA_a\times \tA_a \times \tA_a\times \tA_a$. Then $\Cent^\circ_{\bH}(s)^F$ has a commutator subgroup isomorphic to some central product $\tD_{d,\text{sc}}(q).\tD_{d,\text{sc}}(q). \SL_{a+1}(q). \SL_{a+1}(q).  \SL_{a+1}(q). \SL_{a+1}(q)  $. The group $\AHs$ contains an element that acts on $\Cent_{\bH}^\circ(s)$ by simultaneous graph automorphisms on the two type $\tD$ factors and with trivial action on the type $\tA$ factors. Another element $b\in \AHs$ swaps the two type $\tD$ factors and permutes the four type $\tA$ factors. Let $\phi_\tD\in\UCh(\tD_{d}(q))$ be a unipotent character stable under the graph automorphism of order 2 (i.e. associated with a non-degenerate symbol) and let $\phi_\tA$, $\phi_\tA'\in\UCh(\SL_{a+1}(q))$ be two distinct unipotent characters. Then $\phi_\tD.\phi_\tD.\phi_\tA.\phi'_\tA. \phi'_\tA.\phi_\tA$ defines a unipotent character $\phi$ of $\Cent_{\bH}^\circ(s)^F$, such that $\phi_\tA$ is a unipotent character of the factors corresponding to the eigenvalues $\xi$ and $\varpi\xi^{q_0}$. Then $\phi$ satisfies $|\AHs^F_\phi|=2$. In the terminology of \Cref{cor816}, $\phi$  is $\gamma$-stable and belongs to the set $\UU_{exc}^{[\gamma]}$ defined using $F_{q_0}$ composed with an inner automorphism of $\bH$ as $F_0$. Denoting by $\cC$ the $\bH^F$-conjugacy class containing $s$ and using $\Gamma_{s,F,F_0}$ from \Cref{labelEGFFnull}, one gets a character $\Gamma_{s,F,F_0}(\phi)\in \ov\cE(\GF,\cC)$ with one constituent in $\EE(\cC)$.
\end{rem}
\subsection{Proofs of Theorems A, B and C}\label{ssec5D}
We can now finish the proof of  \Cref{Ainfty_D} that is, verify Condition $\Ai$ from \ref{Ainfty} for the group $\GF$ assuming \Cref{hyp_cuspD_ext} in lower ranks.
\begin{proof}[Proof of \Cref{Ainfty_D}] In the cases of $G={}^2\tD_{l,\text{sc}}(q)$ or $\tD_{l,\text{sc}}(q)$ for an even $q$, the group $\bG =\tD_{l,\text{sc}}(\ov\FF_2)$ has connected (trivial) center, so condition $\Ap$ is empty. The extendibility part of $\Ai$ along with \Cref{hyp_cuspD_ext} are then implied by \Cref{prop_ext_wG}. 

We now assume $q$ to be odd and we  wish to apply \Cref{propE4}. By assumption \Cref{hyp_cuspD_ext} holds for cuspidal characters of groups of lower ranks and hence by \Cref{thm_typeD1} there exists some $E(\GF)$-stable $\wGF$-transversal in $\Irr(\GF)$. This is Assumption (i) of \Cref{propE4}.

Assumption (ii) of \Cref{propE4} requires that, for every non-cyclic $2$-subgroup $D = \spann<\restr F_0|{\GF},\gamma>\leq \uE(\GF)$, each $D$-stable $\wGF$-orbit sum $\ov \chi$ contains some $D$-invariant character $\chi\in\Irr(\ov \chi)^D$ that extends to $\GF D$. Additionally we have to show that every $\chi\in\Irr(\ov \chi)^D$ extends to $\GF D$, unless $|\Irr(\ov \chi)|= 2$.  By \Cref{rem89} we may assume \Cref{hyp76} and apply the results obtained in the present chapter.

For each $\ov \chi\in \big(\Pi_\wGF(\Irr(\GF))\big)^D$ 
there exists some character $\chi\in \Irr(\ov \chi)^D$, as Assumption (i) of \Cref{propE4} holds in the form (transversality of actions of $\wGF$ and $E(\GF)$) given by \cite[Lem.~2.4]{TypeD1} applied with $\wt X=\wbG^F$, $\wh Y= \GF E(\GF) $. Let $\EE=\bigcup _{\calC\in\frakC} \EEnull(\calC)$ be the set from \Cref{cor_EE}, so that $\chi$ extends to $\GF D$ unless $\chi\in \EE$. 

If $\chi\in \EE$, then  \Cref{cor822} tells us that $\Irr(\ov \chi)=\{\chi ,\chi '\}$ where one of those two characters extends to $\GF D$ and the other one doesn't. In both cases this verifies Assumption (ii) of \Cref{propE4}. This also shows that if $\ov \chi\in \big(\Pi_\wGF(\Irr(\GF))\big)^D$ and $|\Irr(\ov\chi)|\neq 2$ then all elements of $\Irr(\ov\chi)^D$ extend to $\GF D$. This gives all our claims.
\end{proof}
We conclude with the proof of the three main statements announced in the Introduction. 
\begin{proof}[Proof of \Cref{thm1}] Let us first review types different from $\tD$ and $^2\tD$. For types $\tA$ and $^2\tA$, this is \cite[Thm 4.1]{CS17A}, for type $\tC$ this is \cite[Thm 3.1]{CS17C} and for types $\tB$ and $\tE$ this is \cite[Thm B]{CS18B}. Other types have $\wG$ acting by inner automorphisms on $G$ since $\Z(G)=1$, and in those types $E(G)$ is cyclic, so that $\Ai$ is immediate.
	
The remaining types are $\tD_l$ and $^2\tD_l$ for $l\geq 4$. In the cases of $G={}^2\tD_{l,\text{sc}}(q)$ or $\tD_{l,\text{sc}}(q)$ for an even $q$, we have seen in the above proof of \Cref{Ainfty_D} that both $\Ai$ and \Cref{hyp_cuspD_ext} are implied by \Cref{prop_ext_wG}. 

We now treat the cases where $q$ is odd. The case of $G=\tD_{4,\text{sc}}(q)$ is implied by \Cref{Ainfty_D} since in that case \Cref{hyp_cuspD_ext} for smaller ranks is empty. The statement for $G=\tD_{l,\text{sc}}(q)$ with $l>4$ is shown by an induction on $l$. We can then assume \Cref{hyp_cuspD_ext} for all $l'<l$. Then \Cref{Ainfty_D} gives Condition $\Ai$ for $G=\tD_{l,\text{sc}}(q)$.

When $G= {}^2\tD_{l,\text{sc}}(q)$, \Cref{thm41} gives our claim since \Cref{hyp_cuspD_ext} in all smaller ranks $l'$ is ensured by what we have just proved about untwisted type $\tD$. 
Note that the above reasoning uses the untwisted case to solve the twisted case while a statement like \Cref{prop82} also shows dependency in the other direction.
\end{proof}

\begin{proof}[Proof of Theorem B]
 By \Cref{thm1} the assumption $\Ap$ of \Cref{thm66} is ensured. So we obtain the $\calZ_F\rtimes E(\GF)$-equivariant Jordan decomposition  in the form of \Cref{thm66}. The form given in Theorem B then comes by defining $\chi_{s,\phi} := \JJ\inv(\ov{(s,\phi)})$.\end{proof}

\begin{proof}[Proof of \Cref{thmC}]
	By the reduction Theorem B of \cite{IMN} it is sufficient to prove that the inductive McKay conditions from \cite[Sect.~10]{IMN} hold for every non-abelian simple group and the prime $3$. This is known for simple groups not of Lie type thanks to \cite{ManonLie}. For simple groups of Lie type,
 Theorem 2.4 of \cite{CS18B} tells us that we just have to check the conditions $\Ai$, $A(1)$, $B(1)$, $A(2)$ and $B(2)$  for any $\bG^F$ where $\bG=\bG_{\mathrm{sc}}$ is simple simply connected and $F\colon\bG\to\bG$ is a Frobenius endomorphism. Now $\Ai$ is ensured by Theorem A while $A(1)$, $B(1)$, $A(2)$ and $B(2)$ hold in all types by \cite[Thm~3.1]{MS16}.
\end{proof}

\printindex

\end{document}